\documentclass[11pt]{article}

\usepackage[utf8]{inputenc}
\usepackage{graphicx} 
\usepackage[top=1in, bottom=1in, left=1in, right=1in]{geometry} 
\usepackage{xcolor}
\usepackage{tikz}
\usepackage{booktabs}
\usepackage[font=small]{caption}
\usepackage[font=small]{subcaption}
\usepackage{multirow}
\usepackage{float}
\usepackage{enumitem}

\usepackage{charter}       

\usepackage{amssymb}
\usepackage{mathtools}
\usepackage{array}

\usepackage{algorithm}
\usepackage{algorithmic}

\usepackage{amsthm}
\newtheorem{theorem}{Theorem}
\newtheorem{lemma}{Lemma}

\newtheorem{corollaryT}[theorem]{Corollary}  
\newtheorem{corollaryL}[lemma]{Corollary}  
\newtheorem{proposition}[lemma]{Proposition}

\newtheorem*{lemmaX}{Lemma} 

\usepackage{hyperref}
\hypersetup{
    colorlinks=true,
    citecolor=blue,
    urlcolor=blue,
    linkcolor=blue,
}
\usepackage{cleveref}
\usepackage{natbib}

\newcommand{\cI}{\mathcal{I}}

\newcommand{\cO}{\mathcal{O}}

\newcommand{\cX}{\mathcal{X}}

\newcommand{\bbR}{\mathbb{R}}

\newcommand{\ourspace}{\bbR^d}
 
\newcommand{\norm}[1]{\left\| #1 \right\|} 
\newcommand{\inner}[2]{ \left\langle #1 ,  #2 \right\rangle }

\newcommand{\set}[1]{\left\{  #1  \right\}}

\newcommand{\pr}[1]{ \left( #1 \right) }

\newcommand{\vfrac}[2]{#1/#2}
\newcommand{\hfrac}[2]{#1/#2}

\newcommand{\ie}{i.e.}

\newcommand{\ineqstar}{\cI_2}
\newcommand{\ineqsonsecutive}{\cI_1}
\newcommand{\gdtype}{GD-type}
\newcommand{\dseq}[1]{$\textstyle #1$}

\newcommand{\ThetaProprtyUsed}[1]{{\color{black} #1}}
\newcommand{\AlphaDefinitionUsed}[1]{{\color{black} #1}}

\newcommand{\adanag}{\hyperref[alg:AdaNAG]{AdaNAG}}
\newcommand{\adanagS}{Algorithm~\ref{alg:AdaNAG-S}}
\newcommand{\adagd}{\hyperref[alg:AdaGD]{AdaGD}}
\newcommand{\adanagG}{\hyperref[alg:AdaNAG-G]{AdaNAG-G}}

\newcommand{\adagdgOne}{\hyperref[sec:trade_off_theory_practical]{AdaGD$^1$}}
\newcommand{\adagdgSqrt}{\hyperref[sec:trade_off_theory_practical]{AdaGD$^{1/2}$}}
\newcommand{\adagdgZero}{\hyperref[sec:trade_off_theory_practical]{AdaGD$^0$}}

\newcommand{\adanagpracticename}{AdaNAG-G$_{12}$}
\newcommand{\adanagsqrtname}{AdaNAG-G$^{1/2}$}
\newcommand{\adanagpractice}{{\hyperlink{adanag_12}{\color{blue} \adanagpracticename}}}
\newcommand{\adanagsqrt}{\hyperlink{adanag_half}{\color{blue} \adanagsqrtname}}

\newcommand{\ssz}{s}
\newcommand{\talpha}{\alpha}
\newcommand{\oldalpha}{\tilde{\alpha}}

\newcommand{\inverLzero}{\frac{1}{L_0}}
\newcommand{\ratio}{r}

\newcommand{\rzeroAdaNAG}{r} 
\newcommand{\rbound}{r}

\newcommand{\upperboundL}{L}
\newcommand{\localupperboundL}{L}
\newcommand{\gdupperboundL}{L}
\newcommand{\upperboundconstant}{R} 

\newcommand{\Rlimit}{\bar{R}}
\newcommand{\rsBA}{\rho_k^{-1} } 
\newcommand{\rsBAreciprocal}{\rho_k } 
\newcommand{\rsBAgdreciprocal}{\rho_k }
\newcommand{\Vdifferencebound}{\bar{Q}}
\newcommand{\CC}{m}

\newcommand{\segmentparameter}{t}
\newcommand{\epsilonbound}{\bar{\epsilon}}
\newcommand{\NN}{N}
\newcommand{\MM}{K}
\newcommand{\mm}{k}

\begin{document}

\title{An Adaptive and Parameter-Free Nesterov's Accelerated Gradient Method for Convex Optimization} 
\author{Jaewook J. Suh\footnotemark[1] \and Shiqian Ma\footnotemark[1]}
\renewcommand{\thefootnote}{\fnsymbol{footnote}}
\footnotetext[1]{Department of Computational Applied Mathematics and Operations Research, Rice University, Houston, TX, USA.}
\date{ }

\maketitle

\begin{abstract}
   We propose AdaNAG, an adaptive accelerated gradient method based on Nesterov's accelerated gradient method. 
    AdaNAG is line-search-free,  parameter-free, and achieves the accelerated convergence rates 
    $f(x_k) - f_\star = \mathcal{O}\pr{\hfrac{1}{k^2}}$ and 
    $\min_{i\in\set{1,\dots, k}} \norm{\nabla f(x_i)}^2 = \mathcal{O}\pr{\hfrac{1}{k^3}}$ 
    for $L$-smooth convex function $f$. 
    We provide a Lyapunov analysis for the convergence proof of AdaNAG, which additionally enables us to propose a novel adaptive gradient descent (GD) method, AdaGD. 
    AdaGD achieves the non-ergodic convergence rate 
    $f(x_k) - f_\star = \mathcal{O}\pr{\hfrac{1}{k}}$, like the original GD. 
    The analysis of AdaGD also motivated us to propose a generalized AdaNAG that includes practically useful variants of AdaNAG. Numerical results demonstrate that our methods outperform some other recent adaptive methods for representative applications. 
\end{abstract}

\section{Introduction}

With the growing prevalence of large-scale optimization problems, there has been increasing interest in adaptive optimization methods that are line-search-free and parameter-free.  
Initiated by~\cite{MalitskyMishchenko2020_adaptive} and further developed in~\cite{LatafatThemelisStellaPatrinos2024_adaptive, LatafatThemelisPatrinos2024_convergence, MalitskyMishchenko2024_adaptive, ZhouMaYang2024_adabb}, the study of adaptive methods for convex optimization 
has been actively explored in recent years. 
However, research on adaptive accelerated gradient methods remains limited, and the first such algorithm was recently proposed  in~\cite{LiLan2024_simple}.

In this paper, we propose a novel adaptive Nesterov's accelerated gradient method (\adanag) for solving the minimization problem
\begin{equation}\label{original-problem}
    \min_{x \in \mathbb{R}^d} f(x),
\end{equation}
where \( f \colon \mathbb{R}^d \to \mathbb{R} \) is a convex and (locally) smooth function. We use $x_\star$ to denote a minimizer of \eqref{original-problem} and $f_\star$ to denote the optimal value. Our method achieves the accelerated convergence rates 
\( f(x_k) - f_\star = \mathcal{O}\left(\hfrac{1}{k^2}\right) \) and 
\( \min_{i \in \{1, \dots, k\}} \|\nabla f(x_i)\|^2 = \mathcal{O}\left(\hfrac{1}{k^3}\right) \), 
similar to the original Nesterov's accelerated gradient method (NAG) \cite{Nesterov1983_method, KimFessler2018_generalizing, ShiDuJordanSu2021_understanding}. 
We establish the convergence of the method using a novel Lyapunov analysis, which also provides intuition for developing new adaptive methods. Based on this intuition, we further propose a novel adaptive gradient descent method (\adagd) that achieves the non-ergodic convergence rate 
\( f(x_k) - f_\star = \mathcal{O}\left(\hfrac{1}{k}\right) \). 
In addition, we provide a family of algorithms that generalize AdaNAG, and present instances that are both practically useful and theoretically sound. 
We demonstrate the efficiency of our algorithms through several practical scenarios in which they outperform existing adaptive methods.

\subsection{Prior works}

\paragraph{Nesterov acceleration.}
Nesterov's accelerated gradient (NAG) method was introduced in his seminal work \cite{Nesterov1983_method}. NAG achieves an accelerated convergence rate $f(x_k) - f_\star = \cO\pr{{1}/{k^2}}$ when $f$ is a globally smooth convex function, while using only first-order information. 
This rate has been proven to be optimal for this problem setup up to order, as shown in \cite[Theorem~2.1.7]{Nesterov2004_introductory}. 
As the demand for first-order methods grew due to the need for large-scale optimization problems, NAG has been studied in various aspects \cite{SuBoydCandes2014_differential, SuBoydCandes2016_differential, WibisonoWilsonJordan2016_variational, LeeParkRyu2021_geometric, dAspremontScieurTaylor2021_acceleration, ShiDuJordanSu2021_understanding}, and numerous variants have been proposed in recent decades \cite{BeckTeboulle2009_fast, ChambolleDossal2015_convergence, Nesterov2015_universal,  KimFessler2016_optimized, AttouchChbaniPeypouquetRedont2018_fast, KimFessler2021_optimizing, KimYang2023_unifying}, to name a few. 
Optimized gradient method (OGM) \cite{KimFessler2016_optimized, KimFessler2017_convergence} improved the constant of the bound by $2$, and is proven to be the optimal method for $L$-smooth convex function minimization \cite{Drori2017_exact}. FISTA \cite{BeckTeboulle2009_fast} extended NAG to the proximal gradient (prox-grad) algorithm with the same convergence rate. 
The works in \cite{SuBoydCandes2014_differential, SuBoydCandes2016_differential, ChambolleDossal2015_convergence, AttouchChbaniPeypouquetRedont2018_fast} extended the original parameter choice $\theta_k = \frac{k+2}{2}$ in NAG to $\theta_k = \frac{k+p}{p}$ with $p \ge 2$, and still achieved the accelerated rate $f(x_k) - f_\star = \cO\pr{{1}/{k^2}}$.  
In addition, \cite{ChambolleDossal2015_convergence, AttouchChbaniPeypouquetRedont2018_fast} further established iterate convergence when $p > 2$. 
An improved convergence rate for the gradient norm square $\min_{i\in\set{1,\dots, k}} \norm{\nabla f(x_i)}^2 = \mathcal{O}\pr{{1}/{k^3}}$ was provided in \cite{KimFessler2018_generalizing, ShiDuJordanSu2021_understanding}.

Recently, OptISTA \cite{JangGuptaRyu2024_computerassisted} improved the convergence rate of FISTA by a constant factor of $2$, and was proven to be optimal for the prox-grad optimization setup. Unlike NAG and FISTA, which achieve the same rate, the optimal methods OGM and OptISTA have different rates, with the rate of OGM having a slightly better constant than that of OptISTA. 
This difference suggests that we may achieve improved convergence rates by focusing on single-function minimization when designing an algorithm. 
This motivated us to develop a new accelerated adaptive algorithm refined on the single function minimization problem, taking OGM as the reference.

\paragraph{Adaptive methods for convex function minimization.}
The demand for adaptive and parameter-free methods has grown as data size increases and as large-scale optimization problems become more prevalent. The seminal work \cite{MalitskyMishchenko2020_adaptive} initiated the study of adaptive methods for convex minimization problems, without using additional expensive computations such as line search \cite{Goldstein1962_cauchys, LarryArmijo1966_minimization, Nesterov2015_universal}. 
They introduced the {ad}aptive {g}radient {d}escent (AdGD) method, 
which achieves an ergodic convergence rate of $\cO\pr{{1}/{k}}$ for the function value error. Without requiring knowledge of the smoothness parameter or performing line-search computations, AdGD selects the step size based on an approximation of the local smoothness parameter using the last two iterates. 

Introducing AdaPGM, an extension to the prox-grad setup was considered in \cite{OikonomidisLaudeLatafatThemelisPatrinos2024_adaptive}. A generalized version of AdaPGM, AdaPG$^{q,r}$ was introduced in \cite{LatafatThemelisPatrinos2024_convergence}. 
The analysis of the original AdGD was refined in \cite{MalitskyMishchenko2024_adaptive}, where an advanced method with a larger step size (AdGD2) was proposed, along with an extension to the prox-grad setup based on the refined proof.
 An adaptive method based on the Barzilai-Borwein method \cite{BarzilaiBorwein1988_twopoint}, AdaBB, was introduced in \cite{ZhouMaYang2024_adabb}. 
However, none of these adaptive gradient descent-based methods achieve a non-ergodic convergence rate, which was identified as a challenging problem in \cite{MalitskyMishchenko2024_adaptive}.

Recently, an adaptive method that achieves the accelerated convergence rate $f(x_k) - f_\star = \cO\pr{{1}/{k^2}}$, called the {a}uto-{c}onditioned {f}ast {g}radient {m}ethod (AC-FGM), was introduced by \cite{LiLan2024_simple}. 
Focusing on the prox-grad setup, they introduced an additional auxiliary sequence called “prox-centers,” which does not appear in the original Nesterov’s acceleration. 
Furthermore, they extended their analysis to convex problems with H\"{o}lder continuous gradients. 
We focus on the single-function convex minimization problem and propose a novel adaptive accelerated method that does not require such an auxiliary sequence.

\paragraph{Lyapunov analysis.} 
A refined proof technique plays an important role in analyzing adaptive methods, particularly in deriving new step sizes and designing new algorithms as observed in \cite{MalitskyMishchenko2024_adaptive, LatafatThemelisPatrinos2024_convergence}. Among such techniques, Lyapunov analysis is a powerful framework that leverages a nonincreasing potential function, and it has been widely used to analyze a variety of optimization methods, including accelerated methods~\cite{BeckTeboulle2009_fast, BG17, TaylorBach2019_stochastic, dAspremontScieurTaylor2021_acceleration, LeeParkRyu2021_geometric, ParkParkRyu2023_factorsqrt2, SuhRohRyu2022_continuoustime} and adaptive methods~\cite{MalitskyMishchenko2020_adaptive, LatafatThemelisPatrinos2024_convergence, OikonomidisLaudeLatafatThemelisPatrinos2024_adaptive, MalitskyMishchenko2024_adaptive, ZhouMaYang2024_adabb}. 
Lyapunov analysis is especially crucial for adaptive methods, since such analyses typically use only recent iterates $x_k$ and $x_{k+1}$, which provides strong motivation to approximate the local smoothness parameter using recent iterates. 
In this work, we present a novel Lyapunov analysis for an adaptive accelerated method. Our proof is inspired by the Lyapunov analysis of OGM~\cite{ParkParkRyu2023_factorsqrt2, dAspremontScieurTaylor2021_acceleration}. As a consequence of our analysis, we derive a novel convergence result $\min_{i\in\{1,\dots, k\}} \norm{\nabla f(x_i)}^2 = \mathcal{O}\pr{1/k^3}$, which is, to our knowledge, the first such rate established for adaptive methods. Furthermore, we design a new gradient-descent-type algorithm that achieves a non-ergodic convergence rate $f(x_k) - f_\star = \cO(1/k)$, as an application of our Lyapunov analysis.

\begin{table}[H] 
    \centering
    \begin{tabular}{ c|cc|c|c } 
     \toprule
     {} 
     & $f(x) - f_\star$ 
     & non-ergodic 
     & $\min_{i\in\set{1,\dots,k}}\norm{\nabla f(x_i)}^2$ 
     & without momentem \\
     \midrule\midrule
     {AdGD2 \cite{MalitskyMishchenko2024_adaptive}} 
     & $\mathcal{O}\pr{ \vfrac{1}{k} }$
     & $\times$
     & $\mathcal{O}\pr{ \vfrac{1}{k} } $ 
     & \checkmark \\
     {AdaPGM \cite{OikonomidisLaudeLatafatThemelisPatrinos2024_adaptive, LatafatThemelisPatrinos2024_convergence}} 
     & $\mathcal{O}\pr{ \vfrac{1}{k} }$
     & $\times$
     & $\mathcal{O}\pr{ \vfrac{1}{k} } $ 
     & \checkmark \\
     {AdaBB \cite{ZhouMaYang2024_adabb}} 
     & $\mathcal{O}\pr{ \vfrac{1}{k} }$
     & $\times$
     & $\mathcal{O}\pr{ {1}/{k} } $ 
     & \checkmark \\
     \textbf{AdaGD (ours)} 
     & $\mathcal{O}\pr{ \vfrac{1}{k} }$
     & \checkmark
     & $\mathcal{O}\pr{ {1}/{k^{2}} } $ 
     & \checkmark \\
     \midrule
     {AC-FGM \cite{LiLan2024_simple}} 
     & $\mathcal{O}\pr{ \vfrac{1}{k^2} }$
     & \checkmark 
     & $\mathcal{O}\pr{ {1}/{k^{2}} } $  
     & $\times$ \\
     \textbf{AdaNAG (ours)} 
     & $\mathcal{O}\pr{ \vfrac{1}{k^2} }$
     & \checkmark
     & $\mathcal{O}\pr{ {1}/{k^{3}} } $ 
     & $\times$ \\
     \bottomrule
    \end{tabular}
    \captionsetup{width=\linewidth} 
    \caption{Comparison between our methods and previous adaptive methods. In cases where the convergence rate of $\min_{i\in\set{1,\dots,k}}\norm{\nabla f(x_i)}^2$ is not explicitly mentioned, we apply $\norm{\nabla f(x)}^2 \le 2L(f(x) - f_\star)$, assuming the smoothness parameter $L$. 
    }
\end{table}

\subsection{Contributions}
The main contributions of this work are as follows.

\begin{itemize}[leftmargin=*]
    \item 
    We provide a novel accelerated adaptive algorithm, \adanag, that is line-search-free, parameter-free, and achieves the accelerated convergence rate $f(x_k)-f_\star = \cO\pr{ \hfrac{1}{k^2}}$. 
    Compared to AC-FGM, 
    we have a better coefficient of $\norm{x_0-x_\star}^2$ in the upper bound of the convergence rate of \adanag.
    Moreover, \adanag\ achieves a convergence rate for the minimum of the squared gradient norm, 
    \( \min_{i \in \{1, \dots, k\}} \|\nabla f(x_i)\|^2 = \mathcal{O}\left( \hfrac{1}{k^3} \right) \). 
    To the best of our knowledge, the latter is the first result of its kind for adaptive algorithms.
     
    \item 
    We provide a Lyapunov analysis that allows for a concise understanding of the proof for \adanag. Such understanding motivates us to consider a 
    strategy to develop a novel adaptive method, which leads to a novel adaptive gradient descent algorithm, \adagd, with a non-ergodic \( \mathcal{O}\left( \hfrac{1}{k} \right) \) convergence rate. 
    This is the first non-ergodic convergence result for adaptive methods without momentum. 
    \adagd \ also achieves the convergence rate \( \min_{i \in \{1, \dots, k\}} \|\nabla f(x_i)\|^2 = \mathcal{O}\left( \hfrac{1}{k^2} \right) \). 
    
    \item 
    Finally, we provide a family of algorithms that generalize AdaNAG and enable practically useful parameter selection. 
    We introduce our ultimate algorithms, \adanagpractice\ and \adanagsqrt, which are both theoretically sound and practically useful. 
    We demonstrate the efficiency of our algorithms 
    by showing that they outperform prior algorithms, including AC-FGM, when applied to some representative applications.
\end{itemize}

\paragraph{Organization.} 
The rest of the paper is organized as follows. 
In \Cref{sec:preliminary}, we introduce the preliminary concepts for our analysis. 
In \Cref{sec:adanag}, we introduce our adaptive accelerated algorithm~\adanag\ and present the main convergence results, 
whose proofs, based on a Lyapunov analysis, are provided in \Cref{sec:AdaNAG_anaylsis}. 
In \Cref{sec:adagd}, we introduce a novel family of GD-type adaptive algorithms, \adagd. 
In \Cref{sec:adanag-g}, we introduce a family of adaptive algorithms that generalizes \adanag\ and allows for practically useful parameter selection. 
As concrete instances derived from this family, we present \adanagpractice\ and \adanagsqrt, both of which are theoretically sound and practically useful. 
Numerical results are presented in \Cref{sec:experiments}. 
Finally, we draw some concluding remarks in \Cref{sec:conclusion}.

\section{Preliminaries} \label{sec:preliminary} 

$f$ is called \emph{$L$-smooth} if it is differentiable and $\nabla f$ is $L$-Lipschitz continuous, \ie, 
\begin{equation*}
     \norm{ \nabla f(x) - \nabla f(y) } \le L \norm{ x - y }, \qquad \forall x, y \in \ourspace.
\end{equation*}
When $f$ is an $L$-smooth convex function, it is known \cite[Theorem~2.1.5]{Nesterov2004_introductory} that the following inequality holds for all $x,y \in \ourspace$
\begin{equation}    \label{eq:L_smooth_ineq}
    f(y) - f(x) + \inner{\nabla f(y)}{x-y} + \frac{1}{2L} \norm{ \nabla f(x) - \nabla f(y) }^2 \le 0. 
\end{equation}
We say $f$ is \emph{locally smooth} if $\nabla f$ is locally Lipschitz continuous, \ie, if for every compact set $K\subset \ourspace$ there exists $L_K>0$ such that
\begin{equation*}
    \norm{ \nabla f(x) - \nabla f(y) } \le L_K \norm{ x - y }, \qquad \forall x, y \in K.
\end{equation*}
We refer to $L_K$ as the \emph{(local) smoothness parameter} of $f$ on $K$. 
It is clear that $L$-smoothness implies local smoothness. 
It is natural to ask whether a locally smooth convex function also satisfies an inequality similar to \eqref{eq:L_smooth_ineq}. 
The answer is positive. However, to the best of our knowledge, this is not immediate from known facts for globally smooth functions. Therefore, we state it as a lemma, along with a useful corollary that will be used 
when we approximate the local smoothness parameter 
in our later arguments.  
We provide the proofs in \Cref{appendix:proofs_for_preliminary}.

\begin{lemma} \label{lemma:local_smoothness}
    Suppose $f$ is a locally smooth convex function and $K\subset\bbR^d$ is a compact set. Let \( K \subset \bar{B}_R(c) \), where \( \bar{B}_R(c) = \{ x \mid \| x - c \| \leq R \} \). 
    Let $\bar{L}_K>0$ be a smoothness parameter of $f$ on $\bar{B}_{3R}(c)$. 
    Then the following inequality is tr
    \begin{equation*}
        f(y) - f(x) + \inner{\nabla f(y)}{x-y} + \frac{1}{2\bar{L}_K} \norm{ \nabla f(x) - \nabla f(y) }^2 \le 0. 
    \end{equation*}
   Note that when $f$ is a (global) $L$-smooth convex function, the above statements hold with $\bar{L}_{K} = L$. 
\end{lemma}

\begin{corollaryL}   \label{cor:L_k_bound}
    Let $f$ be a locally smooth convex function. Then following statements are true:
    \begin{itemize}
        \item [(i)] For all $x,y\in\ourspace$, if $f(y) - f(x) + \inner{\nabla f(y)}{x-y}=0$, then $\norm{ \nabla f(x) - \nabla f(y) }^2 = 0$. 
        \item [(ii)] Let $K\subset\ourspace$ be a compact set. Then there exists $\bar{L}_K>0$ such that 
    \begin{equation*}
        0\le -\frac{ \frac{1}{2} \norm{ \nabla f(x) - \nabla f(y) }^2 }
            { f(y) - f(x) + \inner{\nabla f(y)}{x-y} }
        \le \bar{L}_K, 
        \quad \forall x,y \in K \mbox{ with } f(y) - f(x) + \inner{\nabla f(y)}{x-y}\ne0.
    \end{equation*}  
    \end{itemize}  
\end{corollaryL}

In our analysis, we need an initial guess (say, $L_0$) to the 
smoothness parameter $L$ such that $0 < L_0 \leq L$. This can be done by choosing two arbitrary points $x_0$, $\tilde{x}_{0}$ with $x_0\neq \tilde{x}_{0}$ and defining 
\begin{equation} \label{eq:initial_guess}
    L_0 = \frac{\norm{ \nabla f(x_0) - \nabla f(\tilde{x}_{0}) }}{ \norm{ x_0 - \tilde{x}_{0} } }.
\end{equation}
Moreover, we adopt the convention $\frac{0}{0} = 0$ and $\frac{a}{0} = +\infty$ for all $a > 0$ when the denominator is zero.

\section{Our AdaNAG Algorithm} \label{sec:adanag}

One typical iteration of OGM and NAG can be written in the following form: 
\begin{equation*}
    \begin{aligned}
        y_{k+1} = x_k - s_{k} \nabla f(x_k), \quad \,\,
        z_{k+1} = z_k - s_{k} \alpha_{k} \theta_k \nabla f(x_k), \quad \,\,
        x_{k+1} = \pr{ 1 - \frac{1}{\theta_{k+1}} } y_{k+1} + \frac{1}{\theta_{k+1}} z_{k+1}.
    \end{aligned}
\end{equation*}
Both methods consider constant choice $s_{k} = \frac{1}{L}$. 
As studied in \cite{KimFessler2016_optimized, ParkParkRyu2023_factorsqrt2}, this reduces to OGM when $\alpha_{k} = 2$, and reduces to NAG when $\alpha_{k} = 1$. 
We want to find proper $s_{k}$, $\alpha_{k}$ without knowing $L$ such that similar proof structure of OGM works, so to achieve similar accelerated convergence rate.

\subsection{AdaNAG}
Our \emph{\textbf{Ada}ptive \textbf{N}esterov's \textbf{A}ccelerated \textbf{G}radeint (AdaNAG)} method is described in Algorithm \ref{alg:AdaNAG}.
The convergence results of \adanag\ are summarized in Theorem~\ref{thm:main_tight}, and its proof is given in \Cref{sec:AdaNAG_anaylsis}.

\begin{algorithm}[ht] 
     \caption{\textbf{Ada}ptive \textbf{N}esterov \textbf{A}ccelerated \textbf{G}radient (AdaNAG)}
     \label{alg:AdaNAG} 
     \begin{algorithmic}[1]
     \STATE \textbf{Input:} $x^0=z^0 \in \mathbb{R}^d$, 
     $\ssz_{0} >0$. 
        \STATE \textbf{Define}
        \begin{align} \label{eq:theta}
            \theta_k &= \begin{cases}
               1 & \text{ if } k=0 
               \\ 
                \frac{1}{2} {\scriptstyle \pr{ 1 + \sqrt{1 + 4\theta_{k-1}^2}} }, \quad\,\,  & \text{ if } k\ge1,
            \end{cases}
            \qquad 
            \alpha_k = \begin{cases}
                \frac{2\theta_2}{\theta_2-1} \pr{ \frac{1}{\alpha_3} + \frac{1}{\alpha_2^2} - \frac{1}{\alpha_1} }^{-1} & \text{ if } k=0 \\ 
                \frac{1}{2} \pr{ 1 - \frac{1}{\theta_{k+2}} } & \text{ if } k\ge1
            \end{cases}
        \end{align}
       \FOR{$k = 0,1,\dots$}   
      \STATE 
       \begin{equation}     \label{eq:AdaNAG}  
            \begin{aligned}
            y_{k+1} &= x_k -  \ssz_{k} \nabla f(x_k) \\
            z_{k+1} &= z_k - \ssz_{k} \talpha_{k} \theta_{k+2} \nabla f(x_k) \\
            x_{k+1} &= \pr{ 1 - \frac{1}{\theta_{k+3}} } y_{k+1} + \frac{1}{\theta_{k+3}} z_{k+1}  
            \end{aligned}
        \end{equation} \vspace{-2mm}
        \begin{align}\label{define-Lk+1}
            L_{k+1} &= -\frac{ \frac{1}{2} \norm{ \nabla f(x_{k+1}) - \nabla f(x_{k}) }^2 }            { f(x_{k+1}) - f(x_{k}) + \inner{\nabla f(x_{k+1})}{x_{k}-x_{k+1}} }
        \end{align}
        \begin{equation} \label{eq:step_size_rule} \qquad\,\,
            \ssz_{k+1} = \begin{cases}
                \min\set{ \frac{\alpha_0}{\alpha_1} \frac{\theta_2}{\theta_3(\theta_3-1)} \ssz_{0},  \,\,\frac{ \alpha_{2}^2 \alpha_{3}}{\alpha_{3} + \alpha_{2}^2} \frac{1}{\alpha_{1}} \frac{1}{L_{1}}  } \qquad\qquad & \text{if }  k=0  \\
                \min\set{ \frac{\alpha_{k}}{\alpha_{k+1}}  \ssz_{k},  \,\, \frac{ \alpha_{k}^2 }{ { \alpha_{k+1} + \alpha_{k}^2 } } \frac{1}{L_{k+1}}  } & \text{if }  k\ge 1
            \end{cases}
        \end{equation}
       \ENDFOR
     \end{algorithmic}
\end{algorithm}

\begin{theorem} 
    \label{thm:main_tight}
    Let $f$ be an $L$-smooth convex function. 
    Suppose $\set{x_k}_{k\ge0}$ is a sequence generated by \adanag\ (Algorithm \ref{alg:AdaNAG}) 
    with $s_0 = 0.4255 \frac{1}{L_0}$ where $L_0$ defined as in \eqref{eq:initial_guess}.
    Denote 
    $\upperboundconstant =  \norm{x_{0} - x_\star}^2 +  0.14  \frac{1}{L_0} \pr{\frac{1}{L_0} - \frac{2}{L} } \norm{\nabla f(x_0)}^2$.  
    Then the following convergence rate results hold.
    \begin{align}
        f(x_k) - f_\star
            &\le \frac{5.5 \upperboundL}{  \theta_{k+2}^2}  \upperboundconstant 
            \le \frac{ 22 \upperboundL}{  (k+4)^2}  \upperboundconstant
            = \cO\pr{ \frac{\upperboundL}{k^2} }, \label{thm:main_tight-1}
        \\
        \min_{i\in\set{1,\dots, k}} \norm{\nabla f(x_i)}^2
            &\le \frac{120 \upperboundL^2}{\sum_{i=1}^{k}  \theta_{i+2}^2}  \upperboundconstant
            \le \frac{1440 \upperboundL^2}{k \left(k^2+12 k+47\right)} \upperboundconstant = \cO\pr{ \frac{\upperboundL^2}{k^3} }.\nonumber
    \end{align}
\end{theorem}

Note that when the denominator of $L_{k+1}$ in \eqref{define-Lk+1} is zero, the numerator is also zero by Corollary~\ref{cor:L_k_bound}, so we have $L_{k+1}=0$ from our arithmetic convention. 
There are several useful properties of $\theta_k$, $\alpha_k$ and $s_k$ and we summarize them in Lemma \ref{lem:theta_properties}, whose proof is given in \Cref{appendix:theta_properties_proof}.

\begin{lemma}   \label{lem:theta_properties}
    Suppose $\set{\theta_k}_{k\ge0}$, $\set{\alpha_k}_{k\ge0}$ and $\set{s_k}_{k\ge0}$ are defined as in \eqref{eq:theta} and \eqref{eq:step_size_rule}. 
    Then $\theta_k$ is an increasing sequence that satisfies $\theta_k\ge\frac{k+2}{2}$ for $k\ge0$. 
    Furthermore, the following inequality holds
    \begin{equation}    \label{eq:theta_motivation}
        \theta_{k+1}(\theta_{k+1}-1) - \theta_k^2 \le 0, \qquad \forall k\ge0.
    \end{equation}
    Thus, $\alpha_{k+1}\ge\alpha_k$ for $k\ge1$ and $\lim_{k\to\infty}\alpha_k=\frac{1}{2}$. 
    Moreover, $s_{k+1}\le s_k$ for $k\ge0$. 
\end{lemma}
Note that the definition of $\theta_k$ in \eqref{eq:theta} is the same as the one used for the original NAG \cite{Nesterov1983_method}, which satisfies \eqref{eq:theta_motivation} as an equality. 
The inequality \eqref{eq:theta_motivation} is often considered a core property of $\theta_k$ in studies of variants of NAG 
\cite{BeckTeboulle2009_fast, ChambolleDossal2015_convergence, ParkParkRyu2023_factorsqrt2}, 
and will play a crucial role in our convergence analysis. 
As an additional comment, we note that
\begin{equation}    \label{eq:rational_theta}
    \theta_k = \frac{k+2}{2}    
\end{equation}
is also often considered an alternative to \eqref{eq:theta}, as it behaves similarly but is simpler and more intuitive. 
We show that a simplified variant of \adanag\ obtained by replacing \eqref{eq:theta} with \eqref{eq:rational_theta} can also achieve a similar convergence rate in \Cref{sec:adanag-s}.

\subsection{Comparison with AC-FGM}

AC-FGM is another adaptive method that achieves accelerated convergence rate, studied in \cite{LiLan2024_simple}. 
When $h=0$, $\mathcal{X}=\bbR^d$, AC-FGM reduces to:
\begin{equation}   \label{eq:AC-FGM} \tag{AC-FGM}
    \begin{aligned}
        z_{k+1} &=  y_{k} - \eta_{k+1} \nabla f(x_{k}) \\
        y_{k+1} &= (1-\beta_{k+1})y_{k} + \beta_{k+1} z_{k+1} \\
        x_{k+1} &= \frac{\tau_{k+1}}{1+\tau_{k+1}} x_{k} + \frac{1}{1+\tau_{k+1}} z_{k+1}.
    \end{aligned}
\end{equation}
Note that \ref{eq:AC-FGM} has a different form from our \adanag, 
and therefore, the mechanism of the coefficients differs. 
The concrete choice of parameters in AC-FGM are (\cite[Corollary~1]{LiLan2024_simple}):  
\begin{equation*}
    \begin{aligned}
        \tau_k &= \begin{cases}
            0   &  \text{if } k=1 \\
            \frac{k}{2} & \text{if } k\ge2, 
        \end{cases}
        \qquad\quad
        \beta_k = \begin{cases}
            0   &  \text{if } k=1 \\
            \beta & \text{if } k\ge2,
        \end{cases}
        \qquad\quad
        \eta_k = \begin{cases}
            \min\set{ (1-\beta) \eta_1, \,\, \frac{1}{4L_1} }   &  \text{if } k = 2 \\
            \min\set{ \eta_2, \,\, \frac{1}{4L_2} }   &  \text{if } k = 3 \\
            \min\set{ \frac{k}{k-1} \eta_{k-1}, \,\, \frac{k-1}{8L_{k-1}} }   &  \text{if } k \ge 4, 
        \end{cases}
    \end{aligned}
\end{equation*}
with $L_{k}$ defined in the same way as \eqref{define-Lk+1}. Then \ref{eq:AC-FGM} has the convergence rate
\begin{equation}\label{AC-FGM-main}
    f(x_k) - f_\star 
    \le \frac{ ({12}/{\beta}) \hat{L}_k}{(k+1)(k+2)} 
    \pr{ \norm{ x_0 - x_\star }^2 + \beta \pr{ \frac{5 \eta_2 L_1}{2} - \frac{\eta_2}{\eta_1} } \eta_1^2 \norm{ \nabla f(x_0) }^2 },
\end{equation}
where $\eta_1>0$ is arbitrary and $\hat{L}_k =  \max\big\{\frac{1}{4(1-\beta)\eta_1}, L_1, \dots, L_k\big\}$, $\beta \in \big( 0 , 1 - \sqrt{6}/{3} \big]$. 
The coefficient of $\norm{\nabla f(x_0)}^2$ differs depending on the relationship between $\eta_1$, $L_1$, and $\eta_2$.  
Thus, this coefficient fundamentally contains some uncertainty, depending on how well we estimated the initial step size $\eta_1$ and the point $x_1$, which we do not have full control over.  
The authors provided a line search technique to adjust this coefficient, but in \cite[\S~4.4]{LiLan2024_simple}, they presented an ablation study showing that such an adjustment does not make a significant difference in various practical scenarios.  
In this context, we believe there is no absolute way to compare this coefficient, or that this coefficient is not crucial to the algorithm's behavior.
Therefore, we focus more on the coefficient of $\norm{x_0 - x_\star}^2$.  
In \cite{LiLan2024_simple}, they choose $\beta = 1 - \frac{\sqrt{6}}{3}$ for the experiments, which also results in the tightest coefficient of $\norm{x_0 - x_\star}^2$, which is ${12}/{(1 - \frac{\sqrt{6}}{3})} \approx 65.39.$ 
This is nearly three times larger compared to the coefficient of $\norm{x_0 - x_\star}^2$ in \eqref{thm:main_tight-1}, which is $22$.

\section{Convergence analysis of AdaNAG}     \label{sec:AdaNAG_anaylsis}

We note that the proofs in this section only assume $f$ to be locally smooth rather than globally smooth, except for Corollary~\ref{lemma:stepsize_lowerbound_L_smooth} and Proposition~\ref{prop:adanag_gradient_norm}. 
On the other hand, Corollary~\ref{lemma:stepsize_lowerbound_L_smooth} and Proposition~\ref{prop:adanag_gradient_norm} can even be achieved under the locally smooth condition if we consider a small perturbation to \eqref{eq:step_size_rule} as 
$\ssz_{k+1} = \min\big\{ \frac{\alpha_{k}}{\alpha_{k+1}} \ssz_{k}, \,\, \frac{ \alpha_{k}^2 }{ \alpha_{k+1} + \alpha_{k}^2 (1+10^{-6}) } \frac{1}{L_{k+1}} \big\}$. 
Since the value $10^{-6}$ is negligibly small compared to $\alpha_k^2$ and $\alpha_{k+1}$, the overall behavior of the perturbed variant remains the same as that of \adanag, including the convergence guarantee. 
However, since the analysis requires more technical arguments, we provide the formal discussion of this variant and its proof in \Cref{sec:locally_smooth_extension}.

\subsection{Motivation of the Lyapunov function and $L_{k+1}$} \label{sec:motivation_of_lyapunov_L}

Since we want to design an adaptive method that mimics OGM \cite{KimFessler2016_optimized, KimFessler2017_convergence}, it is worth briefly reviewing the proof structure of OGM. 
OGM considers Lyapunov function of the form 
\begin{equation}    \label{eq:OGM_lyapunov}
    V_k^{\text{OGM}} = \frac{1}{L} 2 \theta_{k+1}(\theta_{k+1}-1) \pr{ f(x_k) - f_\star } - \frac{1}{L^2} \theta_{k}^2 \norm{ \nabla f(x_k) }^2 + \frac{1}{2} \norm{ z_{k+1} - x_\star }^2.
\end{equation}
The core part of the convergence proof for OGM is to show:
\begin{equation}\label{OGM-Lyapunov-decrease}
    \begin{aligned}
        V_{k+1}^{\text{OGM}} - V_k^{\text{OGM}} 
        = \frac{2}{L} \theta_{k+1}(\theta_{k+1}-1) \ineqsonsecutive + \frac{2}{L} \theta_{k+1} \ineqstar 
         \le 0,
    \end{aligned}
\end{equation}
where $\ineqsonsecutive = f(x_{k+1}) -  f(x_k) + \inner{ \nabla f(x_{k+1}) }{ x_k - x_{k+1} } + \frac{1}{2L} \norm{ \nabla f(x_{k})-\nabla f(x_{k+1}) }^2$ and 
$\ineqstar= f(x_{k+1}) - f_\star - \inner{\nabla f(x_{k+1})}{x_{k+1} - x_\star} + \frac{1}{2L} \norm{ \nabla f(x_{k+1}) }^2$. 
For details, see \cite[\S~4.3]{dAspremontScieurTaylor2021_acceleration} or \cite{ParkParkRyu2023_factorsqrt2}. 
Note $\ineqsonsecutive$ and $\ineqstar$ are nonpositive since $f$ is $L$-smooth convex by \eqref{eq:L_smooth_ineq}. 
We want to set the parameters to construct a similar proof for \eqref{OGM-Lyapunov-decrease} when $L$ is not known and is replaced by some estimation $L_{k+1}$. 
A good way of doing it is to simply \emph{define} it as nonpositive. 
That is, define $L_{k+1}$ as in \eqref{define-Lk+1} 
and make $\ineqsonsecutive$ zero.
The idea of defining an adaptive step size motivated by the proof of an optimal non-adaptive method was also considered in \cite[\S~6]{SuhParkRyu2023_continuoustime}. 
We note that this specific definition of $L_{k+1}$ was also previously considered in \cite{LiLan2024_simple}. 

However, we cannot treat $\ineqstar$ using the same approach because it requires information about the optimal point and value, $f_\star$ and $x_\star$. 
Thus instead, we forgo the term $\frac{2}{L} \theta_{k+1} \times \frac{1}{2L} \norm{ \nabla f(x_{k+1}) }^2$ and use the relaxed inequality
$f(x_{k+1}) - f_\star - \inner{\nabla f(x_{k+1})}{x_{k+1} - x_\star} \le 0$.
Then, the $\norm{ \nabla f(x_{k+1}) }^2$ term should be accounted for somewhere.  
This requires us to adjust the definition of the Lyapunov function by flipping the sign of the $\norm{ \nabla f(x_{k}) }^2$ term. 
This ultimately leads to our Lyapunov function for \adanag. 
The details will be provided in the next section.

\subsection{Convergence analysis of AdaNAG}

The following theorem is the core part of our convergence analysis and one of the main contributions of our paper.  

\begin{theorem} \label{lemma:lyapunov_analysis}
    For \adanag\ (Algorithm \ref{alg:AdaNAG}), define Lyapunov function as 
    \begin{equation}    \label{eq:Lyapunov}
        V_{k}  
        = \ssz_{k+1} A_k \pr{ f(x_k) - f_\star } 
        +  \frac{1}{2} \ssz_{k}^2  B_k \norm{ \nabla f(x_k) }^2
        + \frac{1}{2} \norm{ z_{k+1} - x_\star }^2, \qquad \forall k\ge0,
    \end{equation} 
    where $A_k$, $B_k$ are defined as
    \begin{equation}   \label{eq:A_B_AdaNAG}
        \begin{aligned} \,\,
            A_k := \alpha_{k+1} \theta_{k+3} \pr{ \theta_{k+3} - 1 }, \qquad
            B_k := 
            \begin{cases}
                \oldalpha_0 \talpha_0 \theta_2^2   & \text{ if } k=0 \\
               \alpha_{k}^2 \theta_{k+2}^2  & \text{ if } k\ge1,
            \end{cases}
            \qquad \mbox{ with }\,\,
            \oldalpha_0 := \frac{1}{2} \pr{ 1 - \frac{1}{\theta_2} }.
        \end{aligned}
    \end{equation}  
    Then for a locally smooth function $f$, we have 
    \begin{equation*}
        V_{k+1} - V_k 
        \le -\frac{1}{2} \min\set{ s_k^2 B_k, \frac{\ssz_{k+1}}{L_{k+1}} A_k  } \norm{ \nabla f(x_k) }^2 \le 0, \qquad k\ge0.
    \end{equation*} 
\end{theorem}

\begin{proof}
    We first consider the case when $\nabla f(x_k) \ne \nabla f(x_{k+1})$,  
    which implies that the denominator of $L_{k+1}$ is nonzero by Corollary~\ref{cor:L_k_bound}.
    The proof can be broadly summarized into four parts:

    \paragraph{Part 1.} (The proof is given in \Cref{lem:ineq:AdaNAG_generalized_stepsize}). Show the following inequalities hold for $k\ge0$. Note that these are the core properties of the step size and parameters.      
    \begin{subequations}    \label{ineq:AdaNAG_generalized_stepsize}
        \begin{align}
            \ssz_{k+2} &\le \frac{A_{k} + \alpha_{k+1}\theta_{k+3}}{A_{k+1}}  s_{k+1} \label{ineq:AdaNAG_generalized_stepsize-a}\\
            \ssz_{k+1} &\le \rsBA \ssz_{k}    \label{ineq:AdaNAG_generalized_stepsize-b}\\
            \ssz_{k+1} &\le \pr{ \frac{A_k}{B_k} + \rsBAreciprocal }^{-1} \frac{1}{L_{k+1}},\label{ineq:AdaNAG_generalized_stepsize-c}
        \end{align}
    \end{subequations}
    where
    \begin{equation}    \label{eq:adanag_r_definition}
        \rsBAreciprocal = \frac{B_{k+1} + \alpha_{k+1}^2\theta_{k+3}^2}{A_k}.
    \end{equation}

    \paragraph{Part 2.} (The proof is given in \Cref{lem:lyapunov_diff}). Show the following equality holds for $k\ge0$:
    \begin{equation} \label{eq:lyapunov_diff}
        \begin{aligned}
            &{V}_{k+1} - {V}_k  \\
            &= ( \ssz_{k+1} (A_k + \alpha_{k+1}\theta_{k+3} ) - \ssz_{k+2} A_{k+1} ) \pr{ f_\star - f(x_{k+1}) }   \\ &\phantom{=}
            + \ssz_{k+1}\alpha_{k+1}\theta_{k+3} \pr{ f(x_{k+1}) - f_\star - \inner{\nabla f(x_{k+1})}{x_{k+1} - x_\star} } \\ &\phantom{=}
            + \ssz_{k+1} A_k \pr{ f(x_{k+1}) -  f(x_k) + \inner{ \nabla f(x_{k+1}) }{ x_k - x_{k+1} } + \frac{1}{2L_{k+1}} \norm{ \nabla f(x_{k})-\nabla f(x_{k+1}) }^2 } \\ &\phantom{=}
             + Q_k, 
        \end{aligned}
    \end{equation}
    where $Q_k$ is defined as 
    \begin{align}   \label{eq:adanag_lyapunov_diff_quadratic}
        Q_k
        = &- \frac{\ssz_{k+1} A_k}{2L_{k+1}} \Bigg[  \pr{1 - \rsBAreciprocal  \ssz_{k+1}  L_{k+1} } \norm{ \nabla f(x_{k+1}) }^2
         \\    \nonumber
        &\quad\qquad -  2 \pr{ 1 - \ssz_{k} L_{k+1} } \inner{\nabla f(x_{k+1})}{  \nabla f(x_k) }  
        + \pr{ 1 +   \frac{B_k}{A_k} \frac{ \ssz_{k}^2 }{ \ssz_{k+1}} L_{k+1}  } \norm{ \nabla f(x_{k}) }^2
         \Bigg]. 
    \end{align}

    \paragraph{Part 3.} (The proof is given in \Cref{lem:determinant}). 
    Show the following inequality holds 
    \begin{equation}    \label{ineq:determinant}
    \qquad\qquad Q_k \le \Vdifferencebound_k, \qquad \forall k\ge0, 
    \end{equation} 
    where
    \begin{equation*}
    \Vdifferencebound_k := \begin{cases}
            -\frac{1}{2} s_k^2 B_k \norm{ \nabla f(x_k) }^2 \qquad & \text{ if } s_kL_{k+1} \le 1 \\
            -\frac{1}{2} \frac{\ssz_{k+1}}{L_{k+1}} A_k \norm{ \nabla f(x_k) }^2 & \text{ if } s_kL_{k+1} \ge 1.
        \end{cases}
    \end{equation*}

    \paragraph{Part 4.} 
    Finally, we show that the first three lines on the right hand side of \eqref{eq:lyapunov_diff} are nonpositive:
    \begin{itemize}
        \item $\ssz_{k+1}(A_k + \alpha_{k+1}\theta_{k+3}) - s_{k+2} A_{k+1} \ge 0$ 
        follows from \eqref{ineq:AdaNAG_generalized_stepsize-a}. 
        \item $f(x_{k+1}) - f_\star - \inner{\nabla f(x_{k+1})}{x_{k+1} - x_\star} \le 0$ holds by the convexity of $f$. 
        \item $f(x_{k+1}) -  f(x_k) + \inner{ \nabla f(x_{k+1}) }{ x_k - x_{k+1} } + \frac{1}{2L_{k+1}} \norm{ \nabla f(x_{k})-\nabla f(x_{k+1})}^2 = 0$ holds by \eqref{define-Lk+1}. 
    \end{itemize}
    Hence, we obtain that 
    \begin{equation}    \label{ineq:core_is_quadratic}
        V_{k+1} - V_k \leq Q_k, \qquad \forall k\ge0,
    \end{equation}
    and the desired conclusion for the case $\nabla f(x_k) \neq \nabla f(x_{k+1})$ follows from \textbf{Part 3}. 

\paragraph{When $\nabla f(x_k) = \nabla f(x_{k+1})$.} 
We have $L_{k+1} = 0$ in this case, 
and hence $s_k L_{k+1} = 0 \leq 1$. 
Now, by conducting the same calculation as in \textbf{Part 2} but without the term $\frac{\ssz_{k+1} A_k}{2L_{k+1}} \norm{ \nabla f(x_{k+1}) - \nabla f(x_k) }^2$, 
we obtain the same equation as \eqref{eq:lyapunov_diff}, except that this term is omitted:
\begin{equation*} 
    \begin{aligned}
        &{V}_{k+1} - {V}_k   \\
        &=  \pr{ \ssz_{k+1} (A_k + \alpha_{k+1}\theta_{k+3}) - s_{k+2} A_{k+1}  } \pr{ f_\star - f(x_{k+1}) } \\ &\quad
        + \ssz_{k+1}\alpha_{k+1}\theta_{k+3} \pr{ f(x_{k+1}) - f_\star - \inner{\nabla f(x_{k+1})}{x_{k+1} - x_\star} } \\ &\quad
        + \ssz_{k+1} A_k \pr{ f(x_{k+1}) -  f(x_k) + \inner{ \nabla f(x_{k+1}) }{ x_k - x_{k+1} } } \\ &\quad
        + \frac{1}{2} \ssz_{k+1}^2 \pr{ B_{k+1}+\alpha_{k+1}^2\theta_{k+3}^2 } \norm{ \nabla f(x_{k+1}) }^2
        - \frac{1}{2} \ssz_{k}^2 B_k  \norm{ \nabla f(x_{k}) }^2 
        - \ssz_{k} \ssz_{k+1} A_k \inner{\nabla f(x_{k+1})}{  \nabla f(x_k) } \\ 
        &\le \frac{1}{2} \ssz_{k+1}^2 \pr{ B_{k+1}+\alpha_{k+1}^2\theta_{k+3}^2 } \norm{ \nabla f(x_{k+1}) }^2 
        - \frac{1}{2} \ssz_{k}^2 B_k  \norm{ \nabla f(x_{k}) }^2 
        - \ssz_{k} \ssz_{k+1} A_k \inner{\nabla f(x_{k+1})}{  \nabla f(x_k) } \\ 
    &= \frac{1}{2} \ssz_{k+1} A_k  \pr{ \frac{B_{k+1} + \alpha_{k+1}^2\theta_{k+3}^2}{A_k} \ssz_{k+1}  - \ssz_{k}  } \norm{ \nabla f(x_{k}) }^2 
         - \frac{1}{2} \ssz_{k}^2 B_k \norm{ \nabla f(x_{k}) }^2 - \frac{1}{2} s_k s_{k+1} A_k \norm{ \nabla f(x_{k}) }^2 \\ 
         & \le - \frac{1}{2} \ssz_{k}^2 B_k \norm{ \nabla f(x_{k}) }^2,
    \end{aligned}
\end{equation*}
where the second inequality is due to 
\eqref{ineq:AdaNAG_generalized_stepsize-b}.  
\end{proof}

\begin{lemma}   \label{lem:ineq:AdaNAG_generalized_stepsize}
    Inequalities \eqref{ineq:AdaNAG_generalized_stepsize} hold for $k\ge0$.
\end{lemma}

\begin{proof}
    Recall that $A_k$ and $B_k$ are defined as in \eqref{eq:A_B_AdaNAG}. 
    From \eqref{eq:step_size_rule} and \eqref{eq:theta_motivation}, we have 
    \begin{equation*}
        \ssz_{k+2} 
        \le \frac{\alpha_{k+1}}{\alpha_{k+2}} \ssz_{k+1}
        \ThetaProprtyUsed{\le} \frac{\alpha_{k+1}}{\alpha_{k+2}} \frac{\theta_{k+3}^2}{\theta_{k+4}\pr{\theta_{k+4}-1}} \ssz_{k+1}
        = \frac{A_k + \alpha_{k+1}\theta_{k+3}}{A_{k+1}} \ssz_{k+1}, 
        \qquad \forall k\geq 0,
    \end{equation*}
    which proves \eqref{ineq:AdaNAG_generalized_stepsize-a}.
    Next, from \eqref{eq:theta}, 
    we have
    \begin{equation}    \label{eq:AdaNAG_AB_1}
        \rsBAreciprocal
        = \frac{A_k}{B_{k+1} + \alpha_{k+1}^2\theta_{k+3}^2}
        = \frac{\alpha_{k+1}\theta_{k+3}(\theta_{k+3}-1)}{2\alpha_{k+1}^2\theta_{k+3}^2}
        \AlphaDefinitionUsed{=} \frac{\frac{1}{2}\pr{ 1 - \frac{1}{\theta_{k+3}} }}{\alpha_{k+1}}
        = 1.
    \end{equation}
    Thus \eqref{ineq:AdaNAG_generalized_stepsize-b} is equivalent to $s_{k+1} \le s_k$ for $k \geq 0$,  
    which holds true by \Cref{lem:theta_properties}. This proves \eqref{ineq:AdaNAG_generalized_stepsize-b}.  
    Lastly, from \eqref{eq:theta_motivation} we have 
    \begin{equation*}
        \frac{A_k}{B_k}
        = \frac{\alpha_{k+1}}{\alpha_{k}^2} \frac{\theta_{k+3}(\theta_{k+3}-1)}{\theta_{k+2}^2}
        \ThetaProprtyUsed{\le} \frac{\alpha_{k+1}}{\alpha_{k}^2}, \qquad \forall k\ge1,
    \end{equation*}
    which, together with  \eqref{eq:AdaNAG_AB_1}, yields
    \begin{equation*}
        \pr{ \frac{A_k}{B_k} + \rsBAreciprocal }^{-1}
        \AlphaDefinitionUsed{=} \pr{ \frac{A_k}{B_k} + 1 }^{-1}
        \ThetaProprtyUsed{\ge} \pr{ \frac{\alpha_{k+1}}{\alpha_{k}^2} + 1 }^{-1}
        = \frac{ \alpha_{k}^2 }{ \alpha_{k+1} + \alpha_{k}^2 }, \qquad \forall k\ge1.
    \end{equation*}
    For $k=0$, from \eqref{eq:theta} and \eqref{eq:A_B_AdaNAG} 
    we know $\alpha_0 = \frac{1}{\oldalpha_0} \pr{ \frac{1}{\alpha_3} + \frac{1}{\alpha_2^2} - \frac{1}{\alpha_1} }^{-1}$,
    which, together with \eqref{eq:theta_motivation}, yields
    \begin{equation*}
        \frac{ \alpha_{2}^2 \alpha_{3}}{\alpha_{3} + \alpha_{2}^2} \frac{1}{\alpha_{1}}
        \AlphaDefinitionUsed{ = } \pr{\frac{\alpha_1}{\oldalpha_0\alpha_0} + 1}^{-1}
        \ThetaProprtyUsed{\le } \pr{ \frac{\alpha_1}{\oldalpha_0\alpha_0} \frac{\theta_3(\theta_3-1)}{ \theta_2^2} + 1}^{-1} 
        = \pr{ \frac{A_0}{B_0} + \rho_0 }^{-1}.
    \end{equation*}
    Combining the above two inequalities with \eqref{eq:step_size_rule} proves \eqref{ineq:AdaNAG_generalized_stepsize-c}. 
\end{proof}

\begin{lemma}   \label{lem:lyapunov_diff}
    Equality \eqref{eq:lyapunov_diff} holds for $k\ge0$.
\end{lemma}

\begin{proof}
    First, from \eqref{eq:AdaNAG} we know  
    \begin{equation*}
        z_{k+2} - z_{k+1} = -\ssz_{k+1}\alpha_{k+1} \theta_{k+3} \nabla f(x_{k+1}),
    \end{equation*}
    and thus we have
    \begin{align}    \label{eq:z_square_diff}
            \frac{1}{2} \norm{ z_{k+2} - x_\star }^2 - \frac{1}{2} \norm{ z_{k+1} - x_\star }^2  
            &= \frac{1}{2} \inner{ z_{k+2} - z_{k+1} }{ z_{k+1} + z_{k+2} - 2 x_{\star}  }  \\ 
            &= \frac{1}{2} \inner{ z_{k+2} - z_{k+1} }{ 2 \pr{ z_{k+1} - x_\star } + \pr{ z_{k+2} - z_{k+1} }  } \nonumber \\ 
            &= - \ssz_{k+1}\alpha_{k+1} \theta_{k+3} \inner{ \nabla f(x_{k+1}) }{ z_{k+1} - x_\star } + \frac{1}{2} \ssz_{k+1}^2\alpha_{k+1}^2 \theta_{k+3}^2 \norm{ \nabla f(x_{k+1}) }^2.  \nonumber 
    \end{align}
    Combining \eqref{eq:z_square_diff} with the following identity, 
    \begin{equation*}
        \begin{aligned}
            s_{k+2} A_{k+1} \pr{ f(x_{k+1}) - f_\star } -  \ssz_{k+1} A_k \pr{ f(x_k) - f_\star }
            &= ( \ssz_{k+1} (A_k + \alpha_{k+1}\theta_{k+3} ) - s_{k+2} A_{k+1} ) \pr{ f_\star - f(x_{k+1}) }   \\ &\quad
            + \ssz_{k+1}\alpha_{k+1}\theta_{k+3} \pr{ f(x_{k+1}) - f_\star - \inner{\nabla f(x_{k+1})}{x_{k+1} - x_\star} }   \\ &\quad
            + \ssz_{k+1} A_k \pr{ f(x_{k+1}) - f_\star } -  \ssz_{k+1} A_k \pr{ f(x_k) - f_\star } \\&\quad 
            + \ssz_{k+1}\alpha_{k+1} \theta_{k+3} \inner{\nabla f(x_{k+1})}{x_{k+1} - x_\star},
        \end{aligned}
    \end{equation*} 
    we have 
    \begin{align} \label{eq:lyapunov_difference_without_grad}
        &s_{k+2} A_{k+1} \pr{ f(x_{k+1}) - f_\star } + \frac{1}{2} \norm{ z_{k+2} - x_\star }^2 - \pr{ \ssz_{k+1} A_k \pr{ f(x_k) - f_\star } + \frac{1}{2} \norm{ z_{k+1} - x_\star }^2 } \nonumber \\
        &= ( \ssz_{k+1} (A_k + \alpha_{k+1}\theta_{k+3} ) - s_{k+2} A_{k+1} ) \pr{ f_\star - f(x_{k+1}) }   \\ &\quad \nonumber   
        + \ssz_{k+1}\alpha_{k+1}\theta_{k+3} \pr{ f(x_{k+1}) - f_\star - \inner{\nabla f(x_{k+1})}{x_{k+1} - x_\star} }   \\ &\quad
        + \ssz_{k+1} A_k \pr{ f(x_{k+1}) - f(x_k) } 
        + \ssz_{k+1}\alpha_{k+1} \theta_{k+3} \inner{\nabla f(x_{k+1})}{x_{k+1} - z_{k+1}} + \frac{1}{2} \ssz_{k+1}^2 \alpha_{k+1}^2\theta_{k+3}^2  \norm{ \nabla f(x_{k+1}) }^2.  \nonumber 
    \end{align}
    The first two lines of \eqref{eq:lyapunov_difference_without_grad} are the desired terms in \eqref{eq:lyapunov_diff}. 
    We focus on reformulating the terms in the last line.  
    From \eqref{eq:AdaNAG}, we have
    \begin{equation*}
        x_{k+1} - z_{k+1} 
        = \pr{ \theta_{k+3} - 1 } \pr{ y_{k+1} - x_{k+1} }
        = -\pr{ \theta_{k+3} - 1 } \pr{ x_{k+1} - x_k + \ssz_{k} \nabla f(x_k) }.
    \end{equation*} 
    Plugging this into the last line of \eqref{eq:lyapunov_difference_without_grad}, we have 
    \begin{align}
        &\ssz_{k+1} A_k \pr{ f(x_{k+1}) - f(x_k) } 
            + \ssz_{k+1}\alpha_{k+1} \theta_{k+3} \inner{\nabla f(x_{k+1})}{x_{k+1} - z_{k+1}} + \frac{1}{2} \ssz_{k+1}^2 \alpha_{k+1}^2\theta_{k+3}^2 \norm{ \nabla f(x_{k+1}) }^2 \nonumber\\
        &= \ssz_{k+1} A_k \pr{ f(x_{k+1}) -  f(x_k) + \frac{1}{2L_{k+1}} \norm{ \nabla f(x_{k}) - \nabla f(x_{k+1}) }^2 }  \nonumber\\&\quad 
        -  \ssz_{k+1} \underbrace{  \alpha_{k+1} \theta_{k+3} \pr{ \theta_{k+3} - 1 } }_{=A_k} \inner{\nabla f(x_{k+1})}{ x_{k+1} - x_k + \ssz_{k} \nabla f(x_k) } \nonumber\\&\quad
        - \frac{\ssz_{k+1} A_k}{2L_{k+1}} \norm{ \nabla f(x_{k}) - \nabla f(x_{k+1}) }^2  + \frac{1}{2} \ssz_{k+1}^2 \alpha_{k+1}^2\theta_{k+3}^2 \norm{ \nabla f(x_{k+1}) }^2  \label{eq:lyapunov_difference_without_grad-2nd-line} \\
        &= \ssz_{k+1} A_k \pr{ f(x_{k+1}) -  f(x_k) + \inner{ \nabla f(x_{k+1}) }{ x_k - x_{k+1} } + \frac{1}{2L_{k+1}} \norm{ \nabla f(x_{k})-\nabla f(x_{k+1}) }^2 } \nonumber\\&\quad 
        + \ssz_{k+1} A_k \pr{ \frac{1}{L_{k+1}} - \ssz_{k} } \inner{\nabla f(x_{k+1})}{  \nabla f(x_k) } \nonumber\\&\quad
        - \frac{\ssz_{k+1} A_k}{2L_{k+1}} \pr{ 1 -    \frac{\alpha_{k+1}^2\theta_{k+3}^2 }{A_k} \ssz_{k+1} L_{k+1} } \norm{ \nabla f(x_{k+1}) }^2 - \frac{\ssz_{k+1} A_k}{2L_{k+1}} \norm{ \nabla f(x_{k}) }^2.\label{eq:lyapunov_difference_without_grad-3rd-line}
    \end{align} 
    From the definition of $V_k$ we know 
    \begin{equation*}
        \begin{aligned}
            &s_{k+2} A_{k+1} \pr{ f(x_{k+1}) - f_\star } + \frac{1}{2} \norm{ z_{k+2} - x_\star }^2 - \pr{ \ssz_{k+1} A_k \pr{ f(x_k) - f_\star } + \frac{1}{2} \norm{ z_{k+1} - x_\star }^2 } \\
            &= V_{k+1} - \frac{1}{2} \ssz_{k+1}^2 B_{k+1} \norm{ \nabla f(x_{k+1}) }^2 - \pr{ V_k - \frac{1}{2} \ssz_{k}^2 B_{k} \norm{ \nabla f(x_k) }^2 }, \qquad \forall k\ge0,
        \end{aligned}
    \end{equation*}
    which, together with \eqref{eq:lyapunov_difference_without_grad} and \eqref{eq:lyapunov_difference_without_grad-3rd-line}, yields the desired equation \eqref{eq:lyapunov_diff}. 
\end{proof}

\begin{lemma}   \label{lem:determinant}
    Inequality \eqref{ineq:determinant} holds for $k \geq 0$.  
\end{lemma}

\begin{proof} 
    Note that \eqref{ineq:determinant} is equivalent to 
    \begin{equation*}    
        \begin{aligned}
            \frac{2L_{k+1}}{\ssz_{k+1} A_k} Q_k + \frac{B_k}{A_k} \frac{ \ssz_{k}^2 }{ \ssz_{k+1}} L_{k+1} \norm{ \nabla f(x_k) }^2 &\le 0 
              \qquad \text{ if } s_kL_{k+1} \le 1 \\
            \frac{2L_{k+1}}{\ssz_{k+1} A_k} Q_k + \norm{ \nabla f(x_k) }^2 &\le 0  \qquad \text{ if } s_kL_{k+1} \ge 1.
        \end{aligned}
    \end{equation*}
    Moreover, by recalling the definition of $Q_k$ in \eqref{eq:adanag_lyapunov_diff_quadratic}, we know:
    \begin{equation*}
        \begin{aligned}
            \frac{2L_{k+1}}{\ssz_{k+1} A_k} Q_k 
            &= - \pr{1 - \rsBAreciprocal  \ssz_{k+1}  L_{k+1} } \norm{ \nabla f(x_{k+1}) }^2    \\   
            &\phantom{=} + 2 \pr{ 1 - \ssz_{k} L_{k+1} } \inner{\nabla f(x_{k+1})}{  \nabla f(x_k) }
            - \pr{ 1 +   \frac{B_k}{A_k} \frac{ \ssz_{k}^2 }{ \ssz_{k+1}} L_{k+1}  } \norm{ \nabla f(x_{k}) }^2.
        \end{aligned}
    \end{equation*}
    First, we focus on the coefficient of the term $\norm{\nabla f(x_{k+1})}^2$ in the above expression. 
    From \eqref{ineq:AdaNAG_generalized_stepsize-c}:
    \begin{equation*}   \qquad\qquad
        1 - \rsBAreciprocal \ssz_{k+1}  L_{k+1}
        \ge 1 - \rsBAreciprocal \pr{ \frac{A_k}{B_k} + \rsBAreciprocal }^{-1}
        \ge 0, \qquad \forall k\ge0.
    \end{equation*}
    Therefore, our goal reduces to showing that the discriminant of the quadratic form is nonpositive:
    \begin{equation*} 
        \begin{aligned}
            \pr{ 1 - \ssz_{k} L_{k+1} }^2 -  \pr{1 - \rsBAreciprocal  \ssz_{k+1}  L_{k+1} } &\le 0 \qquad \text{ if } s_kL_{k+1} \le 1, \\
            \pr{ 1 - \ssz_{k} L_{k+1} }^2 -  \pr{1 - \rsBAreciprocal  \ssz_{k+1}  L_{k+1} } \frac{B_k}{A_k} \frac{ \ssz_{k}^2 }{ \ssz_{k+1}} L_{k+1}&\le 0 \qquad \text{ if } s_kL_{k+1} \ge 1.
        \end{aligned}
    \end{equation*}
    The proof for each case is as follows.

\begin{itemize}[leftmargin=*]
    \item $s_kL_{k+1}\le1$. 
        From \eqref{ineq:AdaNAG_generalized_stepsize-b} we obtain
        \begin{equation*}
            \rsBAreciprocal \ssz_{k+1}  L_{k+1} 
             \le \ssz_{k}L_{k+1}, \qquad \forall k\ge0. 
        \end{equation*}
        For arbitrary $\delta_1, \delta_2\in[0,1]$, if $\delta_2\le \delta_1$ we know
        \begin{equation}    \label{ineq:when_sL_smaller_1}
            (1-\delta_1)^2
            - (1-\delta_2)
            \le (1-\delta_1) - (1-\delta_2) = \delta_2 - \delta_1 \le 0. 
        \end{equation}
        We obtain the desired result by substituting 
        $\delta_1 = \ssz_{k}L_{k+1}$ and $\delta_2 =  \rsBAreciprocal \ssz_{k+1}  L_{k+1}$.
    \item $s_kL_{k+1}\ge1$. 
        From \eqref{ineq:AdaNAG_generalized_stepsize-c} we know
        \begin{equation*} 
            \frac{1}{\ssz_{k+1}L_{k+1}} \ge \frac{A_k}{B_k} + \rsBAreciprocal, \qquad \forall k\ge0,
        \end{equation*} 
        which, together with the fact that $1 - 2\ssz_k L_{k+1} \leq 0$, yields
        \begin{equation*}    
            \begin{aligned}
                &\pr{1 - \rsBAreciprocal \ssz_{k+1}  L_{k+1} } \frac{ B_k }{ A_k } \frac{ \ssz_{k}^2 }{ \ssz_{k+1} } L_{k+1}  
            =  \frac{ B_k }{ A_k } \pr{ \frac{1}{\ssz_{k+1}L_{k+1}}  -  \rsBAreciprocal  } \ssz_{k}^2 L_{k+1}^2
            \ge \ssz_{k}^2 L_{k+1}^2
            \ge  \pr{ 1 - \ssz_{k} L_{k+1} }^2.
            \end{aligned}
        \end{equation*}
\end{itemize}
This completes the proof. 
\end{proof}

Since $V_0$ involves $x_1$ and $z_1$, which are not the initial points, we need to define $V_{-1}$ and prove $V_0\le V_{-1}$. This is done in Lemma \ref{lemma:lyapunov_analysis_k0}.
\begin{lemma} \label{lemma:lyapunov_analysis_k0}
    Define 
    \begin{equation*}
        {V}_{-1} = \frac{1}{2} \norm{x_{0} - x_\star}^2 +  \frac{1}{2} \ssz_{0}^2  \pr{   \oldalpha_0 + \talpha_0 } \talpha_0 \theta_{2}^2 \norm{ \nabla f(x_{0}) }^2.
    \end{equation*} 
    Let $f$ be a locally smooth convex function, and $L$ be a smoothness parameter of $f$ on $\bar{B}_{3\|x_0-x_\star\|}(x_\star)$. 
    Then, the following inequality holds
    \begin{equation*}
        V_0 \le V_{-1} -\frac{\ssz_{0} \talpha_0 \theta_{2}}{2L} \norm{ \nabla f(x_{0}) }^2. 
    \end{equation*}
\end{lemma}

\begin{proof} 
    Using \eqref{eq:z_square_diff} with $k=-1$, we have (note $x_0=z_0$):
    \begin{equation*}
        \begin{aligned}
            \frac{1}{2} \norm{ z_{1} - x_\star }^2  - \frac{1}{2} \norm{ x_{0} - x_\star }^2 
            &= - \ssz_{0} \talpha_0 \theta_{2} \inner{ \nabla f(x_{0}) }{ x_0- x_\star } + \frac{1}{2} \ssz_{0}^2 \talpha_0^2 \theta_{2}^2 \norm{ \nabla f(x_{0}) }^2,
        \end{aligned}
    \end{equation*}
    which implies 
\begin{equation*}
    \begin{aligned}
        V_{0} - V_{-1} 
        &= \ssz_{1} \talpha_1 \theta_{3} \pr{ \theta_{3} - 1 } \pr{ f(x_{0}) - f_\star } 
        + \frac{1}{2} \norm{ z_{1} - x_\star }^2 
        + \frac{1}{2} \ssz_{0}^2 \oldalpha_0 \talpha_0 \theta_{2}^2 \norm{ \nabla f(x_{0}) }^2  \\&\quad
        - \pr{ \frac{1}{2} \norm{x_{0} - x_\star}^2 +  \frac{1}{2} \ssz_{0}^2  \oldalpha_0 \talpha_0 \theta_{2}^2 \norm{ \nabla f(x_{0}) }^2 +  \frac{1}{2} \ssz_{0}^2  \talpha_0^2 \theta_{2}^2 \norm{ \nabla f(x_{0}) }^2 } \\
        &= \pr{ \ssz_{0} \talpha_0 \theta_{2} - \ssz_{1} \talpha_1 \theta_{3} \pr{ \theta_{3} - 1 } } \pr{ f_\star - f(x_{0}) } \\&\quad 
        +  \ssz_{0} \talpha_0 \theta_{2} \pr{ f(x_{0}) - f_\star - \inner{ \nabla f(x_{0}) }{ x_{0} - x_\star } } \\
        & \le \ssz_{0} \talpha_0 \theta_{2} \pr{ f(x_{0}) - f_\star - \inner{ \nabla f(x_{0}) }{ x_{0} - x_\star } } \\
        & \le -\frac{\ssz_{0} \talpha_0 \theta_{2}}{2L} \norm{ \nabla f(x_{0}) }^2,
    \end{aligned}
\end{equation*}
where the first inequality is due to \eqref{eq:step_size_rule},  
and the second inequality is due to \Cref{lemma:local_smoothness}.  
\end{proof}

Next, we establish a lower bound for $\ssz_{k} \alpha_{k}$. 

\begin{lemma} \label{lemma:stepsize_lowerbound}
    Let $f$ be a locally smooth convex function. 
    Suppose $\set{s_k}_{k\ge0}$ is generated by \adanag.  
    Denote $r_0 = \frac{\theta_3(\theta_3-1)}{\theta_2} \frac{1}{\alpha_{0}} \frac{ \alpha_{2}^2 \alpha_{3}}{\alpha_{3} + \alpha_{2}^2}$ and set $s_0 = r_0 \inverLzero$ for some arbitrary $L_0 >0$.  
    Define
    \begin{equation}    \label{eq:large_S}
        S_{k} = \min\set{ \inverLzero ,\frac{1}{L_{1}}, \dots, \frac{1}{L_{k}} }.
    \end{equation}
    Then 
    \begin{equation*}
        \ssz_k \ge \frac{ \alpha_{2}^2 \alpha_{3}}{\alpha_{3} + \alpha_{2}^2} \frac{1}{\alpha_{k}}  S_{k}, \qquad \forall k\ge 1.
    \end{equation*} 
\end{lemma}

\begin{proof}
    Proof by induction. 
    \begin{itemize}[leftmargin=*]
        \item $k=1$. 
            Applying \eqref{eq:step_size_rule} and the assumption $s_0 = r_0 \inverLzero$, 
            we obtain:
            \begin{equation*}
                \begin{aligned}
                    \ssz_{1}
                = \min\set{ \frac{\alpha_0}{\alpha_1} \frac{\theta_2}{\theta_3(\theta_3-1)} \ssz_{0},  \,\,  \frac{ \alpha_{2}^2 \alpha_{3}}{\alpha_{3} + \alpha_{2}^2} \frac{1}{\alpha_{1}} \frac{1}{L_{1}}  } 
                = \frac{ \alpha_{2}^2 \alpha_{3}}{\alpha_{3} + \alpha_{2}^2} \frac{1}{\alpha_{1}} \min\set{  \frac{1}{L_0},  \,\,   \frac{1}{L_{1}}  }  
                = \frac{ \alpha_{2}^2 \alpha_{3}}{\alpha_{3} + \alpha_{2}^2} \frac{1}{\alpha_{1}} S_{1}.
                \end{aligned}
            \end{equation*}
        \item $k\ge 2$. 
        We first show $\frac{\alpha_{k+1}}{\frac{\alpha_{k+1}}{\alpha_{k}^2} + 1}$ is increasing for $k\ge2$. 
        By \Cref{lem:theta_properties}, the numerator $\alpha_{k+1}$ is increasing, it is suffices to show the denominator $\frac{\alpha_{k+1}}{\alpha_{k}^2}+1$ is decreasing.     
        Since $\frac{1}{\alpha_{k}}$ is decreasing, it is enough to show $\frac{\alpha_{k+1}}{\alpha_k}$ is decreasing. 
        Observe that for $k\ge2$, we have
        \begin{equation*}
            \frac{\alpha_{k+1}}{\alpha_k} 
            = \frac{1-\frac{1}{\theta_{k+3}}}{1-\frac{1}{\theta_{k+2}}} 
            = 1 + \frac{\frac{1}{\theta_{k+2}} - \frac{1}{\theta_{k+3}} }{1-\frac{1}{\theta_{k+2}}}
            = 1 + \frac{ 1 - \frac{\theta_{k+2}}{\theta_{k+3}} }{\theta_{k+2}-1}.
        \end{equation*}
        Since $\theta_{k+2}-1$ is increasing, it is suffices to show $1 - \frac{\theta_{k+2}}{\theta_{k+3}}$ is decreasing.   
        Observing
        \[
            \frac{\theta_{k+2}}{\theta_{k+3}} = \frac{\theta_{k+2}}{\frac{1}{2} + \sqrt{\frac{1}{4} + \theta_{k+2}^2}} 
            = \pr{ \frac{1}{2\theta_{k+2}} + \sqrt{\frac{1}{4\theta_{k+2}^2} + 1} }^{-1},
        \]
        we see $\frac{\theta_{k+2}}{\theta_{k+3}}$ is increasing since $\theta_k$ is increasing by \Cref{lem:theta_properties}. 
        Therefore, $1 - \frac{\theta_{k+2}}{\theta_{k+3}}$ is decreasing, we conclude  $\frac{\alpha_{k+1}}{\frac{\alpha_{k+1}}{\alpha_{k}^2} + 1}$  is increasing. 
        As a result, we obtain \vspace{-3mm}
        \begin{equation}    
        \label{ineq:core_of_stepsize_lowerbound}
            \frac{\alpha_{2}^2\alpha_{3}}{\alpha_{3} + \alpha_{2}^2}
            = \frac{\alpha_{3}}{\frac{\alpha_{3}}{\alpha_{2}^2} + 1}
            \le 
            \frac{\alpha_{k+1}}{\frac{\alpha_{k+1}}{\alpha_{k}^2} + 1}
            = \frac{\alpha_{k}^2}{\alpha_{k+1} + \alpha_{k}^2} \alpha_{k+1}, \qquad \forall k \ge 2.
        \end{equation}
        Applying the induction hypothesis $\ssz_k \ge \frac{ \alpha_{2}^2 \alpha_{3}}{\alpha_{3} + \alpha_{2}^2} \frac{1}{\alpha_{k}}  S_{k}$ 
        and \eqref{ineq:core_of_stepsize_lowerbound}, we conclude:  
        \begin{equation*}
            \begin{aligned}
                \ssz_{k+1} 
            &= \min\set{ \frac{\alpha_{k}}{\alpha_{k+1}} \ssz_{k},  \,\, \frac{ \alpha_{k}^2 }{ { \alpha_{k+1} + \alpha_{k}^2 } } \frac{1}{L_{k+1}}  } \\
            &\ge \min\set{    \frac{\alpha_{2}^2\alpha_{3}}{\alpha_{3} + \alpha_{2}^2} \frac{1}{\alpha_{k+1}} S_{k},  \,\,  \frac{\alpha_{2}^2\alpha_{3}}{\alpha_{3} + \alpha_{2}^2} \frac{1}{\alpha_{k+1}} \frac{1}{L_{k+1} } }  
            =  \frac{\alpha_{2}^2\alpha_{3}}{\alpha_{3} + \alpha_{2}^2} \frac{1}{\alpha_{k+1}} S_{k+1}.
            \end{aligned}
        \end{equation*}
    \end{itemize}
    This completes the proof.
\end{proof}

\begin{corollaryL}  \label{lemma:stepsize_lowerbound_L_smooth}
    Suppose $f$ is an $L$-smooth convex function. 
    Set $s_0$ as in \Cref{lemma:stepsize_lowerbound} with $L_0$ as in \eqref{eq:initial_guess}. 
    Then 
    \begin{equation*}
        \ssz_k \ge \frac{ \alpha_{2}^2 \alpha_{3}}{\alpha_{3} + \alpha_{2}^2} \frac{1}{\alpha_{k}} \frac{1}{\upperboundL}, \qquad \forall k\ge 1.
    \end{equation*}
\end{corollaryL}

\begin{proof}
    From \Cref{lemma:local_smoothness}, we have $\frac{1}{L_k} \ge \frac{1}{L}$ for $k \ge 1$ and also $\frac{1}{L_0}\ge \frac{1}{L}$ by its definition \eqref{eq:initial_guess}.
    Recalling the definition \eqref{eq:large_S}, we obtain $S_k \ge \frac{1}{\upperboundL}$ for $k \ge 0$.
    We obtain the desired conclusion by \Cref{lemma:stepsize_lowerbound}.
\end{proof}

Finally, we are ready to prove the statements of our main theorem. We state and prove generalized results that directly imply it. 
\begin{proposition} \label{prop:proof_of_thm_1_2}
    Suppose $f$ is a locally smooth convex function. 
    Define $L$ as in \Cref{lemma:lyapunov_analysis_k0} and set $s_0$ as in \Cref{lemma:stepsize_lowerbound}.  
    Denote $\upperboundconstant = 2\big( V_{-1} -\frac{\ssz_{0} \talpha_0 \theta_{2}}{2L} \norm{ \nabla f(x_{0}) }^2 \big)$, which is equivalent to
    \[
        \upperboundconstant 
        = \norm{x_{0} - x_\star}^2 +  r_{0}^2 \talpha_0 \pr{   \oldalpha_0 + \talpha_0 } \theta_{2}^2 \inverLzero \pr{ \inverLzero  - \frac{1}{ r_{0}  \pr{   \oldalpha_0 + \talpha_0 } \theta_{2}} \frac{ 1}{L} } \norm{ \nabla f(x_{0}) }^2,
    \]
    when $s_0=r_0\frac{1}{L_0}$. 
    Suppose $\set{x_k}_{k\ge0}$ is generated by \adanag. 
    Then, the following inequality is true: 
    \begin{equation*}
        \begin{aligned}
            f(x_k) - f_\star  
        \le \frac{\alpha_{3} + \alpha_{2}^2}{ 2\alpha_{2}^2 \alpha_{3}} \frac{1/S_{{k+1}}}{\theta_{k+3} \pr{ \theta_{k+3} - 1 } } \upperboundconstant.
        \end{aligned}
    \end{equation*}
\end{proposition}

\begin{proof}      
    From \Cref{lemma:lyapunov_analysis}, Lemma~\ref{lemma:lyapunov_analysis_k0}, and $s_0=r_0\frac{1}{L_0}$, we know that
    \begin{equation*}
        \ssz_{k+1}\alpha_{k+1} \theta_{k+3} \pr{ \theta_{k+3} - 1 } (f(x_k) - f_\star) \le V_k \le \dots \le V_0 \le V_{-1} -\frac{\ssz_{0} \talpha_0 \theta_{2}}{2L} \norm{ \nabla f(x_{0}) }^2
        = \frac{1}{2}\upperboundconstant.
    \end{equation*} 
    Thus we obtain 
    \begin{equation*}
        f(x_k) - f_\star \le  \frac{1}{2\ssz_{k+1}\alpha_{k+1} \theta_{k+3} \pr{ \theta_{k+3} - 1 } } \upperboundconstant,
    \end{equation*}
    which, combining with  Lemma~\ref{lemma:stepsize_lowerbound}, yields the desired result.  
\end{proof}

\begin{proposition} \label{prop:adanag_gradient_norm}
    Let $f$ be an $L$-smooth convex function. 
    Set $s_0$ as in Corollary~\ref{lemma:stepsize_lowerbound_L_smooth} 
    and suppose $\set{x_k}_{k\ge0}$ is generated by \adanag.  
    Denote  
    $\upperboundconstant$ as in Proposition~\ref{prop:proof_of_thm_1_2}. 
    Then, the following inequality is true:
    \begin{equation}    \label{ineq:z_converges}
        \sum_{i=1}^{k}  \theta_{i+3} \pr{ \theta_{i+3} - 1 }  \norm{ \nabla f(x_i) }^2 
        \le 
         \upperboundL^2 \pr{ \frac{\alpha_{3} + \alpha_{2}^2}{ \alpha_{2}^2 \alpha_{3}}  }^2  
        \upperboundconstant.
    \end{equation}
    As a result, we have 
    \begin{equation*}
        \min_{i\in\set{1,\dots, k}} \norm{\nabla f(x_i)}^2
        \le \frac{\upperboundL^2}{\sum_{i=1}^{k}  \theta_{i+3} \pr{ \theta_{i+3} - 1 }}  \pr{ \frac{\alpha_{3} + \alpha_{2}^2}{ \alpha_{2}^2 \alpha_{3}}  }^2 
        \upperboundconstant
        = \cO\pr{ \frac{\upperboundL^2}{k^3} }.
    \end{equation*}
\end{proposition}

\begin{proof}
    Plugging in \eqref{eq:A_B_AdaNAG} and applying \eqref{eq:theta_motivation} to the result of \Cref{lemma:lyapunov_analysis}, we have
    \begin{equation*}
        \qquad \qquad \qquad V_{k+1} - V_k 
        \le \begin{cases}
            - \frac{1}{2} s_k^2 \alpha_{k}^2 \theta_{k+3} \pr{ \theta_{k+3} - 1 }   \norm{ \nabla f(x_k) }^2 \qquad & \text{ if } s_kL_{k+1} \le 1 \\
            - \frac{1}{2} \frac{\ssz_{k+1}\alpha_{k+1}}{L_{k+1}}  \theta_{k+3} \pr{ \theta_{k+3} - 1 }  \norm{ \nabla f(x_k) }^2 & \text{ if } s_kL_{k+1} \ge 1
        \end{cases}
    \end{equation*}
    for $k\ge0$. Here, we have used $\tilde{\alpha}_0 \approx 0.2706 < 0.4707 \approx \alpha_0$. 
    Now, since $f$ is an $L$-smooth convex function, applying Corollary~\ref{lemma:stepsize_lowerbound_L_smooth}, we have:
    \begin{equation*}
        s_k^2 \alpha_k^2 \ge \pr{ \frac{ \alpha_{2}^2 \alpha_{3}}{\alpha_{3} + \alpha_{2}^2}  }^2 \frac{1}{\upperboundL^2},
        \qquad 
        \frac{\ssz_{k+1}\alpha_{k+1}}{L_{k+1}} \ge  \frac{ \alpha_{2}^2 \alpha_{3}}{\alpha_{3} + \alpha_{2}^2}   \frac{1}{\upperboundL^2}.
    \end{equation*}
    From \Cref{tab:numerical_values}, we know that $\frac{ \alpha_{2}^2 \alpha_{3}}{\alpha_{3} + \alpha_{2}^2} < 1$. 
    Collecting these observations, we have
    \begin{equation*}
        V_{k+1} - V_k
        \le - \frac{1}{2} \pr{ \frac{ \alpha_{2}^2 \alpha_{3}}{\alpha_{3} + \alpha_{2}^2}  }^2 \frac{1}{\upperboundL^2} \theta_{k+3} \pr{ \theta_{k+3} - 1 }  \norm{ \nabla f(x_k) }^2.
    \end{equation*}
    Summing up from \( 0 \) to \( k \), 
    and applying \Cref{lemma:lyapunov_analysis_k0}, 
    we obtain \eqref{ineq:z_converges}. 
    Note that since $\theta_k = \Omega\pr{k}$, we have $\sum_{i=1}^{k}  \theta_{i+3} \pr{ \theta_{i+3} - 1 }=\Omega\pr{k^3}$. 
\end{proof}

\begin{proof} [Proof of \Cref{thm:main_tight}] 
    The proofs immediately follow from Proposition~\ref{prop:proof_of_thm_1_2} and Proposition~\ref{prop:adanag_gradient_norm}. 
    Note, we have $\theta_{k+3}(\theta_{k+3}-1)=\theta_{k+2}^2$ by \eqref{eq:theta} and $\theta_{k+2}^2 \ge \frac{1}{4}(k+4)^2$ by \Cref{lem:theta_properties}.
    From \Cref{lemma:local_smoothness} and $L_0\le L$, 
    we know $1/S_{k+1}\le \upperboundL$. 
    Finally, by plugging in the numerical values from \Cref{tab:numerical_values} with appropriate rounding, we obtain the convergence results in \Cref{thm:main_tight}. 
\end{proof}

\subsection{Discussion: Core points of the proof} \label{sec:proof_core_points}

Designing a new algorithm with theoretical guarantees means finding a new proof.  
Organizing the core elements of a proof often enables us to consider a more general strategy for finding a new proof.  
As mentioned in \cite[\S~2.1]{MalitskyMishchenko2024_adaptive}, bounding the term $\norm{ \nabla f(x_{k+1}) }^2$ that appears in $V_{k+1} - V_k$ is an important point and arguably a tricky part of the proof for adaptive algorithms.  
In \cite{MalitskyMishchenko2024_adaptive}, the authors were able to propose AdGD2 with a refined step size by treating this point in a novel way in Lemma~1.
We consider the following to be the core elements of our proof related to this issue. 
Recalling \eqref{ineq:core_is_quadratic}, we note that the term $Q_k$ defined in \eqref{eq:adanag_lyapunov_diff_quadratic} is the core quantity we need to bound.

\begin{itemize}[leftmargin=*]
    \item [\textit{(i)}] \textit{Definition of $L_{k+1}$.} \\
        Observing \eqref{eq:lyapunov_difference_without_grad-2nd-line}, we see that a negative term  
        $- \frac{\ssz_{k+1} A_k}{2L_{k+1}} \norm{ \nabla f(x_{k}) - \nabla f(x_{k+1}) }^2$ appears, 
        which later subsumed into $Q_k$.  
        Such a term can appear because 
        $f(x_{k+1}) - f(x_k) + \inner{ \nabla f(x_{k+1}) }{ x_k - x_{k+1} }$  
        cancels out the term  
        $\frac{1}{2L_{k+1}} \norm{ \nabla f(x_{k}) - \nabla f(x_{k+1}) }^2$. 
        This is a consequence of the definition of $L_{k+1}$ in \eqref{define-Lk+1}, and also the motivation for the definition as discussed in \Cref{sec:motivation_of_lyapunov_L}.  
        Note that the well-definedness of $L_{k+1}$ is based on Corollary~\ref{cor:L_k_bound}, which crucially exploits both the locally smooth and convex assumptions of $f$. 
        
    \item [\textit{(ii)}] \textit{The positive part of the coefficient of $\norm{ \nabla f(x_{k+1}) }^2$ contains $\ssz_{k+1}^2$.} \\ 
        Taking a closer look at the coefficient of the $\norm{ \nabla f(x_{k+1}) }^2$ term in \eqref{eq:adanag_lyapunov_diff_quadratic}, and recalling the definition of $\rho_k$ in \eqref{eq:adanag_r_definition}, we find that its positive part is $\frac{1}{2}\pr{ B_{k+1} + \alpha_{k+1}^2\theta_{k+3}^2 }\ssz_{k+1}^2$. 
        The important point is that it contains $\ssz_{k+1}^2$, which we can control to be small since it is defined after $\ssz_{k}$ and $L_{k+1}$ are determined.  
        Tracing its origin, $B_{k+1}$ originates directly from the definition of $V_{k+1}$ in \eqref{eq:Lyapunov}, and $\alpha_{k+1}^2\theta_{k+3}^2$ originates from the difference in the $\norm{z_{k+1}-x_\star}^2$ term as observed in \eqref{eq:z_square_diff}. 
        A crucial point is that the index of $z$ is $k+1$ rather than $k$. 
        Note that this shifted index of $z$ in $V_k$ was inherited from the Lyapunov function of OGM \eqref{eq:OGM_lyapunov}.    
\end{itemize}
In the next section, we introduce a novel adaptive gradient descent method obtained by applying the above observations. 
\section{A novel gradient-descent-type adaptive method } \label{sec:adagd}

In this section, we propose a gradient-descent-type (GD-type) adaptive algorithm whose update rule is of the following form: 
\[
x_{k+1} = x_k - \ssz_k \nabla f(x_k),
\]
for some step size schedule $s_k$. That is, there is no momentum term.

\subsection{The AdaGD Algorithm}

Our \gdtype\ adaptive method \adagd\ is described in Algorithm \ref{alg:AdaGD}. 

\begin{algorithm}[ht]
 \caption{\textbf{Ada}ptive \textbf{G}radient \textbf{D}escent (AdaGD)}
 \label{alg:AdaGD} 
 \begin{algorithmic}[1]
 \STATE \textbf{Input:} $x_0 \in \mathbb{R}^d$, 
 $\ssz_{0} >0$, $\set{A_k}_{k\ge0}$,  $\set{B_k}_{k\ge0}$, $A_{-1} = 0$
   \FOR{$k = 0,1,\dots$}
  \STATE 
   \begin{equation*}     
        \begin{aligned}
        x_{k+1} &= x_k -  \ssz_{k} \nabla f(x_k) \\
        L_{k+1} &= -\frac{ \frac{1}{2} \norm{ \nabla f(x_{k+1}) - \nabla f(x_{k}) }^2 }            { f(x_{k+1}) - f(x_{k}) + \inner{\nabla f(x_{k+1})}{x_{k}-x_{k+1}} }
        \end{aligned}
    \end{equation*}
    \begin{equation} \label{eq:AdaGD} 
        \qquad s_{k+1} = 
             \min\set{ \frac{A_{k-1}+1}{A_{k}} s_{k}, \pr{ \frac{A_{k}}{B_{k}} + \frac{B_{k+1} +1}{A_{k}} }^{-1} \frac{1}{L_{k+1}}  } 
    \end{equation}
   \ENDFOR
 \end{algorithmic}
\end{algorithm}

Our main result in this section is given in Theorem \ref{thm:gd}. 
We provide its proof in \Cref{section:gd_analysis}. 

\begin{theorem} \label{thm:gd}
    Let $f$ be a locally smooth convex function. Suppose $\set{x_k}_{k\ge0}$ is generated by \adagd\ (Algorithm \ref{alg:AdaGD}) with $s_0 = r A_0 \inverLzero$ 
    where $L_0$ defined as in \eqref{eq:initial_guess}, 
    and positive
    $\set{A_k}_{k\ge0}$, $\set{B_k}_{k\ge0}$  satisfying 
\begin{equation}   \label{ineq:AB_condition_for_step_size_lowerbound} 
    \begin{aligned}
        \frac{A_{k}+1}{A_{k+1}} &\ge 1, \qquad 
        \pr{ \frac{A_k}{B_k} + \frac{B_{k+1}+1}{A_k} }^{-1} \ge r, \qquad \forall k \ge 0,
    \end{aligned}
\end{equation}
with some $r\in(0,1]$ and
\begin{equation}    \label{ineq:gd_A_B_relation}
    B_{k+2} \le \frac{A_{k+1}^2}{A_{k}+1} - 1, \qquad \forall k \ge 0.
\end{equation}
Let $L$ be a smoothness parameter of $f$ on 
$\bar{B}_{3\bar{R}}(x_\star)\cup\bar{B}_{3\|\tilde{x}_0-x_\star\|}(x_\star)$, 
where $\tilde{x}_0$ is the vector used in \eqref{eq:initial_guess} and $\bar{R} = \norm{x_{0} - x_\star}^2 + ( B_0+1 ) s_0^2 \norm{ \nabla f(x_0) }^2$. 
Define
$\upperboundconstant = \bar{R} - \frac{s_0}{\gdupperboundL}\norm{ \nabla f(x_0) }^2$. 
 
\begin{itemize}
    \item 
    Then the following holds, which indicates a non-ergodic convergence rate when $\lim_{k \to \infty} A_k = \infty$.
    \begin{equation}    \label{ineq:adagd_g_nonergodic}
        f(x_k) - f_\star 
        \le \frac{\gdupperboundL}{2 r A_k} \upperboundconstant 
        =  \cO\pr{ \frac{\gdupperboundL}{A_k} }.
    \end{equation}
    
    \item
    Also, $x_k$ achieves minimum selection convergence rate
    \begin{equation}    \label{ineq:adagd_ergodic}
        \begin{aligned}
            \min_{i \in \set{1,\dots, k}}(f(x_{i}) - f_\star) &\le \frac{1}{2\sum_{i=1}^{k} s_{i}} \upperboundconstant 
            = \cO\pr{ \frac{\gdupperboundL}{k} }
        \end{aligned}
    \end{equation}
    and 
    \begin{equation*}
        \min_{i \in \set{1,\dots, k}} \norm{ \nabla f(x_{i}) }^2 
        \le \frac{\gdupperboundL^2}{r\sum_{i=1}^{k}  (1 + r \min\set{A_i,B_i})} \upperboundconstant 
        = \cO\pr{ \frac{\gdupperboundL^2}{ k + \sum_{i=1}^{k}  A_i} }.
    \end{equation*}
\end{itemize}
\end{theorem}

\vspace{4mm}
Note that from the first condition of \eqref{ineq:AB_condition_for_step_size_lowerbound}, we see that:
\begin{equation*}
    \frac{A_{k}+1}{A_{k+1}} \ge 1
    \quad \iff \quad 
    A_{k+1} - A_k \le 1
    \quad \Longrightarrow \quad
    A_{k} \le A_0 + k.
\end{equation*}
Therefore, the fastest rate we can obtain from \eqref{ineq:adagd_g_nonergodic} is $\cO\pr{1/k}$, which matches the order of non-adaptive gradient descent method. In this context, the choice of $A_k = \Theta\pr{k}$ is worth emphasizing.

\begin{corollaryT} 
\label{cor:adagd_1/k_rate} 
    Suppose $\set{x_k}_{k\ge0}$ is generated by \adagd\  with $A_k = \gamma (k + 1) + 2$ and $B_k = \gamma(k+1)$, where $\gamma \in (0,1]$.  
    Then
    \begin{equation}    \label{ineq:adagd_nonergodic_1/k}
        f(x_k) - f_\star
        \le \frac{\gdupperboundL}{2r} \frac{1}{\gamma (k + 1) + 2} \upperboundconstant 
        =  \cO\pr{ \frac{\gdupperboundL}{k} },
    \end{equation}
    where {$r = \pr{ \frac{A_{0}}{B_{0}} + \frac{B_{1} +1}{A_{0}} }^{-1} = \frac{\gamma(\gamma+2)}{3\gamma^2 + 5\gamma + 4}\le1$}, and $\gdupperboundL$, $\upperboundconstant$ are defined as in \Cref{thm:gd}. 
    Also,
    \begin{equation*}
        \min_{i\in\set{1,\dots, k}} \norm{\nabla f(x_i)}^2 
        \le  \frac{2\gdupperboundL^2}{r} \frac{1}{k (r\gamma(k + 3) + 2)}  \upperboundconstant 
        = \mathcal{O}\pr{ \frac{\gdupperboundL^2}{k^2} }.
    \end{equation*}
\end{corollaryT}

\begin{proof}
    It is sufficient to check the two conditions \eqref{ineq:AB_condition_for_step_size_lowerbound} 
 and \eqref{ineq:gd_A_B_relation}. 
    Since $A_{k+1} - A_k = \gamma \le 1$, the first inequality of \eqref{ineq:AB_condition_for_step_size_lowerbound} holds. 
    Observe that
    \begin{equation*}
        \frac{A_k}{B_k} + \frac{B_{k+1}+1}{A_k} 
        = \frac{B_k+2}{B_k} + \frac{B_{k} + \gamma + 1}{B_k+2} 
        = 2 + \frac{ (\gamma+1)B_k+4 }{B_k(B_k+2)}
        = 2 + \frac{ \gamma+1 + \frac{4}{B_k} }{B_k+2}
    \end{equation*}
    is decreasing for $k\ge0$, since $B_k$ is increasing. Hence, the second inequality of \eqref{ineq:AB_condition_for_step_size_lowerbound} holds with $r = \pr{ \frac{A_{0}}{B_{0}} + \frac{B_{1} +1}{A_{0}} }^{-1}$.
    Next, \eqref{ineq:gd_A_B_relation} follows from   
    \[
    B_{k+2} - \pr{ \frac{A_{k+1}^2}{A_k+1} -1}
    = A_k+2(\gamma-1) -  \frac{(A_{k}+\gamma)^2}{A_k+1} +1
    = -\frac{(1-\gamma)^2}{A_k+1}
    \le0.\] 
    Finally, 
    \[
        \sum_{i=1}^k (1 + r\min\set{A_i,B_i}) = \sum_{i=1}^{k} \pr{ 1 + r \gamma (i+1) } = \frac{1}{2}k (r\gamma(k + 3) + 2).
    \]
    So the desired conclusion follows by applying \Cref{thm:gd}. 
\end{proof}

\vspace{-4mm}
\paragraph{Non-ergodic convergence.} 
Note that both \eqref{ineq:adagd_g_nonergodic} and \eqref{ineq:adagd_nonergodic_1/k} are non-ergodic convergence results. To the best of our knowledge, these are the first non-ergodic results for GD-type (i.e., no momentum) 
adaptive algorithms. In contrast, only ergodic convergence results are available for existing GD-type adaptive algorithms \cite{MalitskyMishchenko2020_adaptive, LatafatThemelisStellaPatrinos2024_adaptive, MalitskyMishchenko2024_adaptive, ZhouMaYang2024_adabb}.

\subsection{Convergence analysis of AdaGD (Proof of Theorem \ref{thm:gd})} 
\label{section:gd_analysis}

We use the Lyapunov function motivated by the discussion in \Cref{sec:proof_core_points}. 
The overall proof structure resembles that of \adanag. 
We first prove $V_{k+1}\le V_k$ for $k\ge-1$.

\begin{proposition} \label{lemma:gd_lyapunov_analysis}
Let $f$ be a locally smooth convex function. 
Suppose $\set{x_k}_{k\ge0}$ is generated by \adagd\ 
with $\set{A_k}_{k\ge0}$ and $\set{B_k}_{k\ge0}$ satisfying the assumptions in \Cref{thm:gd}. 
Define $B_{-1} = B_0 + 1$, $x_{-1} = x_0$, $s_{-1}=s_0$, and the Lyapunov function
\begin{equation*}
    V_k = \ssz_{k+1} A_k (f(x_k) - f_\star)
            + \frac{1}{2} \ssz_{k}^2 B_k \norm{ \nabla f(x_k) }^2 
            + \frac{1}{2} \norm{x_{k+1} - x_\star}^2,
            \qquad \forall k\ge-1.
\end{equation*}
Denote $\tilde{\ssz}_{k} = \min\big\{s_k, s_{k+1}, \frac{1}{L_{k+1}}\big\}$ for $k\ge-1$.  
Then the following inequality holds: 
\begin{equation}\label{prop11}
    \begin{aligned}
        V_{k+1} - V_k 
        &\le (\ssz_{k+1} A_k + s_{k+1} - \ssz_{k+2} A_{k+1}) (f_\star - f(x_{k+1}))  \\ & \phantom{=}
        - \frac{s_{k+1}}{2\gdupperboundL}  \norm{ \nabla f(x_{k+1}) }^2  
        - \frac{\tilde{\ssz}_k^2}{2} \min\set{A_k, B_k} \norm{ \nabla f(x_k) }^2, \qquad \forall k\ge-1.
    \end{aligned}
\end{equation} 
Also, $x_k \in \bar{B}_{\bar{R}}(x_\star)$ holds. Here, both $\gdupperboundL$ and $\bar{R}$ are defined as in \Cref{thm:gd}. 
\end{proposition}

\begin{proof}
    Suppose we have shown 
    \begin{equation}    \label{ineq:gd_ineq_goal}
        \begin{aligned}
        V_{k+1} - V_k 
        &\le (\ssz_{k+1} A_k + s_{k+1} - \ssz_{k+2} A_{k+1}) (f_\star - f(x_{k+1})) \\ &\quad
            + s_{k+1} ( f(x_{k+1}) - f_\star - \inner{ \nabla f(x_{k+1}) }{  x_{k+1} - x_\star  } )  
            - \frac{\tilde{\ssz}_k^2}{2} \min\set{A_k, B_k} \norm{ \nabla f(x_k) }^2
        \end{aligned}
    \end{equation}
    for $k\ge-1$.  
    With \eqref{ineq:gd_ineq_goal}, we have 
    \begin{equation*}
        \frac{1}{2} \norm{ x_k - x_\star }^2 \le V_{k-1} \le \dots \le V_{-1} = \frac{1}{2} \bar{R}, \qquad \forall k\ge0.  
    \end{equation*}
    This implies $x_k \in \bar{B}_{\bar{R}}(x_\star)$. 
    Then we have $f(x_{k+1}) - f_\star - \inner{ \nabla f(x_{k+1}) }{  x_{k+1} - x_\star  } \le - \frac{1}{2\gdupperboundL} \norm{ \nabla f(x_{k+1}) }^2$ from \Cref{lemma:local_smoothness}, 
    therefore, \eqref{ineq:gd_ineq_goal} implies \eqref{prop11}. 
    The proof of \eqref{ineq:gd_ineq_goal} resembles the proof of \Cref{lemma:lyapunov_analysis}, and we place it in \Cref{appendix:remaining_proof_gd}. 
\end{proof}

Leveraging the boundedness of the iterate $x_k \in \bar{B}_{\bar{R}}(x_\star)$, we can achieve a lower bound for $s_k$.
We provide it in \Cref{lemma:gd_stepsize_lowerbound},
and the proof in \Cref{appendix:proof_of_gd_stepsize_lowerbound}. 

\begin{lemma} \label{lemma:gd_stepsize_lowerbound}
    Let $f$ be a locally smooth convex function.  
    Suppose $\set{s_k}_{k\ge0}$ is generated by \adagd\ with $\set{A_k}_{k\ge0}$ and $\set{B_k}_{k\ge0}$ satisfying the assumptions in \Cref{thm:gd}. 
    For $s_0$ and $\gdupperboundL$ defined in \Cref{thm:gd}, we have:
    \begin{equation}\label{lem12-2}
        \ssz_k \geq \frac{r}{\gdupperboundL}, \qquad \forall k\ge 1.
    \end{equation} 
\end{lemma}

\vspace{2mm} 
We are now ready to prove \Cref{thm:gd}. 
\vspace{-1mm} 
\begin{proof} [Proof of \Cref{thm:gd}]
    For notation simplicity, denote $\CC_k = \frac{1}{2} \tilde{\ssz}_k^2 \min\set{A_k, B_k}$. 
    First, summing up \eqref{prop11} from $0$ to $k-1$ we have 
    \[
        V_k + \sum_{i=0}^{k-1} \pr{ ( \ssz_{i+1} A_i - \ssz_{i+2} A_{i+1} + s_{i+1}) (f(x_{i+1}) - f_\star) + \frac{s_{i+1}}{2\gdupperboundL} \norm{ \nabla f(x_{i+1}) }^2
        + \CC_{i} \norm{ \nabla f(x_i) }^2    
        }
        \le V_0.
    \] {
    Plugging in $k=-1$ to \eqref{prop11}, we have
    \begin{equation*}    
        \begin{aligned}
        V_{0} - V_{-1}
        &\le (\ssz_{0} A_{-1} + s_{0} - \ssz_{1} A_{0}) (f_\star - f(x_{0})) \\ &\quad
            - \frac{s_{0}}{2\gdupperboundL}  \norm{ \nabla f(x_{0}) }^2 
            - \frac{\tilde{\ssz}_{-1}^2}{2} \min\set{A_{-1}, B_{-1}} \norm{ \nabla f(x_0) }^2 
        \le - \frac{s_{0}}{2\gdupperboundL}  \norm{ \nabla f(x_{0}) }^2.
        \end{aligned}
    \end{equation*}
    The second inequality follows from the fact $\ssz_{1} \le \frac{1}{A_0} \ssz_{0}$ which follows from \eqref{eq:AdaGD},  and $\min\set{A_{-1}, B_{-1}}=0$  since $A_{-1}=0$. 
    Therefore, we obtain $V_0 \le V_{-1} - \frac{s_0}{2\gdupperboundL} \norm{ \nabla f(x_0) }^2 = \frac{1}{2} \upperboundconstant$.} 
    Combining with $ s_{k+1}A_k(f(x_k) - f_\star) \leq V_k$, 
    we have 
    \begin{equation} \label{ineq:gd_core_ineq}
        \begin{aligned}
        s_{k+1}A_k(f(x_k) - f_\star) &+
        \sum_{i=1}^{k} ( \ssz_{i} A_{i-1} - \ssz_{i+1} A_{i} + s_{i}) (f(x_{i}) - f_\star)   \\
        &+ \sum_{i=1}^{k} \frac{s_{i}}{2\gdupperboundL} \norm{ \nabla f(x_{i}) }^2   
        + {\sum_{i=0}^{k-1}} \CC_{i} \norm{ \nabla f(x_i) }^2  
        \le 
        \frac{1}{2} \upperboundconstant.
    \end{aligned}
    \end{equation} 
    
    Since the summations in \eqref{ineq:gd_core_ineq} are nonnegative, from \eqref{ineq:gd_core_ineq} we have $s_{k+1}A_k(f(x_k) - f_\star) \le \frac{1}{2} \upperboundconstant$, 
    dividing both sides by $s_{k+1}A_k$ and applying \Cref{lemma:gd_stepsize_lowerbound}, we obtain 
    \begin{equation*}
        \begin{aligned}
            f(x_k) - f_\star  
            &\le \frac{\gdupperboundL}{2r A_k} \upperboundconstant  
            = \cO\pr{ \frac{\gdupperboundL}{A_k} }.
        \end{aligned}
    \end{equation*} 
        
    Next, considering only the function value terms on the left-hand side of \eqref{ineq:gd_core_ineq}, we obtain 
    \begin{equation}    \label{ineq:core_minimum}
        s_{k+1}A_k(f(x_k) - f_\star) +
           \sum_{i=1}^{k} ( \ssz_{i} A_{i-1} - \ssz_{i+1} A_{i} + s_{i}) (f(x_{i}) - f_\star)
            \le \frac{1}{2} \upperboundconstant.  
    \end{equation}
    Observing the sum of the coefficients of the $f(x_i) - f_\star$ terms, we see that
    \begin{equation*}
        s_{k+1}A_k + \sum_{i=1}^{k} ( \ssz_{i} A_{i-1} - \ssz_{i+1} A_{i} + s_{i}) 
        = s_1 A_0 + \sum_{i=1}^{k} s_{i} 
        \ge \sum_{i=1}^{k} s_{i} 
        \ge \frac{r}{\gdupperboundL} k,
    \end{equation*}
    where the second inequality is from \Cref{lemma:gd_stepsize_lowerbound}. 
    From \eqref{eq:AdaGD}, we know $\ssz_{i+1} A_i - \ssz_{i+2} A_{i+1} + s_{i+1}\ge0$.  
    Gathering these observations, from \eqref{ineq:core_minimum} we conclude 
    \begin{equation*}
        \begin{aligned}
            \min_{i \in \set{1,\dots, k}} (f(x_{i}) - f_\star) &\le \frac{\gdupperboundL}{2rk} \upperboundconstant 
            = \cO\pr{ \frac{\gdupperboundL}{k} }.
        \end{aligned}
    \end{equation*}

    Lastly, from \Cref{lemma:gd_stepsize_lowerbound} we know $s_k, \tilde{s}_k \ge {r}/{\gdupperboundL}$. 
   Focusing on the squared gradient norm terms in \eqref{ineq:gd_core_ineq}, 
   and recalling the definition $\CC_k = \frac{1}{2} \tilde{\ssz}_k^2 \min\set{A_k, B_k}$, 
   we obtain
    \begin{equation}    \label{ineq:gd_gradient_summable}
         \frac{r}{{2\gdupperboundL}^2}  \sum_{i=1}^{k} (1 + r \min\set{A_i,B_i}) \norm{ \nabla f(x_{i}) }^2
         \le \sum_{i=1}^{k+1} \frac{s_{i}}{2\gdupperboundL} \norm{ \nabla f(x_{i}) }^2 + \sum_{i=0}^{k} \CC_i \norm{ \nabla f(x_{i}) }^2
        \le  \frac{1}{2} \upperboundconstant . 
    \end{equation} 
    Therefore, we conclude  
    \[
         \min_{i \in \set{1,\dots, k}} \norm{ \nabla f(x_{i}) }^2 
         \le \frac{\gdupperboundL^2}{r\sum_{i=1}^{k} \pr{1 + r \min\set{A_i,B_i}}}   
         \upperboundconstant. 
    \] 
    Lastly, from the second inequality of \eqref{ineq:AB_condition_for_step_size_lowerbound}, it follows that $\limsup_{k\to\infty} \frac{A_k}{B_k} < \infty$. 
    Therefore, we obtain $A_k = \cO\pr{\min\set{A_k,B_k}}$, and thus $\frac{1}{\sum_{i=1}^{k} \pr{1 + r \min\set{A_i,B_i}}} \le \frac{1}{k+r \sum_{i=1}^{k} \min\set{A_i,B_i}} = \cO\pr{ \frac{1}{k+\sum_{i=1}^{k} A_i} }$. 
\end{proof}

\vspace{4mm}
Previous Lyapunov analysis also implies the pointwise convergence. 
\vspace{-1mm}

\begin{theorem} 
    Let $f$ be a locally smooth convex function. 
    Suppose $\set{s_k}_{k\ge0}$ is generated by \adagd\ with $\set{A_k}_{k\ge0}$ and $\set{B_k}_{k\ge0}$ satisfying the assumptions in \Cref{thm:gd}. Then, for some optimal point $\bar{x}_\star$, 
    \begin{equation*}
        \lim_{k\to\infty} x_k = \bar{x}_\star.
    \end{equation*}
\end{theorem} 

\begin{proof} 
    Recall that $x_k \in \bar{B}_{\bar{R}}(x_\star)$ holds by Proposition~\ref{lemma:gd_lyapunov_analysis}, 
    we know $\set{x_k}_{k\ge0}$ is a bounded sequence. 
    From \eqref{ineq:gd_gradient_summable}, we have $\sum_{k=1}^{\infty} \norm{\nabla f(x_k)}^2 \le \frac{\gdupperboundL^2}{r} \upperboundconstant < \infty$, which implies that $\lim_{k \to \infty} \norm{\nabla f(x_k)}^2 = 0$. Therefore, all limit points converge to an optimal point. 
    Now, we leverage \cite[Lemma~2]{MalitskyMishchenko2020_adaptive} by setting  
    $a_k = 2\ssz_{k} A_{k-1} (f(x_{k-1}) - f_\star) + \ssz_{k-1}^2 B_{k-1} \norm{ \nabla f(x_{k-1}) }^2$, and conclude the desired result. 
    For convenience, we restate the lemma below.
\end{proof}

\begin{lemmaX}  \emph{(\cite[Lemma~2]{MalitskyMishchenko2020_adaptive})}  
Let $\set{x_k}_{k\ge0}$ and $\set{a_k}_{k\ge0}$ be two sequences in $\bbR^d$ and $\bbR_{+}$ respectively. Suppose that $\set{x_k}_{k\ge0}$ is bounded, its cluster points belong to $\cX \subset \bbR^d$ and it also holds that
\[
    \| x^{k+1} - \bar{x} \|^2 + a_{k+1} \le \| x^{k} - \bar{x} \|^2 + a_{k}, \qquad \forall \bar{x} \in \cX.
\]
Then $\set{x_k}_{k\ge0}$ converges to some element in $\cX$.     
\end{lemmaX}

\subsection{Discussion: Trade-off between the step size and the theoretical guarantee} 
\label{sec:trade_off_theory_practical}

\Cref{thm:gd} covers a family of algorithms. 
We provide some representative examples of \adagd\  and compare their behaviors.  
 For notational simplicity, we rewrite \eqref{eq:AdaGD} as  
\[
    s_{k+1} = \min\set{ r_k^s s_{k}, \,\, r_k^L \frac{1}{L_{k+1}}  }, \qquad \mbox{ where } \quad
    r_k^s = \frac{A_{k-1}+1}{A_{k}}, \quad r_k^L = \pr{ \frac{A_{k}}{B_{k}} + \frac{B_{k+1} +1}{A_{k}} }^{-1}.
\]
Note that when $s_{k} \ll \frac{1}{L_{k+1}}$, we have $s_{k+1} = r_k^s s_{k}$.  
Thus, $r_k^s$ represents the rate at which the adaptive step size catches up to the local smoothness parameter $\frac{1}{L_{k+1}}$.  
On the other hand, when $\frac{1}{L_{k+1}} \ll s_{k}$, we obtain $s_{k+1} = r_k^L \frac{1}{L_{k+1}}$.  
In this case, we can take a larger step size when $r_k^L$ is large. 
The following table gives examples of \adagd\ in which $A_k$ is $\Theta\pr{ k }$, $\Theta( \sqrt{k} )$ and $\Theta\pr{ 1 }$, respectively.

\renewcommand{\arraystretch}{1.3}
\begin{table}[H]
    \centering
    \begin{tabular}{c|cc|cc|c}
    \toprule
    &  \dseq{\boldsymbol{A_k}} &  \dseq{\boldsymbol{B_k}}  &  \dseq{\boldsymbol{\ratio_k^s}} &  \dseq{\boldsymbol{\ratio_k^L}} &  \dseq{\boldsymbol{f(x_k)-f_\star}} \\ 
    \midrule\midrule
    \textbf{\adagdgOne} & \dseq{\frac{1}{2}(k+5)} & \dseq{\frac{1}{2}(k+1)} & \dseq{\frac{k+6}{k+5}} & \dseq{\frac{1}{2} + { \textstyle\Theta\pr{{1}/{k}} } }  
    & \dseq{\cO\pr{ {1}/{k}}} \dseq{\vphantom{\frac{1}{\frac{1}{k}}}}  \\
    \hline
    \textbf{\adagdgSqrt} & \dseq{2\sqrt{k+4}} & \dseq{ 2\sqrt{k+2} - 2} & \dseq{\frac{2\sqrt{k+3}+1}{2\sqrt{k+4}}} 
    & \dseq{ \frac{1}{2} + {\textstyle \Theta\big({1}/\sqrt{k} } }\big)
    & \dseq{\cO \big( {1}/{\sqrt{k}}\big)}  \dseq{\vphantom{\frac{\frac{1}{\frac{1}{k}}}{\frac{1}{k}}}} \\
    \hline
    \textbf{\adagdgZero} & \dseq{3} & \dseq{\frac{5}{4}} & \dseq{\frac{4}{3}} & \dseq{\frac{20}{63}} & \dseq{o(1)} \dseq{\vphantom{\frac{\frac{1}{k}}{1}}} \\
    \bottomrule
    \end{tabular} 
\end{table}
\renewcommand{\arraystretch}{1}

\vspace{-4mm}

\begin{itemize}[leftmargin=*]
    \item [(i)]  \adagdgOne: $A_k, B_k = \Theta\pr{k}$. \\
        For this case, we have $\lim_{k\to\infty} r_k^s = 1$. Thus, the growth rate of the step size slows down as the iteration proceeds.  
        This does not mean there is an upper bound on the growth induced by $r_k^s$, since $\prod_{k=N}^{\infty} \frac{k+1}{k} = \infty$, but its growth becomes very slow compared to other cases.  
        On the other hand, $\lim_{k\to\infty} r_k^{L} = \frac{1}{2}$, which is larger than the value in \adagdgZero, where $r_k^{L} = \frac{20}{63}$.
     
    \item [(ii)]   \adagdgSqrt: $A_k, B_k = \Theta\pr{k^q}$ with $0<q<1$. \\ 
        It still holds that $\lim_{k\to\infty} r_k^s=1$ and $\lim_{k\to\infty} r_k^{L} = \frac{1}{2}$.         
        However, comparing $r_k^s$ for \adagdgSqrt\ and \adagdgOne, we see that 
        $1 + \Theta( {1}/{\sqrt{k}})    = \frac{2\sqrt{k+3}+1}{2\sqrt{k+4}} > \frac{k+6}{k+5} = 1 + \Theta \pr{ {1}/{k} }$. 
        In this context, \adagdgSqrt\ has a larger $r_k^s$ compared to \adagdgOne, while the guaranteed convergence rate is slower. 
     
    \item  [(iii)]  \adagdgZero: $A_k, B_k = \Theta\pr{1}$. \\
        For this case, $\lim_{k\to\infty} r_k^s > 1$ since it remains constant. Thus, its growth is larger compared to other cases.  
        However, we do not have a non-ergodic convergence rate for this case.  
        Note that, though \adagdgZero\ still has an ergodic $\cO\pr{1/k}$ rate by \eqref{ineq:adagd_ergodic}, like other GD-type adaptive algorithms in the prior works \cite{MalitskyMishchenko2020_adaptive, LatafatThemelisStellaPatrinos2024_adaptive, MalitskyMishchenko2024_adaptive, ZhouMaYang2024_adabb}.
\end{itemize}
In many practical scenarios, $\frac{1}{L_{k+1}}$ becomes large near the optimum, so a larger $r_k^s$ often allows a larger step size.  
However, from the above context, \emph{a larger step size is not free.} It requires a trade-off between the theoretical guarantee on the convergence rate.\footnote{To clarify, at least for our algorithm family \adagd\  and under our proof argument.}  
Note that the choice $\gamma = 1$ in Corollary~\ref{cor:adagd_1/k_rate} may guarantee the fastest convergence rate, but it is practically not useful since $r_k^s = \frac{(k+2)+1}{k+3} = 1$, implying that the step size $s_k$ is nonincreasing.

Recall from \eqref{eq:step_size_rule}, for \adanag\ we see $r_k^s = \frac{\alpha_{k}}{\alpha_{k+1}} \le 1$ by \Cref{lem:theta_properties}. 
Above discussion motivates us to consider different choice of $\theta_k$ of \adanag\ that makes $r_k^s>1$. 
Therefore, we consider such generalization of \adanag\ in the next section.

\section{Generalized AdaNAG} \label{sec:adanag-g}

The parameter selection of \adanag\ was motivated by the goal of tightening the rate of the theoretical guarantee.  
However, the discussion in \Cref{sec:trade_off_theory_practical}  
suggests that we should consider a broader range of parameter choices to enlarge the step size.  
Previous studies \cite{SuBoydCandes2016_differential, ChambolleDossal2015_convergence, AttouchChbaniPeypouquetRedont2018_fast} considered generalized versions of NAG with $\theta_k = \frac{k+p}{p}$ for $p \geq 2$, which achieve a slower rate compared to $p = 2$ but maintain an $\mathcal{O} \pr{{1}/{k^2}}$ convergence rate.  
Note that \adanag\ can be thought of as the $p = 2$ case, since the simplified variant that uses \eqref{eq:rational_theta} exhibits similar behavior, as shown in \Cref{sec:adanag-s}. 
As we can benefit from a larger step size when considering adaptive methods, we do not need to restrict ourselves to the choices of $\theta_k$ that achieve the tightest convergence guarantee.  

Note, the core properties of the parameters that make the proof of \Cref{lemma:lyapunov_analysis} work were \eqref{ineq:AdaNAG_generalized_stepsize} and the definition $A_k = \alpha_{k+1}\theta_{k+3}(\theta_{k+3}-1)$. 
Motivated by this observation, 
we now introduce the generalized version of AdaNAG, \adanagG, in Algorithm~\ref{alg:AdaNAG-G}.  
The convergence results 
are summarized in Theorem~\ref{thm:adanag_g}. 
The proof resembles the proof of \Cref{thm:main_tight}, and is provided in \Cref{appendix:proof_of_adanag_g}.

\begin{algorithm}[ht] 
     \caption{\textbf{G}eneralized AdaNAG (AdaNAG-G)}
     \label{alg:AdaNAG-G} 
     \begin{algorithmic}[1]
     \STATE \textbf{Input:} $x^0=z^0 \in \mathbb{R}^d$, 
     $\set{\tau_k}_{k\ge0}$, $\set{\alpha_k}_{k\ge0}$, 
     $\ssz_{0} >0$, $B_0>0$, 
     $A_{-1}=0$. 
     \STATE \textbf{Define} 
     \[
        \begin{aligned}
            A_k = \alpha_{k+1} \tau_{k+1}(\tau_{k+1}-1), \qquad  
            B_k = \begin{cases}
            B_0 & \text{ if } k=0 \\
           \alpha_{k}^2\tau_{k}^2 \Big( \frac{(\tau_{k}-1)^2}{\alpha_{k-1}\tau_{k-1}^2} - 1 \Big) & \text{ if } k\ge1
       \end{cases}
        \end{aligned}
     \]
       \FOR{$k = 0,1,\dots$} 
       \STATE   
        \begin{equation*}
            \begin{aligned}
                y_{k+1} &= x_k -  \ssz_{k} \nabla f(x_k)  \\
            z_{k+1} &= z_k - \ssz_{k} \talpha_{k} \tau_{k} \nabla f(x_k)  \\
            x_{k+1} &= \pr{ 1 - \frac{1}{\tau_{k+1}} } y_{k+1} + \frac{1}{\tau_{k+1}} z_{k+1}   \\
            L_{k+1} &= -\frac{ \frac{1}{2} \norm{ \nabla f(x_{k+1}) - \nabla f(x_{k}) }^2 }
                { f(x_{k+1}) - f(x_{k}) + \inner{\nabla f(x_{k+1})}{x_{k}-x_{k+1}} }  \qquad\qquad\qquad\qquad
            \end{aligned}
        \end{equation*}  
        \begin{equation}    \label{eq:AdaNAG-G_stepsize} 
                \ssz_{k+1} = 
                \min\set{ \frac{A_{k-1}+\alpha_{k}\tau_{k}}{A_{k}} s_{k}, \,
         \pr{ \frac{A_{k}}{B_{k}} + \frac{B_{k+1} + \alpha_{k+1}^2\tau_{k+1}^2}{A_{k}} }^{-1} \frac{1}{L_{k+1}} }
            \end{equation}
       \ENDFOR
     \end{algorithmic}
\end{algorithm}

\begin{theorem} 
\label{thm:adanag_g} 
Suppose  $\set{\tau_k}_{k\ge0}$, $\set{\alpha_k}_{k\ge0}$ satisfy 
\begin{equation}   \label{ineq:adanag_AB_condition_for_step_size_lowerbound}
    \begin{aligned}
        0<\alpha_k \le1, 
        \quad 
        \frac{\tau_{k}^2}{\tau_{k+1}(\tau_{k+1}-1)} &\ge 1,  
        \quad 
        \pr{ \frac{A_{k}}{B_{k}} + \frac{B_{k+1} + \alpha_{k+1}^2\tau_{k+1}^2}{A_{k}} }^{-1}
        \ge \frac{ \rbound }{\alpha_{k+1}}, \qquad \forall k\ge0 
    \end{aligned}
\end{equation}
with some $\rbound>0$. 
Set $s_0 = \frac{A_{0}}{\alpha_{0}\tau_{0}} \frac{\rbound}{\alpha_1}  \frac{1}{L_0}$ with $L_0$ defined as in \eqref{eq:initial_guess}. 
Suppose $f$ is an $L$-smooth convex function. 
Denote  
$\upperboundconstant = {\norm{x_{0} - x_\star}^2 +  \pr{ \ssz_{0}^2 \pr{ B_0 + \alpha_0^2\theta_0^2 } -\frac{\ssz_{0} \talpha_0 \theta_{2}}{2L} } \norm{ \nabla f(x_{0}) }^2 }$. 
Then the sequence $\{x_k\}_{k\ge0}$ generated by \adanagG\ satisfies
$s_{k+1}\alpha_{k+1} \ge \frac{\rbound}{\upperboundL}$ for $k\ge0$ and 
\begin{equation*}
    f(x_k) - f_\star 
    \le  \frac{ 1 }{2s_{k+1}A_k } 
    \upperboundconstant
    = \cO\pr{ \frac{ \upperboundL }{\alpha_k\tau_k^2} },
    \qquad 
    \min_{i\in\set{1,\dots, k}} \norm{\nabla f(x_i)}^2
    \le  
    \cO\pr{ \frac{ \upperboundL^2 }{ \sum_{i=1}^{k} \alpha_i\tau_i^2} }.
\end{equation*}
\end{theorem}

The important point of \adanagG\ compared to \adanag\ is,  
it covers the choice of $\tau_k$ that makes $s_k$ possible to increase. 
The worst possible scenario for \adanag\ is that the initial step size $s_0$ is chosen too small, and thus the first option of \eqref{eq:step_size_rule} is activated at every iteration. Then the information about the local smoothness parameter $L_{k+1}$ is never used. 
However, this cannot occur for \adanagG\ with a proper choice of $\tau_k$. 
Observing the first option of \eqref{eq:AdaNAG-G_stepsize}, we see
\begin{equation*}
    \prod_{k=1}^{N} \frac{A_{k-1}+\alpha_{k}\tau_{k}}{A_{k}}
    = \prod_{k=1}^N \frac{\alpha_{k}}{\alpha_{k+1}} \frac{\tau_k^2}{\tau_{k+1}(\tau_{k+1}-1)}  
    \ge \alpha_1 \prod_{k=1}^N \frac{\tau_k^2}{\tau_{k+1}(\tau_{k+1}-1)},
\end{equation*}
where the inequality comes from the fact $\alpha_N\le1$. 
When $\tau_k = \frac{k+p}{p}$ and $p>2$, we can verify
\begin{equation}    \label{eq:p>2_growth_infinity}
    \prod_{k=1}^\infty \frac{\tau_k^2}{\tau_{k+1}(\tau_{k+1}-1)}
    = \prod_{k=1}^\infty \frac{(k+p+1)^2}{(k+2)(k+p+2)}  
    \ge \prod_{k=1}^\infty \pr{ 1 + \frac{p-2}{k+2} }
    \ge 1 + (p-2) \sum_{k=3}^{\infty} \frac{1}{k}
    = \infty.
\end{equation}
Therefore, if only the first option of \eqref{eq:AdaNAG-G_stepsize} is always activated, the step size would diverge to $\infty$. Thus the second option would be activated before that happens. 
On the other hand, the lower bound $\rbound$ appearing in \Cref{thm:adanag_g}, 
does not capture this growth and misrepresents the asymptotic behavior of the step size $\ssz_k$. 
This motivates us to consider the following lemma. 
The proof can be done by formalizing the discussion above, and we provide it in \Cref{appendix:proof_of_asymptotic_lower_bound_step_size}.

\begin{lemma}   \label{lemma:asymptotic_lower_bound_step_size} 
    Suppose the limit $\bar{r} = \lim_{k\to\infty} \Big( \frac{A_{k}}{B_{k}} + \frac{B_{k+1} + \alpha_{k+1}^2\tau_{k+1}^2}{A_{k}} \Big)^{-1}$ exists.   
    Assume that there exists $N>0$ such that $\frac{A_{k-1}+\alpha_{k}\tau_{k}}{A_{k}} \ge 1$ holds for all $k\ge N$. 
    Also, assume that $\prod_{k=1}^{\infty} \frac{\tau_{k}^2}{\tau_{k+1}(\tau_{k+1}-1)} = \infty$.  
    Suppose $\set{s_k}_{k \ge 0}$ is generated by \adanagG\ under the same assumptions as in \Cref{thm:adanag_g}. Then,
    \begin{equation*}
        \liminf_{k\to\infty} s_{k+1} 
        \ge \frac{\bar{r}}{L}.
    \end{equation*}
\end{lemma}

Intuitively speaking, \Cref{lemma:asymptotic_lower_bound_step_size} states that if the growth rate $\frac{\tau_k^2}{\tau_{k+1}(\tau_{k+1}-1)}$ is large enough,  
the step size eventually catches up to a certain ratio of the smoothness parameter. 
However, similar to \adagd, the growth rate has a trade-off between the convergence guarantee. 
We have empirically identified two parameter choices that balance both practical performance and theoretical guarantees, which constitute the final main results of this paper. 
The proofs can be completed by verifying that the assumptions of \Cref{thm:adanag_g} and \Cref{lemma:asymptotic_lower_bound_step_size} hold, and we defer them to \Cref{appendix:adanag_12_proof} and \Cref{appendix:adanag_half_proof}. 
In the proof, we consider the generalized parameter choice $\tau_k = \frac{(k+2)+p}{p}$ with $p > 2$ for Corollary~\ref{cor:adanag_12}, following the spirit of non-adaptive NAG variants~\cite{SuBoydCandes2016_differential, ChambolleDossal2015_convergence, AttouchChbaniPeypouquetRedont2018_fast}.

\begin{corollaryT}  
\label{cor:adanag_12} \hypertarget{adanag_12}{} 
    Let $\tau_k = \frac{(k+2)+12}{12}$, $\alpha_k = \frac{1}{2} \frac{(\tau_{k+1}-1)^2}{\tau_{k}^2}$, $B_0 = \alpha_{0}^2\tau_{0}^2 \pr{ \frac{(\tau_{0}-1)^2}{\alpha_{-1}\tau_{-1}^2} - 1 }$, set $s_0$, $\upperboundL$ as  \Cref{thm:adanag_g}.  
    Suppose $f$ is an $L$-smooth convex function. 
    Then for $\set{x_k}_{k\ge0}$ generated by \adanagG, we have:  
    \begin{equation*}
        f(x_k) - f_\star  
        \le \frac{1}{s_{k+1}} \frac{144 (k+15)}{(k+3) (k+4)^2}  
        \upperboundconstant 
        = \cO\pr{ \frac{ \upperboundL }{k^2} },
        \qquad 
        \min_{i\in\set{1,\dots, k}} \norm{\nabla f(x_i)}^2 = \cO\pr{ \frac{\upperboundL^2}{k^3} }.
    \end{equation*}
    Here, $s_k \ge \frac{1}{250} \frac{1}{\upperboundL}$ and
    $\liminf_{k\to\infty} s_k \ge \frac{1}{3} \frac{1}{\upperboundL}$.  
    We name the algorithm as {\bf \adanagpracticename}.  
\end{corollaryT}

\begin{corollaryT} 
\label{cor:adanag_half} \hypertarget{adanag_half}{}
    Let $\tau_k = 2 \sqrt{k+3}$, $\alpha_k = \frac{1}{2}$, $B_0 = \alpha_{0}^2\tau_{0}^2 \pr{ \frac{(\tau_{0}-1)^2}{\alpha_{-1}\tau_{-1}^2} - 1 }$ and set $s_0$, $\upperboundL$ as in \Cref{thm:adanag_g}. 
    Suppose $f$ is an $L$-smooth convex function. 
    Then for $\set{x_k}_{k\ge0}$ generated by \adanagG, we have:
    \begin{equation*}
        f(x_k) - f_\star 
        \le \frac{1}{s_{k+1}} \frac{1}{4\pr{k+4 - \frac{\sqrt{k+4}}{2}}} 
        \upperboundconstant 
        = \cO\pr{ \frac{ \upperboundL }{k} }, 
        \qquad 
        \min_{i\in\set{1,\dots, k}} \norm{\nabla f(x_i)}^2 = \cO\pr{ \frac{\upperboundL^2}{k^2} }.
    \end{equation*}
    Here,  $s_k \ge \frac{1}{5} \frac{1}{\upperboundL}$ and
    $\liminf_{k\to\infty} s_k \ge \frac{1}{3} \frac{1}{\upperboundL}$. 
    We name the algorithm as {\bf \adanagsqrtname}. 
\end{corollaryT}

\adanagpractice\ has a larger constant in the guaranteed convergence upper bound compared to \adanag, but it still achieves an accelerated $\mathcal{O} \pr{ {1}/{k^2} }$ rate.  
Although \adanagsqrt\ does not achieve the accelerated convergence rate, it still attains a non-ergodic $\mathcal{O} \pr{ {1}/{k} }$ rate. 
The reason for introducing both methods is that there were specific experimental cases where each method performed well.  
From our empirical observations, these algorithms performed well compared to prior methods in various practical scenarios, which we present in the next section.

\section{Numerical Experiments}     \label{sec:experiments}

In this section, we apply our methods to two problems: \textit{logistic regression} and \textit{least squares problem}. The objective function \( f \) is \( L \)-smooth and convex for both problems. We primarily focus on two algorithms, \adanagpractice\ and \adanagsqrt,  
which were considered in Corollary~\ref{cor:adanag_12} and  Corollary~\ref{cor:adanag_half}. 
Among the parameter choices that satisfy \Cref{thm:adanag_g} and that we have tried, these two algorithms performed the best overall.

We compare our algorithms with recently developed adaptive algorithms, mostly focusing on the accelerated adaptive algorithm AC-FGM \cite{LiLan2024_simple}.  
Specifically, we compare with AdaBB \cite{ZhouMaYang2024_adabb}, AdGD \cite{MalitskyMishchenko2024_adaptive}, AdaPG$^{q,r}$ \cite{OikonomidisLaudeLatafatThemelisPatrinos2024_adaptive}, 
and the original Nesterov accelerated gradient method \cite{Nesterov1983_method}, which is a non-adaptive method.  
For AC-FGM, we use the parameters from~\cite[Corollary~2]{LiLan2024_simple}, with the $\alpha$ chosen in their experiments. 
We note that AC-FGM guarantees an $\cO(1/k^2)$ rate for $\alpha > 0$, but only an $\cO(1/k)$ ergodic rate when $\alpha = 0$. 
Overall, our methods performed the best in the scenarios we considered.

\renewcommand{\arraystretch}{1.3}
\begin{table}[H]
    \centering
    \begin{tabular}{c|c|cc|ccc}
    \toprule
    \multirow{2}{*}{\textbf{Algorithm}} & 
    \multirow{2}{*}{\adanag} & 
    \multicolumn{2}{c|}{\textbf{AdaNAG-G}} & \multicolumn{3}{c}{\textbf{AdaGD}} \\
    \cline{3-7}
    & & \small\adanagpractice & \small\adanagsqrt & \small\adagdgOne\ & \small\adagdgSqrt & \small\adagdgZero\ \\  
    \midrule
    $f(x_k) - f_\star$ & $\cO\pr{{1}/{k^2}}$ & $\cO\pr{{1}/{k^2}}$ & $\cO\pr{{1}/{k}}$ & $\cO\pr{{1}/{k}}$ & $\cO\big({{1}/{\sqrt{k}}}\big)$& o(1) \\
    \hline \small \textbf{Key Properties} &\small Theoretically tighter &  \multicolumn{2}{c|}{\small Practically useful} & \multicolumn{3}{c}{\small Without momentum} \\
    \bottomrule
    \end{tabular}
    \caption{Summary of our algorithms. Note that the convergence rates of $f(x_k) - f_\star$ are non-ergodic, and both \adagdgSqrt\ and \adagdgZero\ still achieve an $\cO(1/k)$ ergodic convergence rate.}
    \label{tab:our_algorithms}
\end{table}
\renewcommand{\arraystretch}{1}
\vspace{-2mm}

Our codes were written in Python 3.13.1 and leveraged the implementations provided by \cite{MalitskyMishchenko2020_adaptive} and \cite{LiLan2024_simple}. All numerical experiments were conducted on a personal computer with an Intel Core i5-8265U processor, Intel UHD Graphics 620, and 16GB of memory.
Details of the dataset and the parameters used for each algorithm in the experiments are provided in the respective subsections. 
For convenience, we summarize our algorithms in Table \ref{tab:our_algorithms}.

\subsection{Logistic regression}

In this subsection, we present several experimental cases of the logistic regression problem 
\begin{equation*}
    \min_{x \in \bbR^n} f(x)
    = -\frac{1}{m} \sum_{i=1}^{m} \pr{ y_i \log(\sigma(a_i^{\intercal}x))  + (1-y_i) \log(1-\sigma(a_i^\intercal x)) } + \frac{\gamma}{2} \norm{x}^2,
\end{equation*}
where our algorithm outperforms others. 
Here, \(m\) denotes the number of observations, \(a_i \in \mathbb{R}^n\), \(y_i \in \{0,1\}\), \(\sigma(z) = \frac{1}{1+e^{-z}}\) and  \(\gamma > 0\) is a regularization parameter. The gradient of \(f\) is given by \( \nabla f(x) = \frac{1}{m}\sum_{i=1}^{m} a_i\big(\sigma(a_i^{\intercal}x) - y_i\big) + \gamma x \). As discussed in \cite{MalitskyMishchenko2020_adaptive}, the function \(f\) is \(L\)-smooth with \( L = \frac{1}{4}\lambda_{\text{max}}(A^{\intercal}A) + \gamma \), where $ A = {\begin{bmatrix} a_1^{\intercal} & \dots & a_m^{\intercal} \end{bmatrix}}^{\intercal} \in \bbR^{m \times n}$, and \(\lambda_{\text{max}}\) denotes the largest eigenvalue.

We follow the experimental setup described in \cite[\S~6.1]{ZhouMaYang2024_adabb}. The value \( f_\star \) is defined as the lowest objective function value achieved among all tested algorithms.  
We used the \textit{mushrooms}, \textit{w8a}, and \textit{covtype} datasets from LIBSVM \cite{ChangLin2011_libsvm}.  
The parameter details for each dataset are given in Table \ref{tab:logistic_parameters}.
\begin{table}[H]
    \centering
    \begin{tabular}{c|c c c c c c} 
        \toprule
        \textbf{dataset} & {$m$} & {$n$} & {$L$} & {$\gamma$} & {max iteration}   \\ 
        \midrule
        \textbf{mushrooms} & 8,124 & 112 & 2.59 & ${L}/{m}$ & 600    \\
        \textbf{w8a} & 49,749 & 300 & 0.66 & ${L}/{m}$ & 1,000    \\
        \textbf{covtype} & 581,012 & 54 & 5.04 $\times$ 10$^6$ & ${L}/{(10m)}$ & 9,000    \\
        \bottomrule
    \end{tabular}
    \caption{Parameters used in the logistic regression problems}
    \label{tab:logistic_parameters}
\end{table}

\subsubsection{Comparison between our algorithms}

Before comparing our algorithms with prior methods, we first analyze the comparison among our own algorithms. This analysis justifies selecting \adanagpractice\ and \adanagsqrt\ as the representative algorithms. The algorithms under comparison include \adanag\ from \Cref{sec:adanag}, \adagdgOne, \adagdgSqrt, and \adagdgZero\ from \Cref{sec:adagd}, as well as \adanagpractice\ and \adanagsqrt\ from \Cref{sec:adanag-g}.  
We also include the algorithms defined by \(\tau_k = \frac{(k+2)+p}{p}\) with \(p=3, 20\) for \adanagG, to justify the choice \(p=12\) for \adanagpractice. We label them as AdaNAG-G\(_{3}\) and AdaNAG-G\(_{20}\), respectively.  
As shown in the proof of Corollary~\ref{cor:adanag_12}, provided in \Cref{appendix:adanag_12_proof}, AdaNAG-G\(_{3}\) and AdaNAG-G\(_{20}\) also achieve a convergence rate of \(\cO\left(1 / k^2\right)\) in \eqref{eq:convergence_rate_generalized_p}.  

We use $L_0$ as defined in \eqref{eq:initial_guess} to define the initial step size for all our algorithms, as stated in the theorems. 
Specifically, we set  
\begin{equation}    \label{eq:tilde_x_0}
    \tilde{x}_0 = x_0 + u,
\end{equation}
where the entries of $u$ are chosen randomly from the uniform distribution over $[0,1]$. 
Also, we set $x_0$ as the zero vector. 
The corresponding plots are presented in \Cref{fig:logistic_our_methods}.  

\begin{figure}[H]
    \centering
    \begin{subfigure}[b]{0.32\textwidth}
        \centering        
         \includegraphics[width=\textwidth]{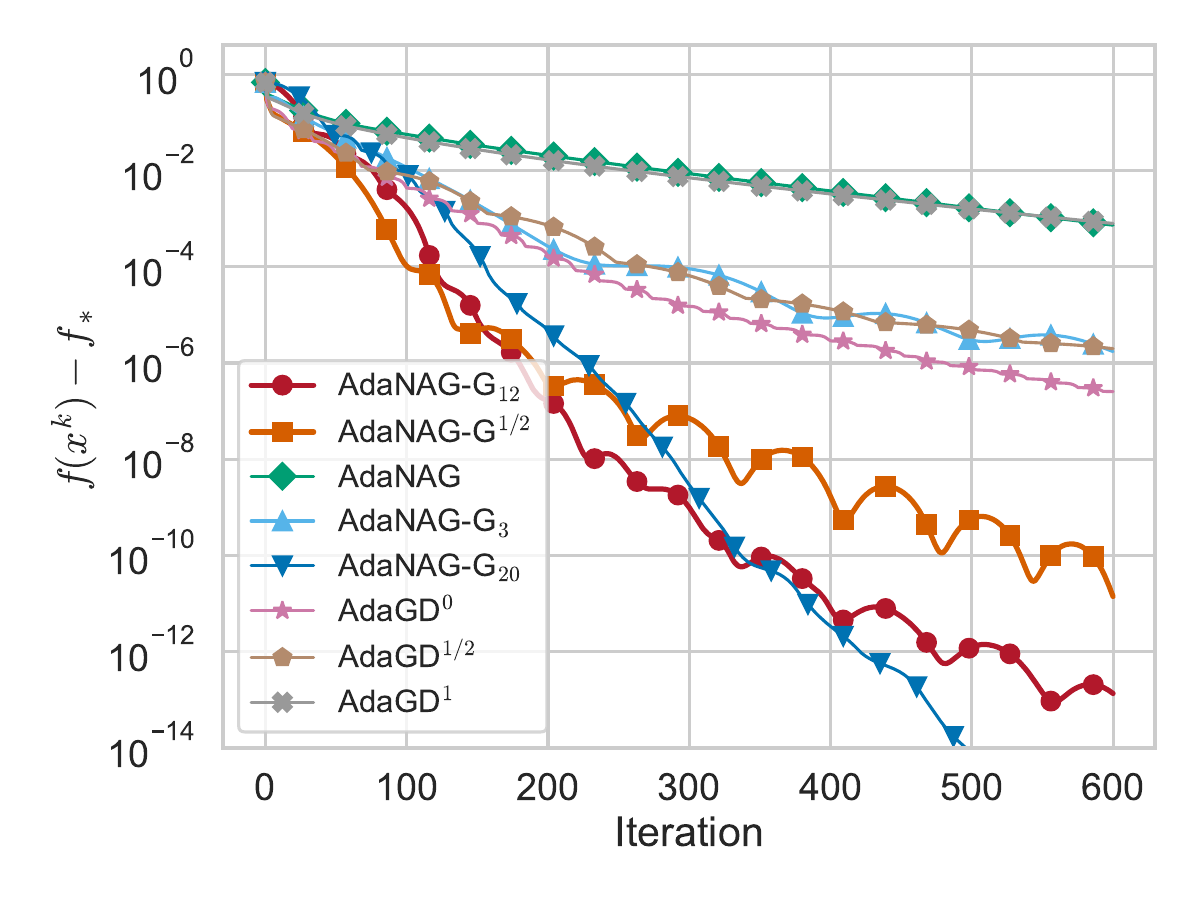}
        \caption{mushrooms, iteration $6 \times 10^2$}
    \end{subfigure}
    \begin{subfigure}[b]{0.32\textwidth}
        \centering
        \includegraphics[width=\textwidth]{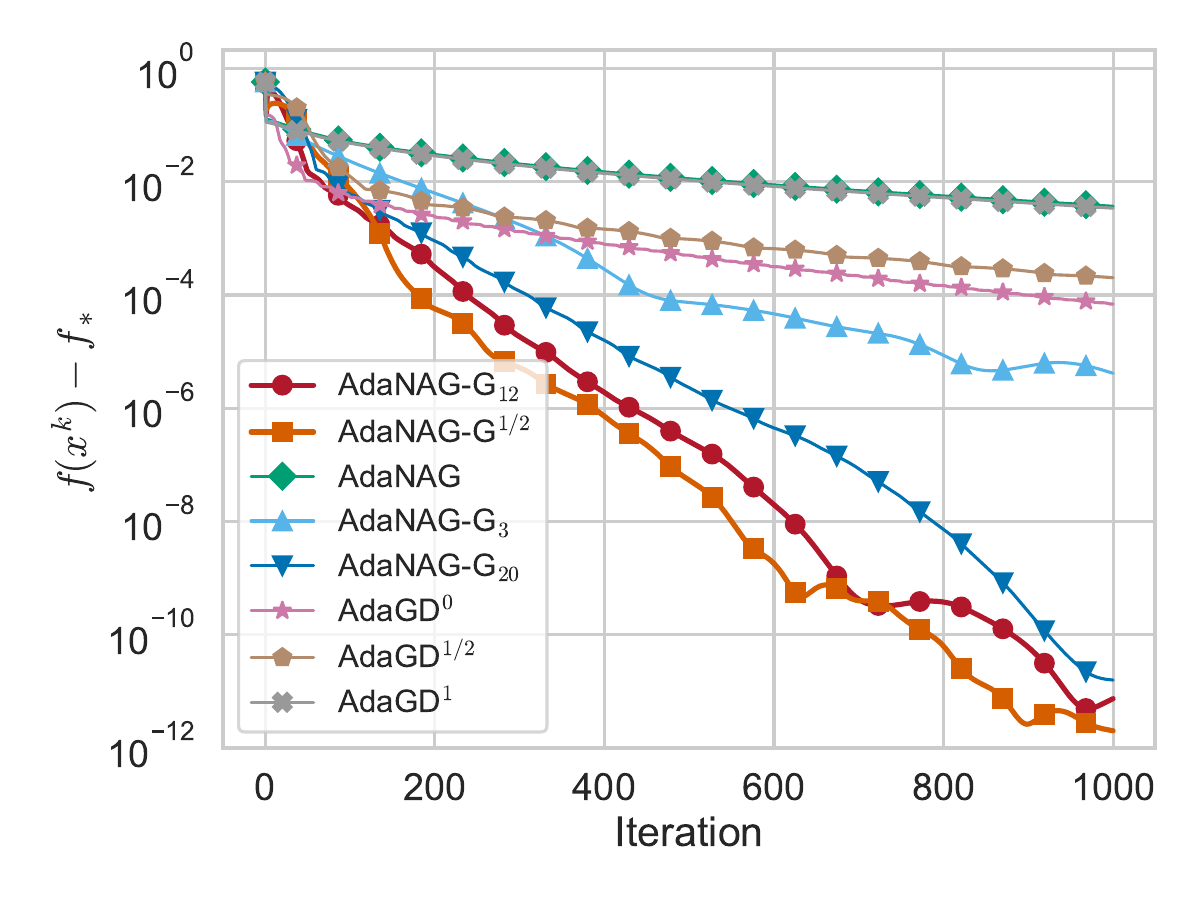}
        \caption{w8a, iteration $10^3$}
    \end{subfigure}   
    \begin{subfigure}[b]{0.32\textwidth}
        \centering        
        \includegraphics[width=\textwidth]{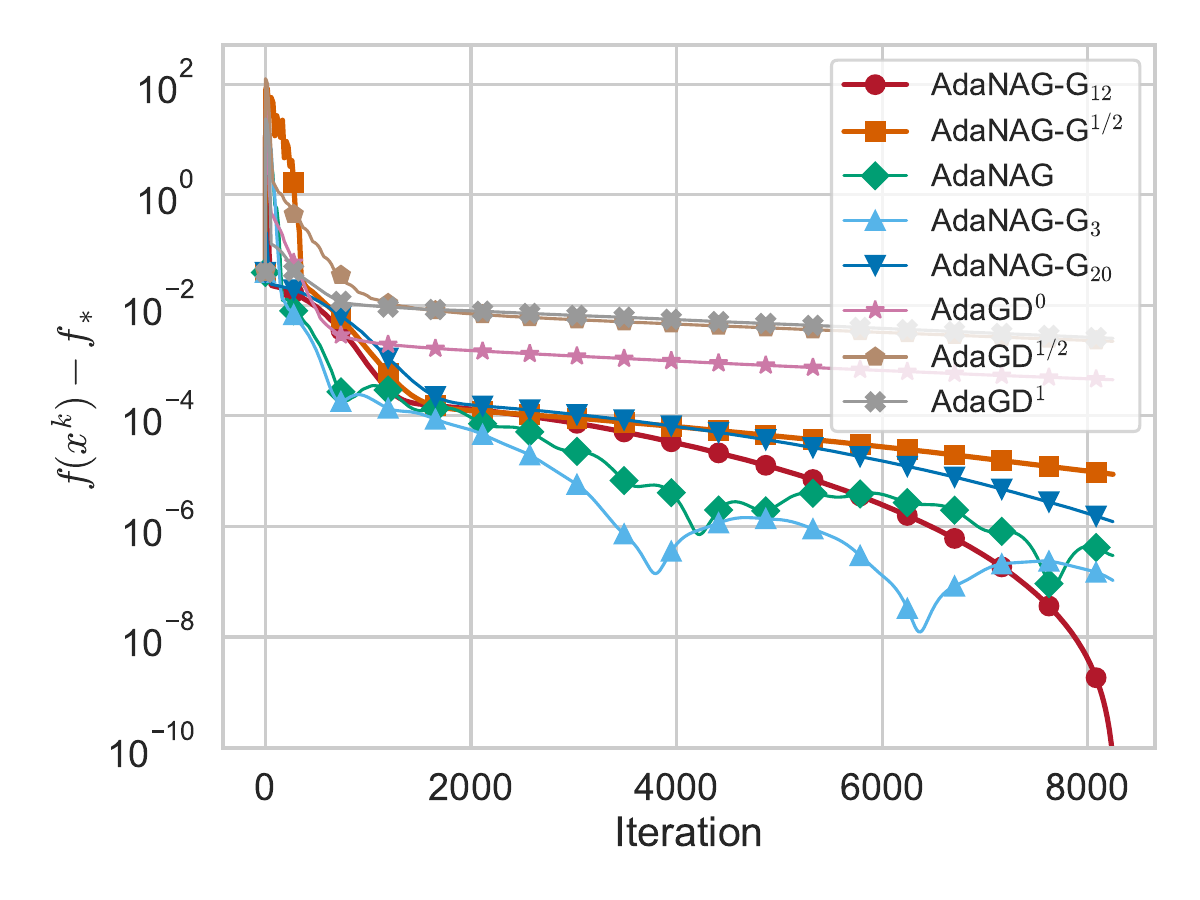}
        \caption{covtype, iteration $9\times10^3$}
    \end{subfigure}
    \caption{Logistic regression problem, comparison between our algorithms}
    \label{fig:logistic_our_methods}
\end{figure}
First, we focus on the ones with the best performance. 
Comparing the three accelerated algorithms with \(\cO\pr{{1}/{k^2}}\) convergence rate, AdaNAG-G\(_{3}\), \adanagpractice\ and AdaNAG-G\(_{20}\), the best-performing one varies depending on the case.  
However, \adanagpractice\ never placed the third. 
\adanagsqrt, which has an \(\cO\pr{{1}/{k}}\) convergence rate, outperformed other AdaNAG instances in the \textit{w8a} dataset experiment. 
\adanag\  and \adagdgOne, which are of the best theoretical guarantees for accelerated and GD-type algorithms, respectively, performed poorly in practice.  

Comparing the GD-type algorithms, \adagdgZero\ always performed the best, while \adagdgOne\ always performed the worst. Note that this is the opposite order of the speed of their theoretical guarantees.  
Overall, GD-type algorithms never performed the best compared to other algorithms for logistic regression problems.  

Based on these observations, we will only compare \adanagpractice\ and \adanagsqrt\ with the previous algorithms from now on.

\subsubsection{Scenarios that our algorithms perform the best} \label{sec:log_prior_algorithms}

The comparisons of \adanagpractice\ and \adanagsqrt\ with prior algorithms are given in Figures \ref{fig:logistic-reg-iteration} and \ref{fig:logistic-reg-cpu}. 
We observe that our algorithms outperform previous methods for the three tested datasets. 
In particular, \adanagpractice\ consistently outperformed all prior algorithms across all three cases. Meanwhile, \adanagsqrt\ achieved the best performance for the \emph{w8a} dataset, while in the other cases, it performed similarly to the AC-FGM algorithms.

\begin{figure}[H]
    \centering
    \begin{subfigure}[b]{0.32\textwidth}
        \centering        
         \includegraphics[width=\textwidth]{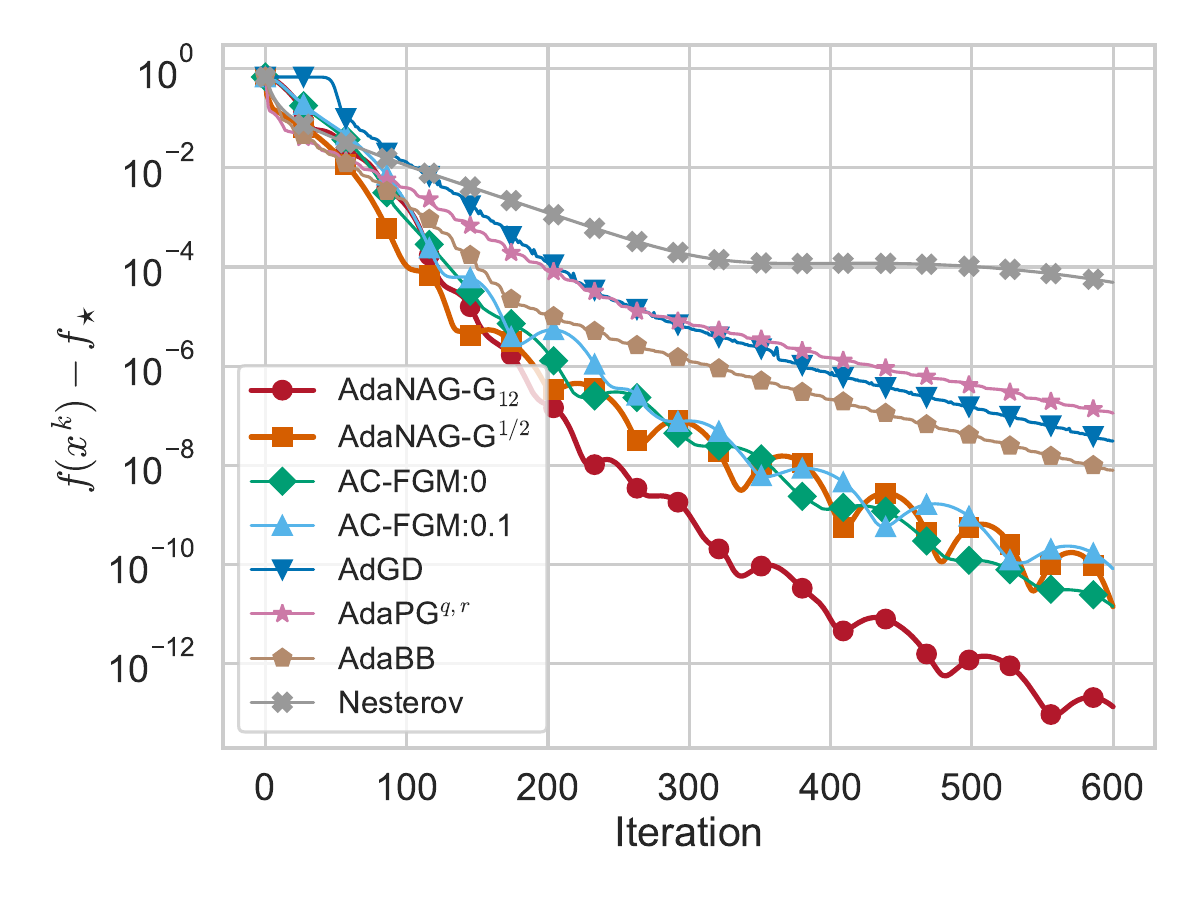}
        \caption{mushrooms, iteration $6 \times 10^2$}
    \end{subfigure}
    \begin{subfigure}[b]{0.32\textwidth}
        \centering
        \includegraphics[width=\textwidth]{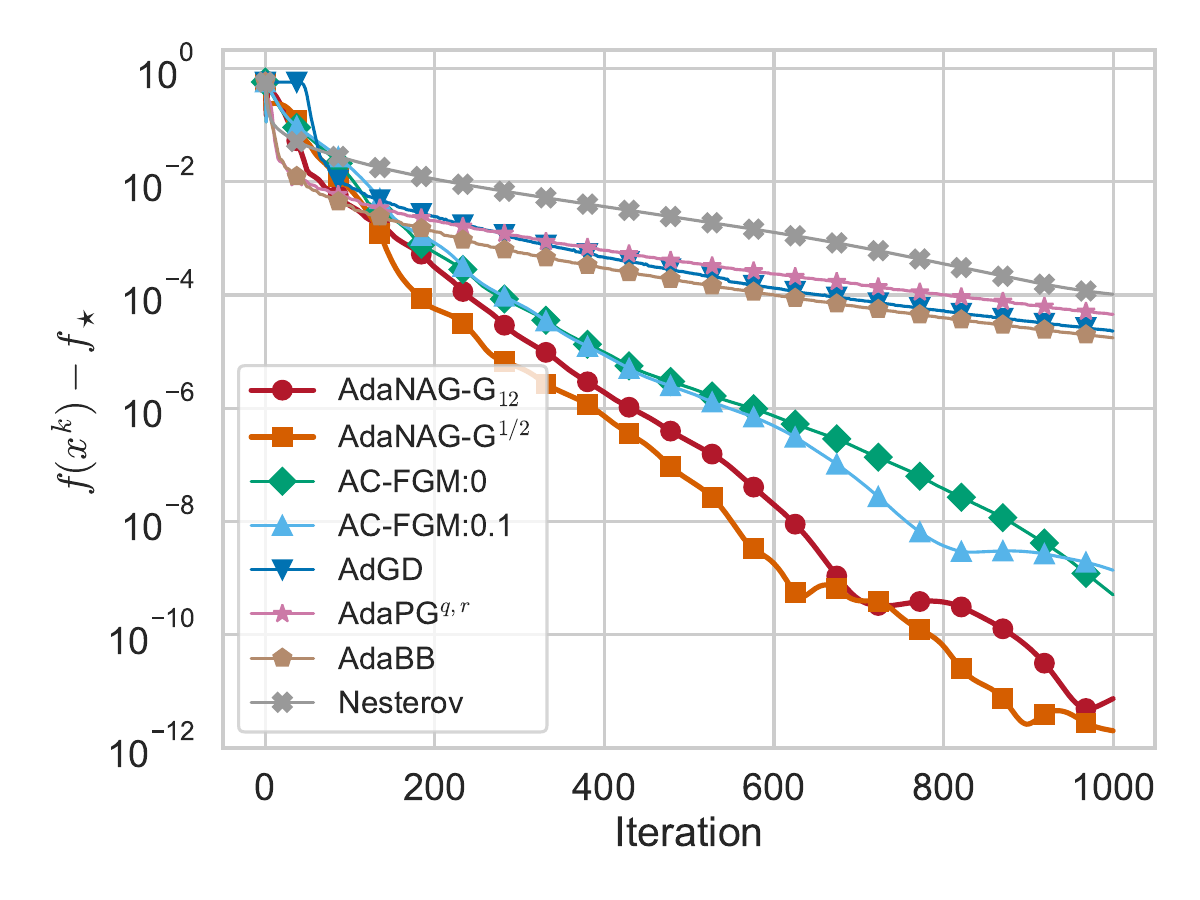}
        \caption{w8a, iteration $10^3$}
    \end{subfigure}   
    \begin{subfigure}[b]{0.32\textwidth}
        \centering        
        \includegraphics[width=\textwidth]{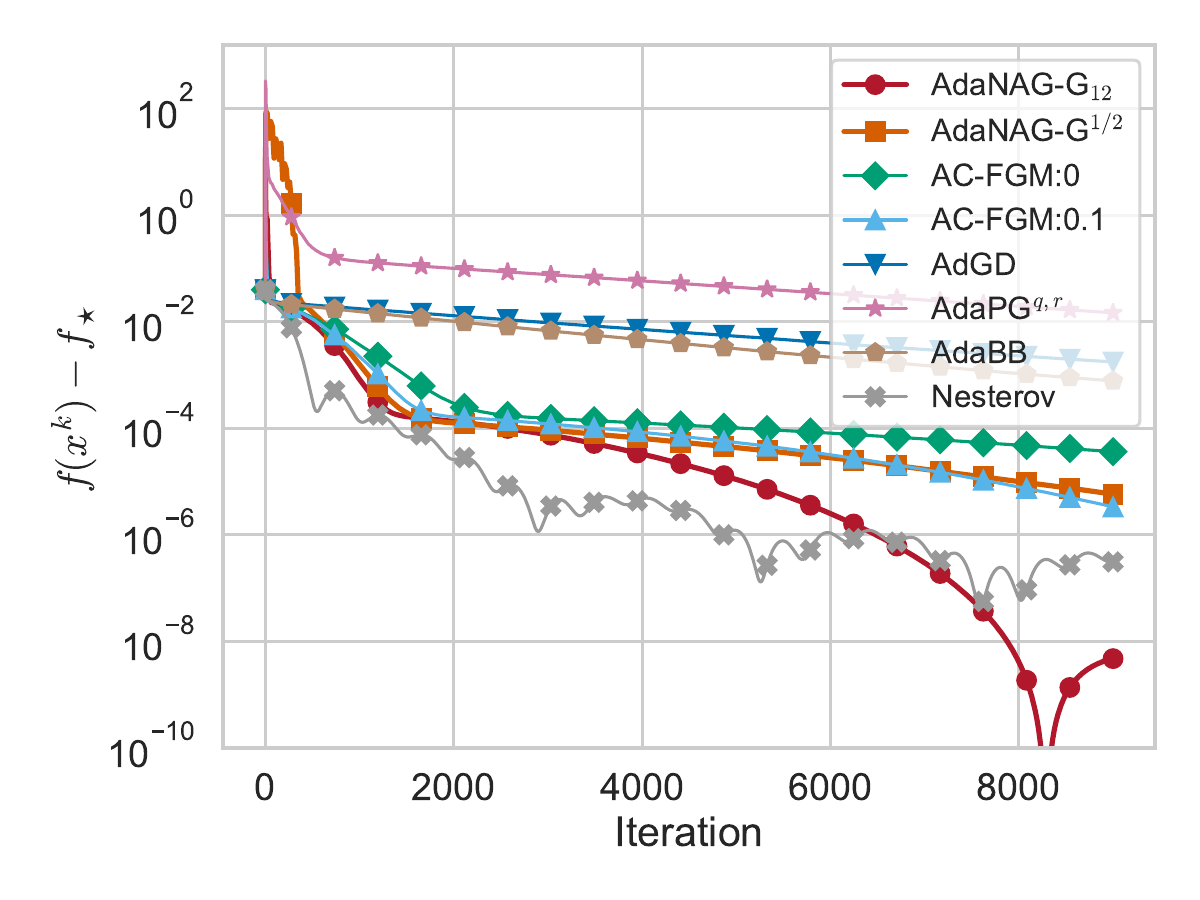}
        \caption{covtype, iteration $9\times10^3$}
    \end{subfigure}
    \caption{Logistic regression problem: iteration number}\label{fig:logistic-reg-iteration}
\end{figure}

\vspace{-3mm}

\begin{figure}[H]
    \centering
    \begin{subfigure}[b]{0.32\textwidth}
        \centering        
         \includegraphics[width=\textwidth]{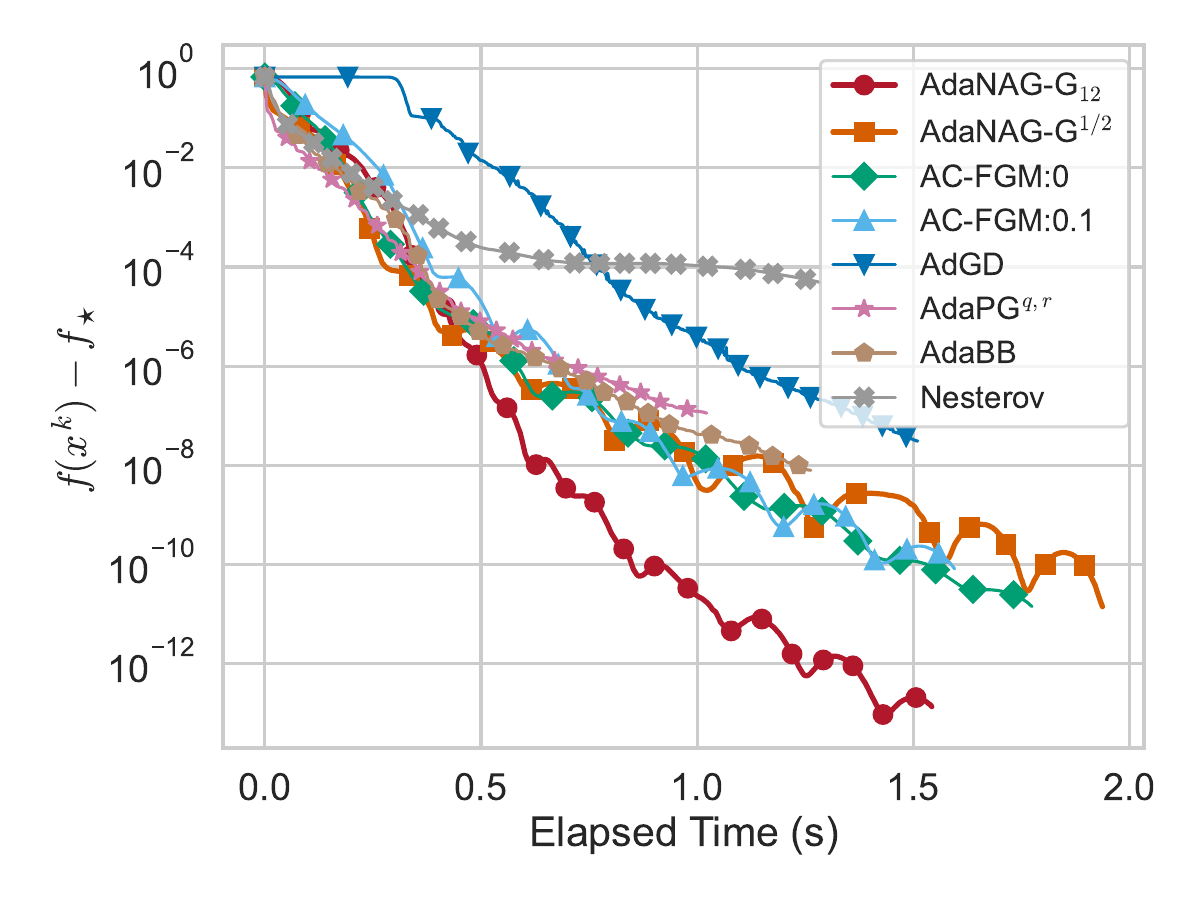}
        \caption{mushrooms}
    \end{subfigure}
    \begin{subfigure}[b]{0.32\textwidth}
        \centering
        \includegraphics[width=\textwidth]{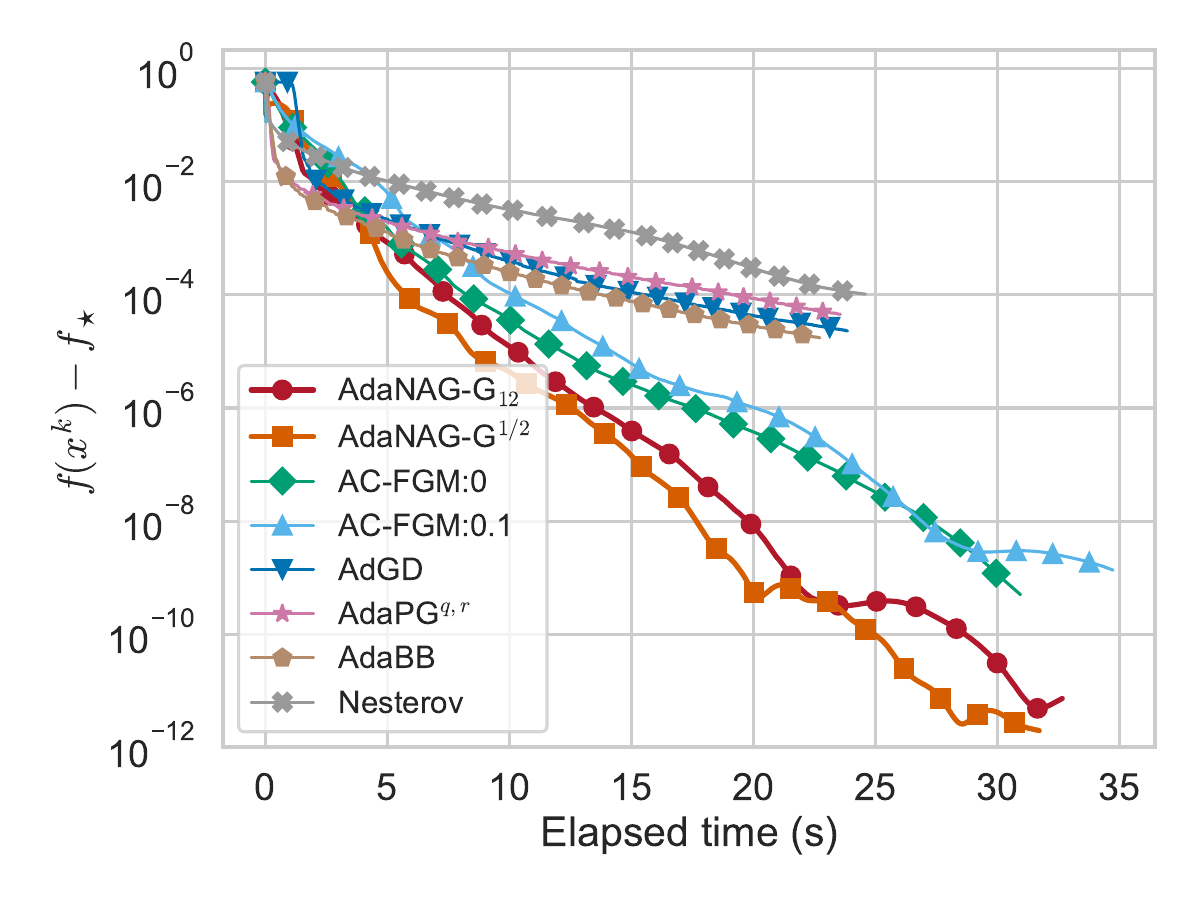}
        \caption{w8a}
    \end{subfigure}   
    \begin{subfigure}[b]{0.32\textwidth}
        \centering        
        \includegraphics[width=\textwidth]{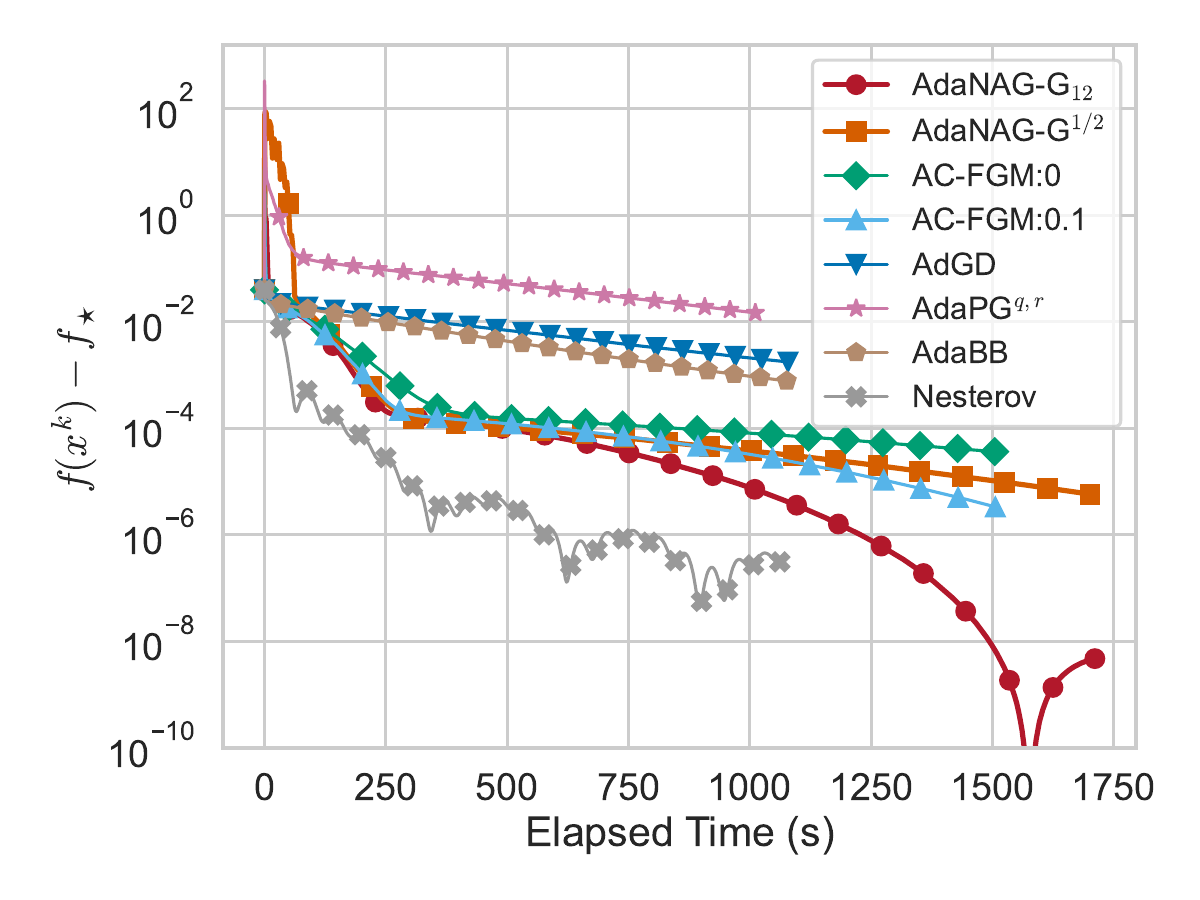}
        \caption{covtype}
    \end{subfigure}
    \caption{Logistic regression problem: CPU time}\label{fig:logistic-reg-cpu}
\end{figure}

The details of the algorithm parameters are as follows. For our algorithms \adanagpractice\ and \adanagsqrt, we use $L_0$ as defined in \eqref{eq:initial_guess}, with $\tilde{x}_0$ as considered in \eqref{eq:tilde_x_0}.  
Overall, for the parameter settings of algorithms from prior works, we selected the best-performing variant introduced in each respective paper.  
For AC-FGM, we set $\beta = 1 - \frac{\sqrt{6}}{3}$ and $\eta_1 = \frac{2}{5L_0}$, as in \cite[\S~4]{LiLan2024_simple}, and consider $\alpha = 0$ and $0.1$. 
In the figures, AC-FGM:0 and AC-FGM:0.1 mean AC-FGM with $\alpha = 0$ and $\alpha = 0.1$, respectively. 
For AdGD, we chose one of Algorithms 1 and 2 from \cite{MalitskyMishchenko2024_adaptive}, and we performed one line search to ensure $\alpha_0 L_1 \in [1/\sqrt{2}, 2]$ for Algorithm 2.  
For AdaPG$^{q,r}$, we set $q = 5/3$ and $r = 5/6$, based on comparisons of the suggested values in Table 1 of \cite{OikonomidisLaudeLatafatThemelisPatrinos2024_adaptive}. We use
$\gamma_0 = {\norm{\tilde{x}_0 - x_0}^2}/{\inner{\tilde{x}_0 - x_0}{\nabla f(\tilde{x}_0) - \nabla f(x_0)}} = \gamma_{-1}$, 
where $\tilde{x}_0$ is defined as in \eqref{eq:tilde_x_0}. 
Finally, for AdaBB, we adopt AdaBB3 as introduced in Table 2 of \cite[\S~6]{ZhouMaYang2024_adabb}, and set $\alpha_0 = 10^{-10}$, $\theta_0 = \max\set{\lambda_1^2 / (2\alpha_0^2) - 1, 0}$, and $\theta_1 = 1$.  

Overall, the accelerated adaptive algorithms and their variants (AdaNAG, AC-FGM) demonstrated superior performance compared to GD-type adaptive algorithms (AdaBB, AdGD, AdaPG$^{q,r}$).  
This performance gap is especially noticeable in the \textit{covtype} dataset,  
which has a significantly larger \( m \) compared to the others, requiring more iterations and a longer running time.  
Interestingly, the \textit{covtype} dataset is also the only case where the non-adaptive Nesterov algorithm outperformed most of the other algorithms.

\subsection{Least squares problem}

In this subsection, we present several experimental cases of the classical least squares problem 
\begin{equation*}
    \min_{x\in\mathbb{R}^n} f(x) = \frac{1}{m} \|Ax - b\|^2,
\end{equation*}
where our algorithm outperforms others or achieves the best performance. 
Here, $A \in \mathbb{R}^{m\times n}$ and $b\in\mathbb{R}^m$.  
We provide two cases with randomly generated synthetic data and two cases with real-world datasets, \textit{bodyfat} and \textit{cadata}, from LIBSVM \cite{ChangLin2011_libsvm}.  

To ensure a fair comparison, we follow most of the experimental details in \cite[\S~4.1]{LiLan2024_simple}, using the same settings as our benchmark algorithm \ref{eq:AC-FGM}.  
We consider both synthetic and real-world datasets, comparing two cases with relatively larger and smaller values of $m$, which corresponds to the number of data points.  
In the synthetic data setting, $A$ is generated with entries uniformly distributed in $[0,1]$. An optimal point $x_\star$ is randomly chosen from the unit ball $B_1(0)$, and we set $b = Ax_\star$, which implies $f_\star = 0$. 
The parameter details for the algorithms are the same as in \Cref{sec:log_prior_algorithms} and parameters for each dataset are summarized in \Cref{tab:qp_parameters}. 
\begin{table}[H]
    \centering
    \begin{tabular}{c|c c c c c c} 
        \toprule
        \textbf{dataset} & {$m$} & {$n$} & {$L$} &  {max iteration}   \\ 
        \midrule
        \textbf{random-small} & 1,000 & 4,000 & 2,000 & 6 $\times$ 10$^{3}$    \\
        \textbf{random-large} & 4,000 & 8,000 & 4,000 & 10$^{4}$    \\
        \textbf{bodyfat} & 252 & 14 & 156,269 & 2 $\times$ 10$^{4}$    \\
        \textbf{cadata} & 20,640 & 8 & 20  &   10$^{5}$   \\ 
        \bottomrule
    \end{tabular}
    \caption{Parameters used in the least squares problem}
    \label{tab:qp_parameters}
\end{table} 
As captured in \Cref{fig:least_square_synthetic} and \Cref{fig:least_square_real_world}, 
our algorithms \adanagpractice\ and \adanagsqrt\ outperform others or achieve the best performance. 
For synthetic datasets in \Cref{fig:least_square_synthetic}, we observe that \adanagsqrt\ outperforms other algorithms for smaller datasets, while \adanagpractice\ outperforms others for larger datasets. 
For the real-world datasets in \Cref{fig:least_square_real_world}, \adanagpractice\ achieves the best performance on both datasets, tying with AC-FGM:0.1 on \textit{bodyfat} but performing noticeably better on \textit{cadata}. 
Note that the performance gap between \adanagpractice\ and \adanagsqrt\ is larger on \textit{cadata}, which is the larger dataset.

\begin{figure}[ht]
    \centering
    \begin{subfigure}[b]{0.34\textwidth}
        \centering        
         \includegraphics[width=\textwidth]{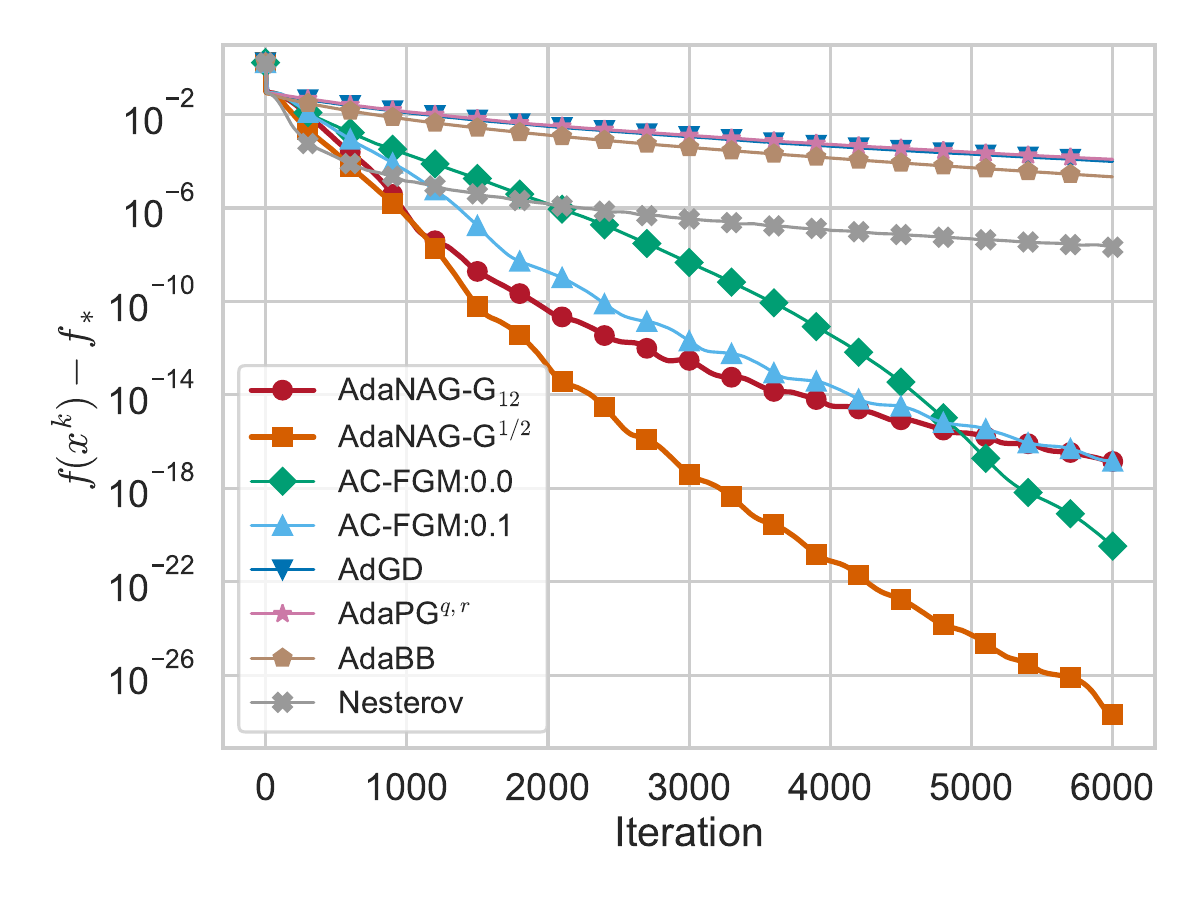}
        \caption{random-small, iteration $6\times10^3$}
    \end{subfigure}
    \begin{subfigure}[b]{0.34\textwidth}
        \centering        
         \includegraphics[width=\textwidth]{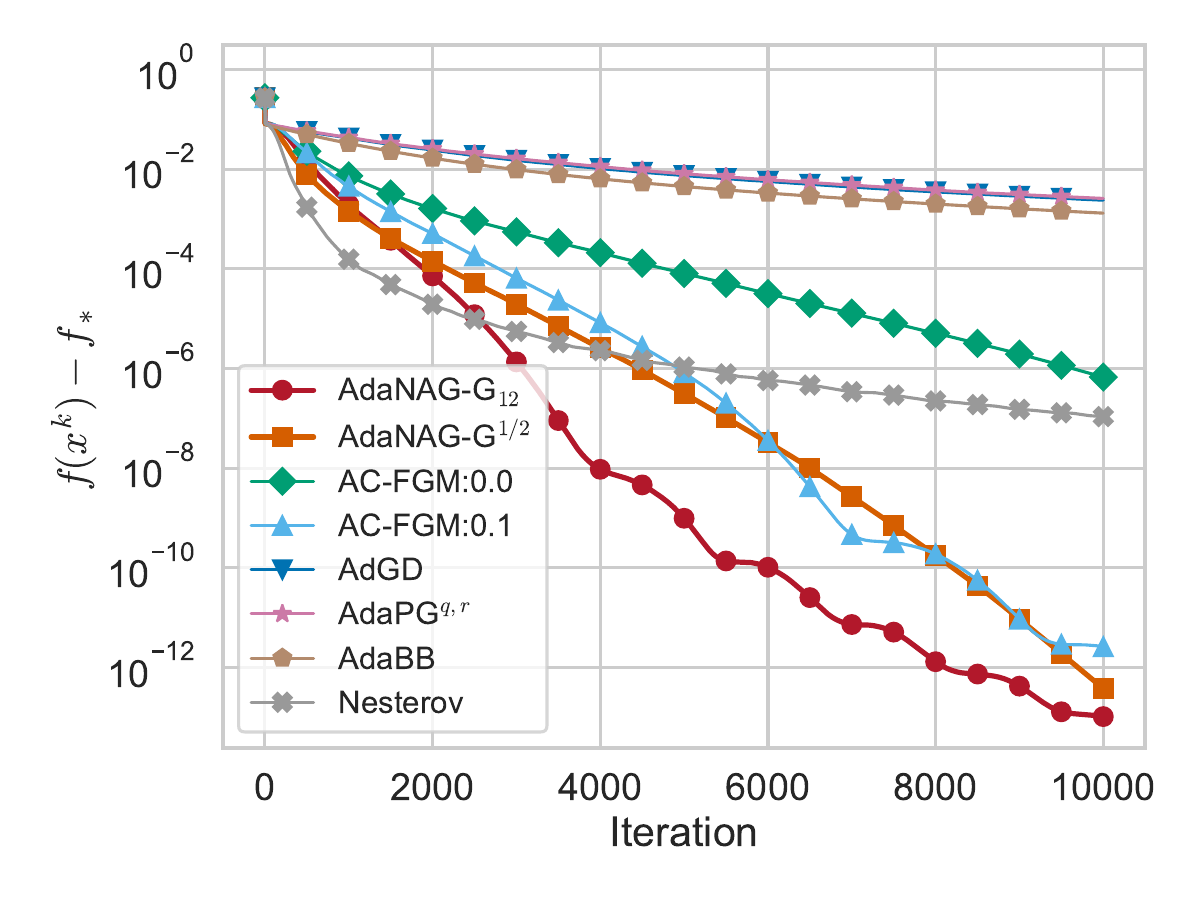}
        \caption{random-large, iteration $10^4$}
    \end{subfigure} 
    \caption{Least squares problem, synthetic datasets}
    \label{fig:least_square_synthetic}
\end{figure}

\begin{figure}[ht]
    \centering
    \begin{subfigure}[b]{0.34\textwidth}
        \centering        
        \includegraphics[width=\textwidth]{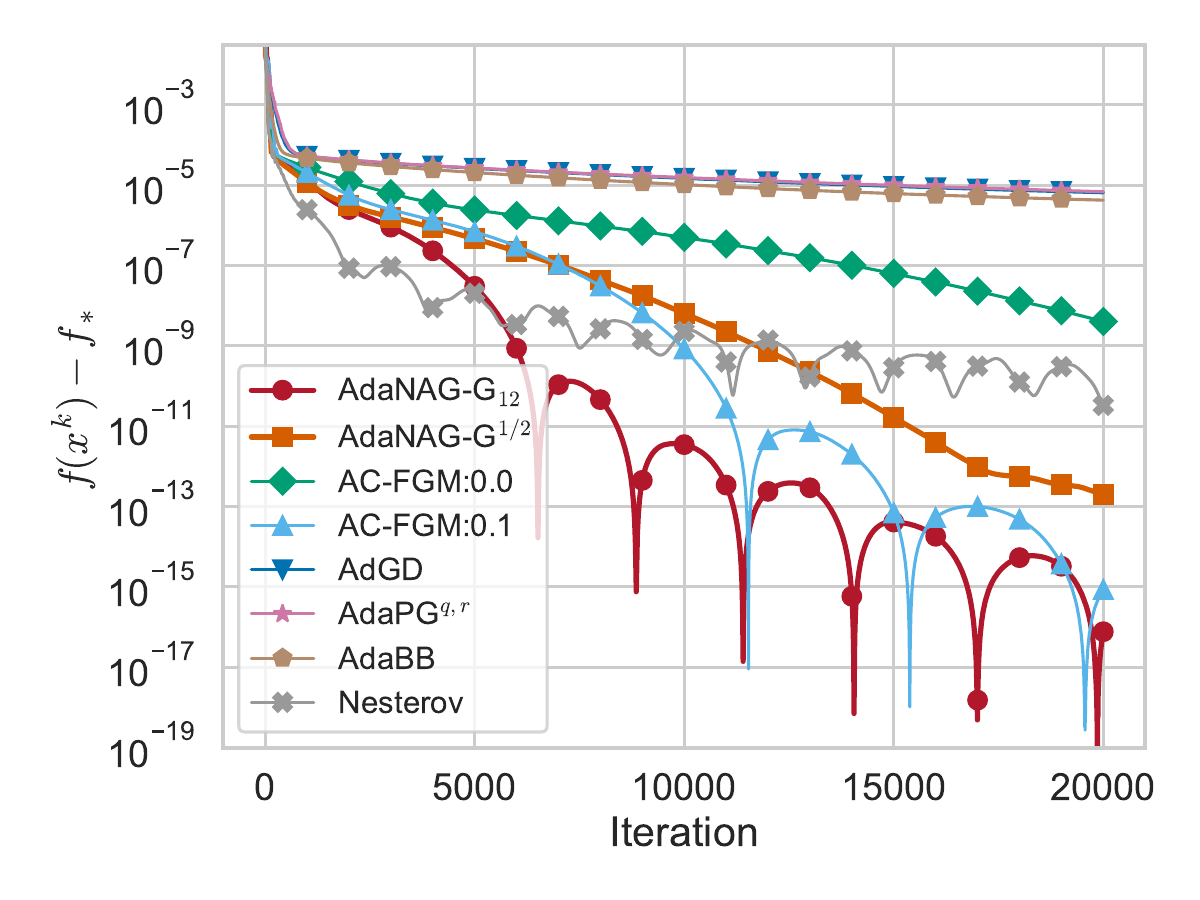}
        \caption{bodyfat, iteration $2 \times 10^4$}
    \end{subfigure}
    \begin{subfigure}[b]{0.34\textwidth}
        \centering
        \includegraphics[width=\textwidth]{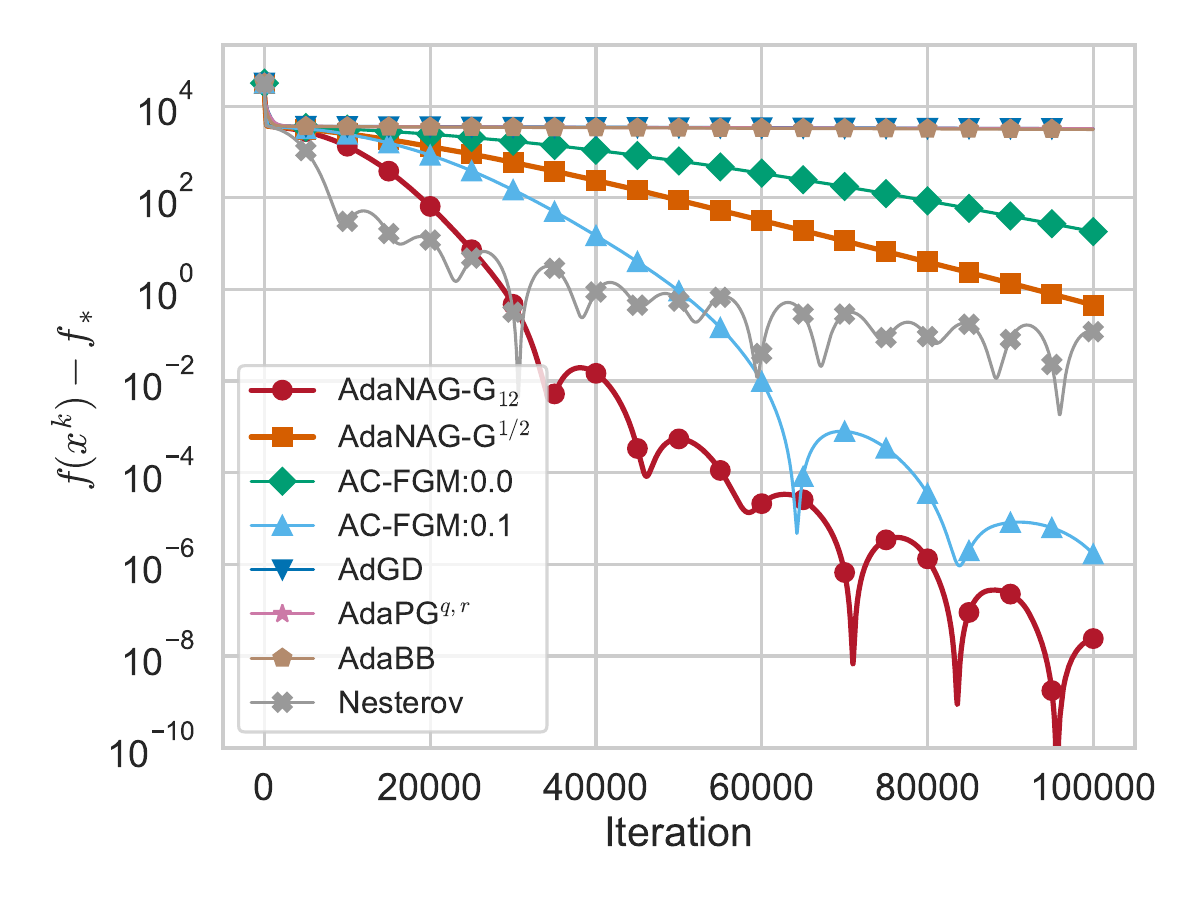}
        \caption{cadata, iteration $10^5$}
    \end{subfigure}  

    \caption{Least squares problem, real-world datasets}
    \label{fig:least_square_real_world}
\end{figure}

To summarize, \adanagpractice, which achieves the accelerated rate of $\cO\pr{{1}/{k^2}}$, tends to perform the best overall and performs even better on larger datasets that require more iterations. 
Another noteworthy observation is that, for the least squares problem, variants of accelerated adaptive algorithms (AdaNAG, AC-FGM) perform significantly better than GD-type adaptive algorithms (AdaGD, AdaPG$^{q,r}$, AdaBB). 
\section{Conclusion} \label{sec:conclusion}

This paper introduces a novel adaptive accelerated algorithm based on Nesterov's accelerated gradient method, named \adanag. The algorithm is line-search-free, parameter-free, and achieves the accelerated convergence rate 
$f(x_k) - f_\star = \mathcal{O}\pr{\hfrac{1}{k^2}}$. 
We establish this result via a Lyapunov analysis, and also achieve $\min_{i\in\set{1,\dots, k}} \norm{\nabla f(x_i)}^2 = \mathcal{O}\pr{\hfrac{1}{k^3}}$, as in the original Nesterov's algorithm, which is the first of its kind. 
Applying a similar proof idea, we further introduce a family of novel GD-type adaptive algorithms, denoted as \adagd, and provide an instance that achieves a non-ergodic rate $f(x_k) - f_\star = \mathcal{O}\pr{\hfrac{1}{k}}$, analogous to classical gradient descent. 
We observe that, within a family of adaptive algorithms, choosing a parameter associated with a slower convergence rate can allow for a larger step size. 
Inspired by the previous observation, we propose another family of adaptive algorithms that generalizes \adanag, which includes instances that both theoretically sound and practically useful, \adanagpractice\ and \adanagsqrt.

Our adaptive method \adanag\ resembles the original Nesterov's method in the form of the algorithm and the Lyapunov-based proof structure, and achieves the same order of convergence rates.  
A natural next step would be to investigate whether other properties and extensions developed for the original Nesterov's method can also be applied to \adanag.  
Studying whether \adanag\ or its variants can achieve iterate convergence~\cite{ChambolleDossal2015_convergence, AttouchChbaniPeypouquetRedont2018_fast}, 
extend to the prox-grad setup 
\cite{BeckTeboulle2009_fast, KimFessler2018_another, JangGuptaRyu2024_computerassisted}, or extend to  other advanced setups \cite{AhnSra2020_nesterovs, KimYang2022_accelerated, ChenShiJiangWang2024_penaltybased, CaoJiangHamedaniMokhtari2024_accelerated}, 
would be interesting directions for future work.

\bibliographystyle{plain}   
\bibliography{ref}

\newpage
\appendix
\section{Omitted proofs in \Cref{sec:preliminary}} 
\label{appendix:proofs_for_preliminary}

\subsection{Proof of \Cref{lemma:local_smoothness}} \label{appendix:proof_of_local_smoothness}

We use arguments similar to those for $L$-smooth functions that are provided in \cite[Theorem~2.1.5]{Nesterov2004_introductory}. 
Let $K\subset\ourspace$ be a compact set. 
Let $R>0$ and $c\in\ourspace$ satisfy $K \subset \bar{B}_R(c) = \set{x \in \bbR^d \mid \norm{ x - c } \le R}$. 
For notation simplicity, denote $\bar{K} = \bar{B}_{3R}(c)$. Since $f$ is locally smooth, there is $\bar{L}_{K}>0$ satisfying
\begin{equation*}
    \qquad\qquad
    \norm{ \nabla f(\bar{x}) - \nabla f(\bar{y}) } \le \bar{L}_{K} \norm{ \bar{x} - \bar{y} }, \qquad \forall \bar{x}, \bar{y} \in \bar{K}.
\end{equation*}
Now, take $x,y \in K$ such that $x\ne y$. 
Define $\tilde{f}_y\colon\ourspace\to\bbR$ as
\begin{equation}    \label{eq:translated_function}
    \tilde{f}_y(z) = f(z) - \inner{\nabla f(y)}{z}.
\end{equation}
Note
\begin{equation}    \label{eq:grad_translated_function}
    \nabla \tilde{f}_y(z) = \nabla f(z) - \nabla f(y). 
\end{equation}
Form \eqref{eq:grad_translated_function} we know that $\tilde{f}_y$ has the same smoothness parameter as $f$, that is
\begin{equation*}
    \qquad\qquad
    \norm{ \nabla \tilde{f}_y(\bar{x}) - \nabla \tilde{f}_y(\bar{y}) } = 
    \norm{ \nabla f(\bar{x}) - \nabla f(\bar{y}) } \le \bar{L}_{K} \norm{ \bar{x} - \bar{y} }, \qquad \forall \bar{x}, \bar{y} \in \bar{K}.
\end{equation*}
Also, we know $\tilde{f}_y(y) = \min_{u \in \bbR^d} \tilde{f}_y(u)$ since $\nabla \tilde{f}_y(y) = 0$ and  $\tilde{f}_y$ is a convex function. 
Now, define
\begin{equation}    \label{eq:xplus_smoothness_proof}
    x^{+} = x - \frac{1}{\bar{L}_{K}} \nabla \tilde{f}_y(x).
\end{equation}
Then from \eqref{eq:grad_translated_function} and the fact $x,y \in \bar{K}$, we have
\begin{equation*}
    \norm{ x^{+} -x }
    = \frac{1}{\bar{L}_{K}} \norm{ \nabla \tilde{f}_y(x) } = \frac{1}{\bar{L}_{K}} \norm{ \nabla f(x) - \nabla f(y) }
    \le \norm{ x-y }.
\end{equation*}
By triangular inequality and the fact $x,y\in \bar{B}_R(c)$, we have 
\begin{equation*}
    \norm{ x^{+} - c } 
    \le \norm{ x - c } + \norm{ x^{+} -x }
    \le \norm{ x - c } + \norm{ x-y } 
    \le 3R,
\end{equation*}
and we can conclude $x^{+} \in \bar{B}_{3R}(c) = \bar{K}$. 
Thus for $\segmentparameter\in[0,1]$ we have
\begin{equation}    \label{ineq:tilde_grad_difference}
    \begin{aligned}
        &\norm{ \nabla \tilde{f}_y(x + \segmentparameter(x^+-x)) - \nabla \tilde{f}_y(x) } 
        = \norm{ \nabla f(x + \segmentparameter(x^+-x)) - \nabla f(x) }  
        \le \segmentparameter \bar{L}_{K} \norm{x^{+}-x}, 
    \end{aligned}
\end{equation}
where the equality comes from \eqref{eq:grad_translated_function} and the inequality comes from the fact $x, x^{+} \in \bar{K}$. 

Now we can proceed with the standard argument of \cite[Lemma~1.2.3]{Nesterov2004_introductory}. 
From the Cauchy–Schwarz inequality, \eqref{ineq:tilde_grad_difference}, and \eqref{eq:xplus_smoothness_proof}, we obtain: 
\begin{equation*}
    \begin{aligned}
         \tilde{f}_y(x^{+}) 
         &= \tilde{f}_y(x) + \inner{\nabla \tilde{f}_y(x)}{x^{+}-x} + \int_{0}^{1} \inner{\nabla \tilde{f}_y( x + \segmentparameter(x^{+}-x) ) - \nabla \tilde{f}_y(x)}{x^{+}-x} d\segmentparameter \\
         &\le \tilde{f}_y(x) + \inner{\nabla \tilde{f}_y(x)}{x^{+}-x} + \int_{0}^{1} \norm{\nabla \tilde{f}_y( x + \segmentparameter(x^{+}-x) ) - \nabla \tilde{f}_y(x)}\norm{x^{+}-x} d\segmentparameter \\
         &\le \tilde{f}_y(x) + \inner{\nabla \tilde{f}_y(x)}{x^{+}-x} + \int_{0}^{1} \segmentparameter \bar{L}_{K} \norm{x^{+}-x}^2 d\segmentparameter \\ 
         &= \tilde{f}_y(x) - \frac{1}{2\bar{L}_{K}} \norm{\nabla \tilde{f}_y(x)}^2.
    \end{aligned}
\end{equation*} 

Now recalling the fact $\tilde{f}_y(y) = \min_{u \in \bbR^d} \tilde{f}_y(u)$ and applying the above inequality, we have
\begin{equation*}
    \begin{aligned} 
    \tilde{f}_y(y) 
    \le \tilde{f}_y(x^{+}) 
    &\le \tilde{f}_y(x) - \frac{1}{2\bar{L}_{K}} \norm{\nabla \tilde{f}_y(x)}^2. 
    \end{aligned}
\end{equation*}
Plugging the definition of $\tilde{f}_y$ in \eqref{eq:translated_function} and \eqref{eq:grad_translated_function} we have
\begin{equation*}
    \begin{aligned}
    f(y) - \inner{\nabla f(y)}{y}
    \le f(x) - \inner{\nabla f (y)}{x} - \frac{1}{2\bar{L}_{K}} \norm{\nabla f(x) - \nabla f(y)}^2,
    \end{aligned}
\end{equation*}
which implies 
\begin{equation} \label{eq:local_smooth_ineq}
    f(y) - f(x) + \inner{\nabla f(y)}{x-y} + \frac{1}{2\bar{L}_{K}} \norm{\nabla f(x) - \nabla f(y)}^2 \le 0.
\end{equation}
Since $x,y\in K$ were arbitrary, this completes the proof.  
\qed

\subsection{Proof of Corollary \ref{cor:L_k_bound}} \label{appendix:L_k_bound} 
\begin{itemize}
    \item [(i)] Suppose $f(y) - f(x) + \inner{\nabla f(y)}{x-y}=0$. For any compact set $K$ that contains $x$ and $y$, the desired conclusion follows immediately from Lemma \ref{lemma:local_smoothness}. 

    \item [(ii)] Suppose $f(y) - f(x) + \inner{\nabla f(y)}{x - y} \ne 0$. Then we have $f(y) - f(x) + \inner{\nabla f(y)}{x - y} < 0$ by the convexity of $f$. 
    Multiplying both sides of \eqref{eq:local_smooth_ineq} by $-\frac{1}{f(y) - f(x) + \inner{\nabla f(y)}{x-y}} \bar{L}_{K}$ yields the desired inequalities. 
    \qed
\end{itemize}

\section{Further Discussions on AdaNAG}

\subsection{Proof of Lemma~\ref{lem:theta_properties}}  \label{appendix:theta_properties_proof} 
From \eqref{eq:theta}, we have $\theta_{k+1} = \frac{1}{2} \Big( 1 + \sqrt{1 + 4\theta_{k}^2}\Big)   \ge \frac{1}{2} + \theta_k$. Thus $\theta_k$ is increasing, and the inequality $\theta_{k}\ge\frac{k+2}{2}$ for $k\ge0$ follows by induction and the fact $\theta_0=1$. 
Next, from $\theta_{k+1} = \frac{1}{2}  \Big( 1 + \sqrt{1 + 4\theta_{k}^2}\Big)$ we have $(2\theta_{k+1}-1)^2 = 1 + 4\theta_{k}^2$ and so $\theta_{k+1}^2 - \theta_{k+1} = \theta_k^2$. Therefore $\theta_k$ satisfies \eqref{eq:theta_motivation} as an equality. 

Moreover, from \eqref{eq:theta} we have $\alpha_{k+1} = \frac{1}{2} \Big( 1 - \frac{1}{\theta_{k+3}} \Big) \ge \frac{1}{2} \Big( 1 - \frac{1}{\theta_{k+2}} \Big) = \alpha_k$ for $k\ge1$ since $\theta_k$ is increasing, and $\lim_{k\to\infty} \alpha_k = \lim_{k\to\infty} \frac{1}{2} \Big( 1 - \frac{1}{\theta_{k+2}} \Big) = \frac{1}{2}$ since $\lim_{k\to\infty} \theta_k = \infty$. 
Finally, we have $s_{k+1} \le \frac{\alpha_{k}}{\alpha_{k+1}} s_k \le s_k$ since $\alpha_{k+1} \ge \alpha_k$ for $k\ge1$.   
When $k=0$, leveraging the numerical values in the first column of \Cref{tab:numerical_values}, we can obtain \( s_1 \le \frac{\alpha_0}{\alpha_1} \frac{\theta_2}{\theta_3(\theta_3-1)} \ssz_{0} < s_0 \).  
\qed

\subsection{Simple AdaNAG} \label{sec:adanag-s}
For NAG, \( \theta_k \) defined in \eqref{eq:theta} is of a rather complex form. 
A simpler alternative to it is $\theta_k=\frac{k+2}{2}$ as defined in \eqref{eq:rational_theta}. 
It is easy to verify that $\theta_k = \frac{k+2}{2}$ also satisfies \eqref{eq:theta_motivation}. 
Other properties mentioned in \Cref{lem:theta_properties} can also be verified using the same argument as in \Cref{appendix:theta_properties_proof}. 
In this section, we provide a simplified version of \adanag\ with this \( \theta_k \).  
In this case, \( \alpha_k \) defined in \eqref{eq:theta} becomes:
\begin{equation*} 
    \begin{aligned}
        \alpha_k = \frac{1}{2} \pr{ 1 - \frac{2}{k+4} }, 
        \mbox{ for } k \ge 1, \quad \mbox{ and } \quad \alpha_0 = 
        \frac{60}{127}. 
    \end{aligned}
\end{equation*} 
Plugging in the above, we can check that the coefficients in \eqref{eq:step_size_rule} become $\frac{\alpha_{k}}{\alpha_{k+1}} = \frac{k(k+3)}{(k+1)(k+2)}$ and  
$\frac{\alpha_{k}^2}{\alpha_{k+1} + \alpha_{k}^2} = \frac{k^2 (k+3)}{3 k^3+13 k^2+16 k+8}$. 
Calculating other coefficients in \adanag\ with similar manner, we get a simplified version of \adanag\ described in Algorithm \ref{alg:AdaNAG-S}. 
\begin{algorithm}[H]
     \caption{\textbf{Simple AdaNAG}}
     \label{alg:AdaNAG-S} 
     \begin{algorithmic}[1]
     \STATE \textbf{Input:} $x^0=z^0 \in \mathbb{R}^d$, 
     $\ssz_{0} >0$, 
     \STATE    
     \begin{equation*}
        \begin{aligned}
        y_{1} &= x_0 -  \ssz_{0} \nabla f(x_k) \qquad \\
        z_{1} &= z_0 - \ssz_{0}  \frac{60}{127} \nabla f(x_0) \qquad \\
        x_{1} &= \frac{3}{5} y_{1} + \frac{2}{5} z_{1} \\
        L_{1} &= -\frac{ \frac{1}{2} \norm{ \nabla f(x_{1}) - \nabla f(x_{0}) }^2 }            { f(x_{1}) - f(x_{0}) + \inner{\nabla f(x_{1})}{x_{0}-x_{1}} } \\ \qquad 
        s_1 &= \min\set{ \frac{320}{381} \ssz_{0},  \,\, \frac{50}{177} \frac{1}{L_{1}}  }
        \end{aligned}
    \end{equation*}
    \FOR{$k = 1,2,\dots$}     \STATE
        \begin{equation*}     
                \begin{aligned}
            y_{k+1} &= x_k -  \ssz_{k} \nabla f(x_k) \\
            z_{k+1} &= z_k - \frac{k+2}{4} \ssz_{k} \nabla f(x_k) \\
            x_{k+1} &= \frac{k+3}{k+5} y_{k+1} + \frac{2}{k+5} z_{k+1}  \\
            L_{k+1} &= -\frac{ \frac{1}{2} \norm{ \nabla f(x_{k+1}) - \nabla f(x_{k}) }^2 }            { f(x_{k+1}) - f(x_{k}) + \inner{\nabla f(x_{k+1})}{x_{k}-x_{k+1}} } \\
            \ssz_{k+1}  &= 
                  \min\set{ \frac{k(k+3)}{(k+1)(k+2)} \ssz_{k},  \,\, \frac{k^2 (k+3)}{3 k^3+13 k^2+16 k+8} \frac{1}{L_{k+1}}  } 
            \end{aligned}
        \end{equation*}
       \ENDFOR
     \end{algorithmic}
\end{algorithm} 

Using the same proof argument, 
we can also obtain the convergence result similar to \Cref{thm:main_tight} for \adanagS. 
\begin{proposition} \label{thm:main_simple}
    Let $f$ be an $L$-smooth convex function. 
    In \adanagS, set $s_0 = \frac{635}{1888} \frac{1}{L_0}$ with $L_0$ defined as in \eqref{eq:initial_guess}. 
    Define $\upperboundconstant = \norm{x_{0} - x_\star}^2 + 0.15  \frac{1}{L_0} \pr{ \frac{1}{L_0} - \frac{2}{L} } \norm{\nabla f(x_0)}^2 $. 
    Let $\set{x_k}_{k\ge0}$ be a sequence generated by \adanagS. 
    Then the following convergence rate results hold.
    \begin{equation*}
        \begin{aligned}
        f(x_k) - f_\star
         &\le \frac{ 24 \upperboundL}{ (k+3)(k+5) }  \upperboundconstant
         = \cO\pr{ \frac{\upperboundL}{k^2} }, \\
        \min_{i\in\set{1,\dots, k}} \norm{\nabla f(x_i)}^2 
        &\le \frac{1671\upperboundL^2}{k \left(k^2+10 k+33\right)} \upperboundconstant
        = \cO\pr{ \frac{\upperboundL^2}{k^3} }.
        \end{aligned}
    \end{equation*}
\end{proposition}
\begin{proof} 
The proof follows the same argument as that of \Cref{thm:main_tight}. 
The only difference is that since $\theta_{k+3}(\theta_{k+3}-1)=\theta_{k+2}^2$ does not hold, we instead use $\theta_{k+3}(\theta_{k+3}-1)=\frac{1}{4}(k+3)(k+5)$. 
\end{proof}

\subsection{Extension of \adanag\ to locally smooth $f$} \label{sec:locally_smooth_extension}

Various previous adaptive methods \cite{MalitskyMishchenko2020_adaptive, LatafatThemelisStellaPatrinos2024_adaptive, MalitskyMishchenko2024_adaptive, ZhouMaYang2024_adabb, LiLan2024_simple} achieve convergence for locally smooth convex functions $f$, which are highly dependent on the boundedness of the iterates. 
In this section, we show that we can achieve boundedness with a locally smooth function $f$ and the same convergence rate, if we consider a small perturbation to the step size selection \eqref{eq:step_size_rule} of \adanag.

\begin{theorem} 
    \label{thm:adanag_local_smooth}
    Let $f$ be a locally smooth convex function. 
    Set $s_0$, $\upperboundconstant$ as in \Cref{thm:main_tight}. 
    Suppose $\set{x_k}_{k\ge0}$ is generated by \adanag\ but with \eqref{eq:step_size_rule} replaced by the following rule when $k\ge 3$: 
    \begin{equation}    \label{eq:adanag_r_step_size_rule} 
        \ssz_{k+1} = 
            \min\set{ \frac{\alpha_{k}}{\alpha_{k+1}}  \ssz_{k},  \,\, \frac{ \alpha_{k}^2 }{ { \alpha_{k+1} + \alpha_{k}^2 (1+10^{-6}) } } \frac{1}{L_{k+1}}  },\qquad \text{if }  k\ge 3.
    \end{equation} 
    Then the sequence $\set{x_k}_{k\ge0}$ is bounded.  
    Let \( \bar{R} > 0 \) such that \( x_k \in \bar{B}_{\bar{R}}(x_\star) \). 
    Define $L$ as a smoothness parameter of $f$ on 
    $\bar{B}_{3\bar{R}}(x_\star)\cup\bar{B}_{3\|\tilde{x}_0-x_\star\|}(x_\star)$, 
    where $\tilde{x}_0$ is the vector in \eqref{eq:initial_guess}. 
    Then the convergence rates stated in \Cref{thm:main_tight} continue to hold with the (local) smoothness parameter \( \localupperboundL \), \ie, 
    \begin{equation}    \label{eq:adanag_convergence_rate}
        \begin{aligned}
        f(x_k) - f_\star
         &\le \frac{5.5 \localupperboundL}{  \theta_{k+2}^2}  \upperboundconstant 
         = \cO\pr{ \frac{\localupperboundL}{k^2} } 
         , \qquad 
         \min_{i\in\set{1,\dots, k}} \norm{\nabla f(x_i)}^2
        \le \frac{120 \localupperboundL^2}{\sum_{i=1}^{k}  \theta_{i+2}^2}  \upperboundconstant 
        = \cO\pr{ \frac{\localupperboundL^2}{k^3} } .
        \end{aligned}
    \end{equation}
\end{theorem}
We note that the value $10^{-6}$ in~\eqref{eq:adanag_r_step_size_rule} is negligibly small compared to $\alpha_k^2$ and $\alpha_{k+1}$.  
To check this point, recall that $\alpha_k$ is increasing for $k \ge 1$ and $\lim_{k \to \infty} \alpha_k = 0.5$, as stated in \Cref{lem:theta_properties}.  
Also, from numerical calculation, we find that $\alpha_1 \approx 0.3182$.  
Therefore, we can view this method as a perturbed version of \adanag, with a negligibly small relaxation in the coefficient of $\frac{1}{L_{k+1}}$. 

The selection of the value $10^{-6}$ is simply an arbitrary example of a small number, chosen to help clear understanding. This choice can be generalized, as stated in Proposition~\ref{prop:adanag_local_smooth}.

\begin{proposition} 
    \label{prop:adanag_local_smooth} 
    Choose $N\ge3$. Define $\alpha_k$ as in \eqref{eq:theta}, and define 
    \begin{equation}    \label{eq:epsilon_bar}
        \bar{\epsilon}_N = \pr{\frac{1}{\alpha _2^2}+\frac{1}{\alpha _3}-\frac{1}{\alpha _N^2} } \alpha _{N+1} - 1.
    \end{equation}
    Take an arbitrary $\epsilon \in (0,\bar{\epsilon}_N)$. 
    Suppose $\set{x_k}_{k\ge0}$ is generated by \adanag\ but with \eqref{eq:step_size_rule} replaced by 
    \begin{equation}    \label{eq:step_size_rule_relaxed}
        \ssz_{k+1} = 
            \min\set{ \frac{\alpha_{k}}{\alpha_{k+1}}  \ssz_{k},  \,\, \frac{ \alpha_{k}^2 }{ { \alpha_{k+1} + \alpha_{k}^2 (1+\epsilon) } } \frac{1}{L_{k+1}}  },\qquad \text{if }  k\ge N.
    \end{equation}
    Then the sequence $\set{x_k}_{k\ge0}$ is bounded.  
    Define $L$ as in \Cref{thm:adanag_local_smooth}. Then the inequalities \eqref{eq:adanag_convergence_rate} hold. 
\end{proposition}

We postpone the proof of Proposition~\ref{prop:adanag_local_smooth} to \Cref{appendix:proof_of_local_extension}.  
Here, we first assume Proposition~\ref{prop:adanag_local_smooth} and prove \Cref{thm:adanag_local_smooth}.
\begin{proof} [Proof of \Cref{thm:adanag_local_smooth}]
    Let $N = 3$.  
    Using numerical calculation, we check that $\bar{\epsilon}_3 \approx 0.399>10^{-6}$.  
    Therefore, we obtain the desired conclusion by Proposition~\ref{prop:adanag_local_smooth}.
\end{proof}

Note that when $\epsilon = 0$ or $N = \infty$, \eqref{eq:step_size_rule_relaxed} reduces to the original step size update rule \eqref{eq:step_size_rule} of \adanag.  
Therefore, the variant in Proposition~\ref{prop:adanag_local_smooth} resembles \adanag\ more closely as $\epsilon$ becomes smaller and $N$ becomes larger.  
Nevertheless, regardless of how small $\epsilon$ is or how large $N$ becomes, Proposition~\ref{prop:adanag_local_smooth} tells us that we still achieve the same convergence rates for locally smooth convex $f$, possibly with a larger local smoothness parameter $\localupperboundL$.

\subsubsection{Proof of Proposition~\ref{prop:adanag_local_smooth}} \label{appendix:proof_of_local_extension}

As mentioned at the beginning of \Cref{sec:AdaNAG_anaylsis}, the statements in the convergence proof of \adanag\ that require global smoothness are only Corollary~\ref{lemma:stepsize_lowerbound_L_smooth} and Proposition~\ref{prop:adanag_gradient_norm}.  
Since the reason that Proposition~\ref{prop:adanag_gradient_norm} requires global smoothness is to leverage Corollary~\ref{lemma:stepsize_lowerbound_L_smooth}, our goal is to extend Corollary~\ref{lemma:stepsize_lowerbound_L_smooth} to the locally smooth case. 

Recall that Corollary~\ref{lemma:stepsize_lowerbound_L_smooth} is about the upper bound of $L_k$.  
Unlike the globally smooth case, when \( f \) is only locally smooth, the boundedness of the iterates becomes crucial for obtaining an upper bound on \( L_k \). 
To establish the boundedness of the iterates, we first prove a statement similar to \Cref{lemma:lyapunov_analysis}, but with a sharper bound involving the coefficient \( s_k^2 B_k \) on the right-hand side.

\begin{lemma}   \label{lemma:lyapunov_analysis_locally}
    Suppose $\set{x_k}_{k\ge0}$ is generated by \adanag\ but with \eqref{eq:step_size_rule} replaced by  \eqref{eq:step_size_rule_relaxed} for some $N\ge3$.   
    Let $f$ be a locally smooth convex function. 
    Define \( V_k \), \( A_k \), and \( B_k \) as in Theorem \ref{lemma:lyapunov_analysis}. 
    Then there exists $\delta>0$ such that 
    \begin{equation*}
        V_{k+1} - V_k 
        \le - \frac{\delta}{2} s_k^2  B_k \norm{ \nabla f(x_k) }^2, \qquad \forall k\ge N.
    \end{equation*}
\end{lemma}

\begin{proof}   
    Using the same argument as in \Cref{lem:ineq:AdaNAG_generalized_stepsize}, from \eqref{eq:AdaNAG_AB_1}, we have:
    \begin{equation*} 
        \pr{ \frac{A_k}{B_k} + \rsBAreciprocal +\epsilon }^{-1}
        \AlphaDefinitionUsed{=} \pr{ \frac{A_k}{B_k} + 1+\epsilon }^{-1}
        \ThetaProprtyUsed{\ge} \pr{ \frac{\alpha_{k+1}}{\alpha_{k}^2} + 1 +\epsilon}^{-1}
        = \frac{ \alpha_{k}^2 }{ \alpha_{k+1} + \alpha_{k}^2 (1+\epsilon) }.
    \end{equation*}
    Therefore, for \( k \ge N \), using \eqref{eq:step_size_rule_relaxed}, we obtain a tighter inequality compared to \eqref{ineq:AdaNAG_generalized_stepsize-c}:
    \begin{equation}    \label{ineq:AdaNAG_generalized_stepsize_tight} 
    \tag{\ref{ineq:AdaNAG_generalized_stepsize-c}'}
         \ssz_{k+1} \le \pr{ \frac{A_k}{B_k} + \rsBAreciprocal + \epsilon }^{-1} \frac{1}{L_{k+1}}.
    \end{equation}          
    Next, for notational simplicity, we define:
    \begin{equation}    \label{eq:delta_k_definition}
        \delta_k = \epsilon  s_{k+1} L_{k+1} \pr{1 - \rsBAreciprocal \ssz_{k+1}  L_{k+1} }^{-1}.
    \end{equation}
    With \( Q_k \) defined as in \Cref{lem:determinant}, we now aim to prove a tighter inequality than \eqref{ineq:determinant}, stated below:  
    \begin{equation}    \label{ineq:determinant_tight} 
    \tag{\ref{ineq:determinant}'}
        Q_k \le \begin{cases}
            - \frac{1}{2} s_k^2 B_k \norm{ \nabla f(x_k) }^2 & \text{ if } s_kL_{k+1} \le 1 \\ 
            - \pr{ \frac{1}{2} \frac{\ssz_{k+1}}{L_{k+1}} A_k  +  \frac{\delta_k}{2} s_k^2 B_k  } \norm{ \nabla f(x_k) }^2 & \text{ if } s_kL_{k+1} \ge 1.
        \end{cases}
    \end{equation}
    Note that the difference from \eqref{ineq:determinant} is the additional term \( \frac{\delta_k}{2} s_k^2 B_k \| \nabla f(x_k) \|^2 \) when \( s_k L_{k+1} \ge 1 \).
    
    We first show that the above inequality implies our desired conclusion. Suppose \eqref{ineq:determinant_tight} is true. Recall that $\lim_{k\to\infty} \alpha_k = \frac{1}{2}$ from \Cref{lem:theta_properties}. 
    Thus \( \frac{\alpha_{k}}{\alpha_{k+1}} \) and \( \frac{ \alpha_{k}^2 }{ \alpha_{k+1} + \alpha_{k}^2 (1+\epsilon) } \) have positive limits, namely \( 1 \) and \( \frac{1}{3+\epsilon} \) respectively, thus there exists \( \delta \in (0,1) \) such that  
    \[
        \epsilon \min\left\{ \frac{\alpha_{k}}{\alpha_{k+1}}, \, \frac{ \alpha_{k}^2 }{ \alpha_{k+1} + \alpha_{k}^2 (1+\epsilon) } \right\} > \delta ,\quad \forall k\ge N.
    \]
    For such $\delta$, when $s_kL_{k+1}\ge1$, recalling the definition of $\delta_k$ in \eqref{eq:delta_k_definition}, we have:
    \begin{equation*}
        \begin{aligned}
        \delta_k
        &\ge \epsilon s_{k+1}L_{k+1} 
        = \epsilon \min\set{ \frac{\alpha_{k}}{\alpha_{k+1}}  \ssz_{k}L_{k+1},  \,\, \frac{ \alpha_{k}^2 }{ \alpha_{k+1} + \alpha_{k}^2 (1+\epsilon) } }  
        > \delta, \qquad \forall k\ge N.
        \end{aligned}
    \end{equation*}
    From the same argument as in \Cref{lemma:lyapunov_analysis}, we have \( V_{k+1} - V_k \le Q_k \). Therefore, we obtain:
    \begin{equation*}
        V_{k+1} - V_k 
        \le Q_k 
        \le -\frac{1}{2} \min\set{ 1, \delta_k } s_k^2 B_k \norm{ \nabla f(x_k) }^2 
        \le -\frac{\delta}{2} s_k^2 B_k \norm{ \nabla f(x_k) }^2, \qquad \forall k\ge N.
    \end{equation*}
    It now remains to show \eqref{ineq:determinant_tight}. 
    The overall argument follows the same structure as the proof of \Cref{lem:determinant}, except that we use a slightly tighter argument in the case \( s_k L_{k+1} \ge 1 \).  
    With the same argument of \Cref{lem:determinant}, showing \eqref{ineq:determinant_tight} reduces to showing that the discriminant of the quadratic form is nonpositive: 
    \begin{equation*} 
        \begin{aligned}
            \pr{ 1 - \ssz_{k} L_{k+1} }^2 -  \pr{1 - \rsBAreciprocal  \ssz_{k+1}  L_{k+1} } \frac{B_k}{A_k} \frac{ \ssz_{k}^2 }{ \ssz_{k+1}} L_{k+1} \pr{ 1 - \delta_k } &\le 0, \qquad \text{ if } s_kL_{k+1} \ge 1.
        \end{aligned}
    \end{equation*}
    Note that from \eqref{ineq:AdaNAG_generalized_stepsize_tight}, we have $\frac{1}{\ssz_{k+1}L_{k+1}} \ge \frac{A_k}{B_k} + \rsBAreciprocal + \epsilon$.  
    Together with the definition of $\delta_k$ in \eqref{eq:delta_k_definition}, we can prove the above inequality as follows:
    \begin{equation*}    
        \begin{aligned}
        &\pr{1 - \rsBAreciprocal  \ssz_{k+1}  L_{k+1} } \frac{B_k}{A_k} \frac{ \ssz_{k}^2 }{ \ssz_{k+1}} L_{k+1} \pr{ 1 - \delta_k } 
        = \pr{ \frac{1}{\ssz_{k+1}  L_{k+1}} -  \rsBAreciprocal }  \frac{ B_k }{ A_k } \ssz_{k} ^2 L_{k+1}^2  - \epsilon \frac{ B_k }{ A_k }  \ssz_{k}^2 L_{k+1}^2  \\ 
        &\ge  \frac{ B_k }{ A_k } \pr{ \frac{1}{\ssz_{k+1}  L_{k+1}}  -  \rsBAreciprocal  } \ssz_{k}^2 L_{k+1}^2 - \epsilon \frac{ B_k }{ A_k }  \ssz_{k}^2 L_{k+1}^2 + 2 \pr{ 1-\ssz_{k}L_{k+1} } \\
        &\ge  \frac{ B_k }{ A_k } \pr{ \frac{A_k}{B_k} + \epsilon  } \ssz_{k}^2 L_{k+1}^2 - \epsilon \frac{ B_k }{ A_k }  \ssz_{k}^2 L_{k+1}^2 + 2 \pr{ 1-\ssz_{k}L_{k+1} } 
            = 2 +  \ssz_{k}^2 L_{k+1}^2 - 2\ssz_{k}L_{k+1}
            \ge (1 - \ssz_{k}L_{k+1})^2.
        \end{aligned}
    \end{equation*}
    This concludes the proof. 
\end{proof}

We now establish the boundedness of the iterates by leveraging \Cref{lemma:lyapunov_analysis_locally}.

\begin{proposition} 
    \label{prop:adanag_boundedness_locally}
    Let $f$ be a locally smooth convex function. 
    Suppose $\set{x_k}_{k\ge0}$ is generated by \adanag\ but with \eqref{eq:step_size_rule} replaced by  \eqref{eq:step_size_rule_relaxed} for some $N\ge3$.     
    Define $R_k = \sqrt{2V_{-1}} + \sum_{i=0}^{k-1} s_k \norm{ \nabla f(x_i) }$, where  \( V_{-1} \) defined as in Proposition~\ref{lemma:lyapunov_analysis_k0}. 
    Then, $\Rlimit = \lim_{k\to\infty} R_k$ exists and    
    \begin{equation}    \label{eq:boundedness_of_iterate_local_smooth}
        x_k \in \bar{B}_{\Rlimit}(x_\star), \qquad \forall k\ge0.
    \end{equation}
\end{proposition}

\begin{proof}
    By summing the inequality obtained in \Cref{lemma:lyapunov_analysis_locally}, we have $\frac{\delta}{2} \sum_{k=N}^{\infty} s_k^2 B_k  \norm{\nabla f(x_k)}^2 \le V_N$.  
    Recall that from \eqref{eq:A_B_AdaNAG} we know $B_k = \alpha_k^2 \theta_{k+2}^2$ for $k\ge N$. Since $\theta_{k+2} \ge \frac{1}{2} (k+4) \ge \frac{1}{2}k$ and $\alpha_k \ge \alpha_N$ for $k\ge N$ by \Cref{lem:theta_properties}, we have 
    \[
        \sum_{k=N}^{\infty} s_k^2 k^2 \norm{\nabla f(x_k)}^2 
        \le \frac{4}{\alpha_N^2} \sum_{k=N}^{\infty} s_k^2 B_k \norm{\nabla f(x_k)}^2
        \le \frac{4}{\alpha_N^2} \frac{2}{\delta } V_N <\infty.
    \] 
    Therefore, \( \sum_{k=1}^{\infty} s_k^2 k^2 \| \nabla f(x_k) \|^2 < \infty \). 
    Now, from the Cauchy–Schwarz inequality in the $\ell^2$ space, 
    \begin{equation*}
        \sum_{k=1}^{\infty} s_k \norm{ \nabla f(x_k) }
        = \sum_{k=1}^{\infty} \pr{ \frac{1}{k} \times s_k k \norm{ \nabla f(x_k) } }
        \le \pr{ \sum_{k=1}^{\infty} \frac{1}{k^2} }^{\frac{1}{2}} \pr{ \sum_{k=1}^{\infty} s_k^2 k^2\norm{ \nabla f(x_k) }^2 }^{\frac{1}{2}} < \infty.
    \end{equation*}
    This implies that $R_k = \norm{x_0 - x_\star} + \sum_{i=0}^{k-1} s_i \norm{\nabla f(x_i)} < \infty$. Since $R_k$ is a nondecreasing sequence, we conclude that $R_k$ converges. 

    To prove \eqref{eq:boundedness_of_iterate_local_smooth}, since $\bar{B}_{R_k}(x_\star) \subset \bar{B}_{\Rlimit}(x_\star)$, it suffices to show $x_k \in \bar{B}_{R_k}(x_\star)$ for all $k\ge0$. 
    We prove this by induction. 
    By definition, $x_0 \in \bar{B}_{R_0}(x_\star)$ is clear. Now, assume $x_k \in \bar{B}_{R_k}(x_\star)$ is true, and we want to prove $x_{k+1} \in \bar{B}_{R_{k+1}}(x_\star)$. 
    From \eqref{eq:AdaNAG}, we have $y_{k+1} = x_k - s_k \nabla f(x_k)$.  
    Hence,
    \begin{equation*}
        \norm{ y_{k+1} - x_\star } 
        \le \norm{ x_{k} - x_\star } + s_k \norm{ \nabla f(x_k) }
        \le R_k + s_k \norm{ \nabla f(x_k) } = R_{k+1}.
    \end{equation*}
    From \Cref{lemma:lyapunov_analysis}, we have $\norm{z_{k+1} - x_\star} \le \sqrt{2V_{k}} \le \sqrt{2V_{-1}} \le R_{k+1}$. 
    Using \eqref{eq:AdaNAG} again, 
    we have:
    \begin{equation*}
        \norm{ x_{k+1} - x_\star } 
        \le \pr{ 1 - \frac{1}{\theta_{k+3}} } \norm{ y_{k+1} - x_\star } + \frac{1}{\theta_{k+3}} \norm{ z_{k+1} - x_\star }
        \le R_{k+1}.
    \end{equation*}
    Therefore $x_{k+1} \in \bar{B}_{R_{k+1}}(x_\star)$, and this completes the proof.
\end{proof}

The following lemma is an extension of \Cref{lemma:stepsize_lowerbound} to the relaxed step size \eqref{eq:step_size_rule_relaxed}. 
The definition of $\bar{\epsilon}_N$ in \eqref{eq:epsilon_bar} plays a crucial role in keeping the coefficient in front of $S_k$ the same.

\begin{lemma} \label{lemma:stepsize_lowerbound_locally_smooth}
    Suppose $\set{x_k}_{k\ge0}$ is generated by \adanag\ but with \eqref{eq:step_size_rule} replaced by  \eqref{eq:step_size_rule_relaxed} for some $N\ge3$.   
    Let $f$ be a locally smooth convex function.  
    Set $s_0$ as in \Cref{lemma:stepsize_lowerbound} and define $S_k$ as \eqref{eq:large_S}. 
    Then 
    \begin{equation*}
        \ssz_k \ge \frac{ \alpha_{2}^2 \alpha_{3}}{\alpha_{3} + \alpha_{2}^2} \frac{1}{\alpha_{k}}  S_{k}, \qquad \forall k\ge 1.
    \end{equation*} 
\end{lemma}

\begin{proof}    
    Note that $\epsilonbound_N$ in \eqref{eq:epsilon_bar} satisfies the equality $\frac{ \alpha_{N}^2 \alpha_{N+1} }{ \alpha_{N+1} + \alpha_{N}^2 (1+\epsilonbound_N) } = \frac{\alpha_{2}^2 \alpha_{3}}{\alpha_{3} + \alpha_{2}^2}$.  
    Since \( \epsilon \mapsto \frac{ \alpha_{N}^2 \alpha_{N+1} }{ \alpha_{N+1} + \alpha_{N}^2 (1+\epsilon) } \) is decreasing in \( \epsilon \) when $\epsilon>0$, we have:
    \begin{equation*} \qquad\qquad\qquad 
        \frac{ \alpha_{N}^2 \alpha_{N+1} }{ \alpha_{N+1} + \alpha_{N}^2 (1+\epsilon) } > \frac{\alpha_{2}^2 \alpha_{3}}{\alpha_{3} + \alpha_{2}^2}, \qquad \forall \epsilon \in (0,\epsilonbound_N).
    \end{equation*}
    Note that when $1\le k\le N-1$, the inequality $\ssz_k \ge \frac{ \alpha_{2}^2 \alpha_{3}}{\alpha_{3} + \alpha_{2}^2} \frac{1}{\alpha_{k}} S_{k}$ is immediate from \Cref{lemma:stepsize_lowerbound}. 
    In particular, the inequality is true for $k=N-1$. 
    Now, we proceed by induction for the case \( k \ge N-1 \). Repeating the same argument as in \Cref{lemma:stepsize_lowerbound}, we can show that for every fixed \( \epsilon > 0 \), the function \( k \mapsto \frac{ \alpha_{k}^2 \alpha_{k+1} }{ \alpha_{k+1} + \alpha_{k}^2(1+\epsilon) } = \frac{ \alpha_{k+1} }{ \frac{\alpha_{k+1}}{\alpha_{k}^2} + 1+\epsilon } \) is increasing. 
    Therefore, we have \vspace{-1mm}
    \begin{equation*}
         \frac{ \alpha_{k}^2 \alpha_{k+1} }{ { \alpha_{k+1} + \alpha_{k}^2(1+\epsilon) } }
         \ge  \frac{ \alpha_{k}^2 \alpha_{k+1} }{ { \alpha_{k+1} + \alpha_{k}^2(1+\epsilonbound_N) } }
         \ge  \frac{ \alpha_{N}^2 \alpha_{N+1} }{ { \alpha_{N+1} + \alpha_{N}^2(1+\epsilonbound_N) } }
         \ge  \frac{\alpha_{2}^2 \alpha_{3}}{\alpha_{3} + \alpha_{2}^2}, \qquad \forall k\ge N.
    \end{equation*}
     Applying the above inequality and the induction hypothesis $\ssz_k \ge \frac{ \alpha_{2}^2 \alpha_{3}}{\alpha_{3} + \alpha_{2}^2} \frac{1}{\alpha_{k}} S_{k}$, we have 
    \begin{equation*}
        \begin{aligned}
            \ssz_{k+1} 
        &= \min\set{ \frac{\alpha_{k}}{\alpha_{k+1}} \ssz_{k},  \,\, \frac{ \alpha_{k}^2 }{ { \alpha_{k+1} + \alpha_{k}^2(1+\epsilon) } } \frac{1}{L_{k+1}}  } \\ 
        &\ge \min\set{    \frac{\alpha_{2}^2\alpha_{3}}{\alpha_{3} + \alpha_{2}^2} \frac{1}{\alpha_{k+1}} S_{k},  \,\,  \frac{\alpha_{2}^2\alpha_{3}}{\alpha_{3} + \alpha_{2}^2} \frac{1}{\alpha_{k+1}} \frac{1}{L_{k+1}}  }  
        =  \frac{\alpha_{2}^2\alpha_{3}}{\alpha_{3} + \alpha_{2}^2} \frac{1}{\alpha_{k+1}} S_{k+1}.
        \end{aligned}
    \end{equation*}
    We conclude the desired result by induction.      
\end{proof}

The following corollary is our promised goal, an extension of Corollary~\ref{lemma:stepsize_lowerbound_L_smooth} to a locally smooth function $f$. 
The purpose of the relaxation \eqref{eq:step_size_rule_relaxed} and Proposition~\ref{prop:adanag_boundedness_locally} was to prove this corollary.

\begin{corollaryL}  \label{lemma:stepsize_constant_lowerbound_locally_smooth}
    Suppose $f$ is a locally smooth convex function and define
    $\set{s_k}_{k\ge0}$ as in \Cref{lemma:stepsize_lowerbound_locally_smooth}. Let $\bar{R}$ be as in Proposition~\ref{prop:adanag_boundedness_locally}, and define $\localupperboundL$ as in \Cref{thm:adanag_local_smooth}. 
    Then 
    \begin{equation*}
        \ssz_k \ge \frac{ \alpha_{2}^2 \alpha_{3}}{\alpha_{3} + \alpha_{2}^2} \frac{1}{\alpha_{k}} \frac{1}{\localupperboundL}, \qquad \forall k\ge 1.
    \end{equation*}
\end{corollaryL}

\begin{proof}
     Note that we have $x_k \in \bar{B}_{\bar{R}}(x_\star)$ for all $k \ge 0$ from Proposition~\ref{prop:adanag_boundedness_locally}.  
     Since $\tilde{x}_0 \in \bar{B}_{\|\tilde{x}_0 - x_\star\|}(x_\star)$, we obtain the conclusion from \Cref{lemma:local_smoothness}.
\end{proof}

Since we have established a lower bound for \( s_k \), the proof of Proposition~\Cref{prop:adanag_local_smooth} follows immediately by repeating the previous arguments. 
\begin{proof} [Proof of Proposition~\ref{prop:adanag_local_smooth}]
    We obtain the boundedness of \( \{x_k\}_{k\ge0} \) from Proposition~\ref{prop:adanag_boundedness_locally}. 
    By applying \Cref{lemma:stepsize_lowerbound_locally_smooth} and Corollary~\ref{lemma:stepsize_constant_lowerbound_locally_smooth} in place of \Cref{lemma:stepsize_lowerbound} and Corollary~\ref{lemma:stepsize_lowerbound_L_smooth}, and repeating the same arguments used in the proof of \Cref{thm:main_tight}, we obtain the convergence rate \eqref{eq:adanag_convergence_rate}. 
\end{proof}

\subsubsection{Boundedness of \adanag\ for globally smooth $f$}

As an independent and interesting observation, we show that \adanag\ also achieves the boundedness of the iterates for globally smooth $f$, using an argument similar to that of Proposition~\ref{prop:adanag_boundedness_locally}.  
This implies that \( L_k \) is bounded by some local smoothness parameter smaller than the global smoothness constant $L$, and thus we actually have a tighter guarantee than \Cref{thm:main_tight}. 

\begin{lemma} 
\label{prop:adanag_boundedness}
    Let $f$ be an $L$-smooth convex function. 
    Suppose $\set{x_k}_{k\ge0}$ is generated by \adanag. 
    Define $R_k$ as in Proposition~\ref{prop:adanag_boundedness_locally}. 
    Then, $\Rlimit = \lim_{k\to\infty} R_k$ exists and $x_k \in \bar{B}_{\Rlimit}(x_\star)$ for all $k\ge0$. 
\end{lemma}

\begin{proof}
    Recalling the proof of Proposition~\ref{prop:adanag_boundedness_locally}, it suffices to show that $\sum_{i=1}^{\infty} s_i^2 i^2 \norm{ \nabla f(x_i) }^2 < \infty$. 
    Considering  $k \to \infty$ in \eqref{ineq:z_converges} and using the fact 
    $\theta_{i+3} \geq \frac{1}{2}(i+5)$ for $i\ge1$ from \Cref{lem:theta_properties}, we obtain $\frac{1}{4} \sum_{i=1}^{\infty} i^2 \norm{ \nabla f(x_i) }^2 < \infty$. 
    From \Cref{lem:theta_properties}, we know that $s_{i+1} \le s_{i}$ holds for $i\ge1$. 
    Therefore, 
    \[
        \sum_{i=1}^{\infty} s_i^2 i^2 \norm{ \nabla f(x_i) }^2
        \le 
        s_1^2 \sum_{i=1}^{\infty} i^2 \norm{ \nabla f(x_i) }^2  < \infty.
    \]
    The rest of the proof can be obtained with the same argument of Proposition~\ref{prop:adanag_boundedness_locally}. 
\end{proof}

\subsection{Numerical values used in the proof of Theorem~\ref{thm:main_tight} and Proposition~\ref{thm:main_simple}}

The numerical values used in the proof of \Cref{thm:main_tight} and Proposition~\ref{thm:main_simple} are provided in Table \ref{tab:numerical_values}. Recall that we denote $r_0=\frac{\theta_3(\theta_3-1)}{\theta_2} \frac{1}{\alpha_{0}} \frac{ \alpha_{2}^2 \alpha_{3}}{\alpha_{3} + \alpha_{2}^2}$.

\begin{table}[H]
    \centering
    \begin{tabular}{c|c c c c c c} 
        \toprule
        \textbf{$\theta_k$} & {$\frac{\alpha_0}{\alpha_{1}}\frac{\theta_2}{\theta_3(\theta_3-1)} $} & {$r_0$} & {$\frac{2(\alpha_{3} + \alpha_{2}^2)}{ \alpha_{2}^2 \alpha_{3}}$} & {$r_{0}^2 \talpha_0 \pr{   \oldalpha_0 + \talpha_0 } \theta_{2}^2$} & {$\frac{1}{r_0 \pr{   \oldalpha_0 + \talpha_0 } \theta_{2}}$}   \\ 
        \midrule
        \textbf{AdaNAG} & 0.6745 & 0.4255 & 21.9032 & 0.3045 & 1.4424    \\
        \textbf{\adanagS} & 0.8399 & 0.3363 & 23.6 & 0.1544 & 2.0578    \\
        \bottomrule
    \end{tabular}
    \caption{Numerical values (rounded to the fourth decimal place) used in the proof of Theorems \ref{thm:main_tight} and \ref{thm:main_simple}.}
    \label{tab:numerical_values}
\end{table}

\section{Omitted details in \Cref{sec:adagd}}

\subsection{Remaining Proof of Proposition \ref{lemma:gd_lyapunov_analysis}} \label{appendix:remaining_proof_gd}

By similar calculation as \eqref{eq:z_square_diff}, for $k\ge-1$ we have
\begin{equation*}
    \begin{aligned}
        \frac{1}{2} \norm{x_{k+2} - x_\star}^2 - \frac{1}{2} \norm{x_{k+1} - x_\star}^2 
        &= -s_{k+1} \inner{ \nabla f(x_{k+1}) }{  x_{k+1} - x_\star  }
            + \frac{s_{k+1}^2}{2} \norm{ \nabla f(x_{k+1}) }^2.
    \end{aligned}
\end{equation*}
When $k=-1$, we can verify \eqref{ineq:gd_ineq_goal} as follows:
\begin{equation*}
    \begin{aligned}
        &V_0 - V_{-1} \\
        &= s_1 A_0 (f(x_0) - f_\star) + \frac{1}{2} \norm{x_{1} - x_\star}^2 
            + \frac{1}{2} s_{0}^2 B_0 \norm{ \nabla f(x_0) }^2
        - \pr{ \frac{1}{2} \norm{x_{0} - x_\star}^2 
            + \frac{1}{2} s_0^2 \pr{B_0+1} \norm{ \nabla f(x_0) }^2  }  \\
        &= \pr{ s_0 - s_1 A_0 } \pr{ f_\star - f(x_0) }
            + s_0 \pr{ f(x_0) - f_\star - \inner{ \nabla f(x_{0}) }{  x_{0} - x_\star  } } .
    \end{aligned}
\end{equation*} 
Now suppose $k\ge0$. 
We first consider the case $L_{k+1} \ne 0$. It is easy to check 
\begin{equation*}
    \begin{aligned}
        &\ssz_{k+2} A_{k+1}(f(x_{k+1}) - f_\star) - \ssz_{k+1} A_k(f(x_k) - f_\star) \\
        &= (\ssz_{k+1} A_k + s_{k+1} - \ssz_{k+2} A_{k+1}) (f_\star - f(x_{k+1})) + s_{k+1} ( f(x_{k+1}) - f_\star) \\ &\quad
        + \ssz_{k+1} A_k \underbrace{  \pr{ f(x_{k+1}) - f(x_k) - \inner{\nabla f(x_{k+1})}{ x_{k+1} - x_k} +\frac{1}{2L_{k+1}}  \norm{ \nabla f(x_{k+1}) - \nabla f(x_{k}) }^2 }}_{=0}  \\ &\quad
        + \ssz_{k+1} A_k \inner{\nabla f(x_{k+1})}{ x_{k+1} - x_k} - \frac{\ssz_{k+1} A_k }{2L_{k+1}}  \norm{ \nabla f(x_{k+1}) - \nabla f(x_{k}) }^2.
    \end{aligned}
\end{equation*}
Applying $x_{k+1} - x_k = -\ssz_{k} \nabla f(x_{k+1})$ to the last line, 
we obtain
\begin{equation}    \label{eq:gd_lyapunov_core_eq}
    \begin{aligned}
        V_{k+1} - V_k 
        &= (\ssz_{k+1} A_k + s_{k+1} - \ssz_{k+2} A_{k+1}) (f_\star - f(x_{k+1})) \\ &\quad
            + s_{k+1} ( f(x_{k+1}) - f_\star - \inner{ \nabla f(x_{k+1}) }{  x_{k+1} - x_\star  } ) +  Q_k,
    \end{aligned}
\end{equation} 
where 
\begin{equation*}
    \begin{aligned}
         Q_k = -\frac{\ssz_{k+1} A_k}{2L_{k+1}} &\bigg[ \pr{ 1 - \frac{B_{k+1}+1}{A_k} s_{k+1} L_{k+1} } \norm{ \nabla f(x_{k+1}) }^2  \\ &             
            - 2\pr{ 1 - \ssz_{k} L_{k+1} } \inner{\nabla f(x_{k+1})}{ \nabla f(x_{k}) }  +  \pr{ 1 + \frac{B_{k}}{A_k}  \frac{\ssz_{k}^2}{s_{k+1}} L_{k+1} } \norm{ \nabla f(x_{k}) }^2 \bigg].
    \end{aligned}
\end{equation*} 
For simplicity, denote $\rsBAgdreciprocal=\frac{B_{k+1}+1}{A_k}$. 
We now repeat the argument in the proof of \Cref{lem:determinant}. 
Since 
\[
    \tilde{\ssz}_k^2 \min\set{A_k, B_k} \le \min \set{\frac{\ssz_{k+1}}{L_{k+1}} A_k, s_k^2 B_k},
\] 
where $\tilde{\ssz}_{k} = \min\big\{s_k, s_{k+1}, \frac{1}{L_{k+1}}\big\}$, 
it suffices to show:
\begin{equation*}    
    \begin{aligned}
        Q_k + \frac{B_k}{2}s_k^2 \norm{ \nabla f(x_k) }^2 &\le 0 
          \qquad \text{ if } s_kL_{k+1} \le 1 \\
        Q_k + \frac{A_k}{2} \frac{\ssz_{k+1}}{L_{k+1}} \norm{ \nabla f(x_k) }^2 &\le 0  \qquad \text{ if } s_kL_{k+1} \ge 1.
    \end{aligned}
\end{equation*} 
First, we focus on the coefficient of the \( \norm{ \nabla f(x_{k+1}) }^2 \) term. 
From \eqref{eq:AdaGD}, we have 
\begin{equation*} 
    1 - \rsBAgdreciprocal s_{k+1} L_{k+1}
    \ge 1 - \pr{ \frac{A_{k}}{B_{k}} + \rsBAgdreciprocal } s_{k+1} L_{k+1}
    \ge 0.
\end{equation*} 
Considering the discriminant of the quadratic form, our goal reduces to showing:

\begin{equation*} 
    \begin{aligned}
        \pr{ 1 - \ssz_{k} L_{k+1} }^2 -  \pr{1 - \rsBAgdreciprocal  \ssz_{k+1}  L_{k+1} } &\le 0 \qquad \text{ if } s_kL_{k+1} \le 1, \\
        \pr{ 1 - \ssz_{k} L_{k+1} }^2 -  \pr{1 - \rsBAgdreciprocal  \ssz_{k+1}  L_{k+1} } \frac{B_k}{A_k} \frac{ \ssz_{k}^2 }{ \ssz_{k+1}} L_{k+1}&\le 0 \qquad \text{ if } s_kL_{k+1} \le 1.
    \end{aligned}
\end{equation*}
Now, we divide into two cases, as we have done in the proof of \Cref{lem:determinant}. 

\begin{itemize}
    \item $\ssz_{k}L_{k+1} \le 1$.  
        From \eqref{ineq:gd_A_B_relation} and \eqref{eq:AdaGD} we have 
        $\rsBAgdreciprocal s_{k+1} L_{k+1} \le \frac{A_{k}}{A_{k-1}+1} s_{k+1} L_{k+1} \le \ssz_{k+1} L_{k+1}$. 
        Again leveraging \eqref{ineq:when_sL_smaller_1}, we conclude the desired inequality.
    \item $\ssz_{k}L_{k+1} \ge 1$. 
    From \eqref{eq:AdaGD} we know $\frac{1}{\ssz_{k+1}L_{k+1}} \ge \frac{A_k}{B_k} + \rsBAgdreciprocal$ holds for $k\ge0$. 
    Therefore, recalling that $1 - 2\ssz_k L_{k+1} \leq 0$, we obtain:
    \begin{equation*}    
        \begin{aligned}
            &\pr{1 - \rsBAgdreciprocal \ssz_{k+1}  L_{k+1} } \frac{ B_k }{ A_k } \frac{ \ssz_{k}^2 }{ \ssz_{k+1} } L_{k+1} 
        =  \frac{ B_k }{ A_k } \pr{ \frac{1}{\ssz_{k+1}L_{k+1}}  - \rsBAgdreciprocal  } \ssz_{k}^2 L_{k+1}^2
        \ge \ssz_{k}^2 L_{k+1}^2
        \ge  \pr{ 1 - \ssz_{k} L_{k+1} }^2.
        \end{aligned}
    \end{equation*} 
\end{itemize}   

Now we move on to the case $L_{k+1} = 0$. Note that this implies $\ssz_{k}L_{k+1} = 0 \le 1$. 
From Corollary~\ref{cor:L_k_bound} we have $\nabla f(x_k)=\nabla f(x_{k+1})$. 
Then from the same calculation done to \eqref{eq:gd_lyapunov_core_eq} but without $\frac{1}{L_{k+1}}\norm{ \nabla f(x_k) - \nabla f(x_{k+1}) }^2$, we have 
\begin{equation*}
    \begin{aligned}
        V_{k+1} - V_k  
        &= (\ssz_{k+1} A_k + s_{k+1} - \ssz_{k+2} A_{k+1}) (f_\star - f(x_{k+1})) \\ &\phantom{=}
            + s_{k+1} ( f(x_{k+1}) - f_\star - \inner{ \nabla f(x_{k+1}) }{  x_{k+1} - x_\star  } )  \\ &\phantom{=}
             + \ssz_{k+1} A_k \pr{ f(x_{k+1}) - f(x_k) - \inner{\nabla f(x_{k+1})}{ x_{k+1} - x_k}  } \\ &\phantom{=}
            - \ssz_{k} \ssz_{k+1} A_k \inner{\nabla f(x_{k+1})}{ \nabla f(x_{k}) } 
            + \frac{1}{2} s_{k+1}^2 \pr{B_{k+1}+1} \norm{ \nabla f(x_{k+1}) }^2  - \frac{1}{2} \ssz_{k}^2 B_{k} \norm{ \nabla f(x_{k}) }^2 \\ 
            &\le- \ssz_{k} \ssz_{k+1} A_k \inner{\nabla f(x_{k+1})}{ \nabla f(x_{k}) }  
        + \frac{1}{2} s_{k+1}^2 \pr{B_{k+1}+1} \norm{ \nabla f(x_{k+1}) }^2  - \frac{1}{2} \ssz_{k}^2 B_{k}  \norm{ \nabla f(x_{k}) }^2 \\ 
        &= \frac{1}{2} s_{k+1} \pr{ s_{k+1} \pr{B_{k+1}+1} - s_{k} A_k } \norm{ \nabla f(x_{k}) }^2 
        -  \pr{ \frac{1}{2} \ssz_{k}\ssz_{k+1} A_k + \frac{1}{2}\ssz_{k}^2 B_k } \norm{ \nabla f(x_{k}) }^2 \\ 
        &\le - \frac{1}{2}\ssz_{k}^2 B_k \norm{ \nabla f(x_{k}) }^2 .
    \end{aligned}
\end{equation*} 
The second inequality holds since the following inequalities follow from \eqref{ineq:gd_A_B_relation} and \eqref{eq:AdaGD}:
\begin{equation*}
    s_{k+1} \pr{B_{k+1}+1} - s_{k} A_k
    \le s_{k+1} \frac{A_{k}^2}{1+A_{k-1}} - s_{k} A_k 
    \le 0.
\end{equation*} 
This concludes the proof. \qed

\subsection{Proof of Lemma~\ref{lemma:gd_stepsize_lowerbound}} \label{appendix:proof_of_gd_stepsize_lowerbound}

Define $S_k$ as in \eqref{eq:large_S}, \ie, $S_{k} = \min\big\{ \frac{1}{L_0}, \frac{1}{L_{1}}, \dots, \frac{1}{L_{k}} \big\}$. 
We first prove the following by induction:
\begin{equation}    \label{lem12-1}
    \ssz_k \ge r S_{k}, \qquad \forall k\ge 1.
\end{equation}
From \eqref{eq:AdaGD}, we have $\ssz_{1} = \min \set{ \frac{1}{A_0} s_{0}, \, r \frac{1}{L_{1}} } = \min \set{ r \frac{1}{L_0}, \, r \frac{1}{L_{1}} } = r S_{1}$.  
Thus \eqref{lem12-1} holds for $k=1$. Assume \eqref{lem12-1} holds for $k$, now we show that it also holds for $k+1$. 
Applying the induction hypothesis $ \ssz_k \ge r S_k $ and \eqref{ineq:AB_condition_for_step_size_lowerbound}, we obtain 
the following inequality, which completes the proof of \eqref{lem12-1}:
\begin{equation*}
    \begin{aligned}
        \ssz_{k+1}  
        &\ge \min \set{ \frac{A_{k-1}+1}{A_{k}} r S_{k}, \pr{ \frac{A_{k}}{B_{k}} + \frac{B_{k+1} +1}{A_{k}} }^{-1} \frac{1}{L_{k+1}} } 
        \ge \min\set{ r S_{k} , \, r \frac{1}{L_{k+1}} } 
        = r S_{k+1}.
    \end{aligned}
\end{equation*}

To prove \eqref{lem12-2}, recall that we have $x_k \in \bar{B}_{R}(x_\star)$ from Proposition~\ref{lemma:gd_lyapunov_analysis} and $\tilde{x}_0 \in \bar{B}_{\| \tilde{x}_0 - x_\star \|}(x_\star)$.  
Therefore, from \Cref{lemma:local_smoothness}, we know that $L_k \leq \gdupperboundL$ for $k \geq 0$, where $\gdupperboundL$ is a  smoothness parameter of $f$ on $\bar{B}_{3\bar{R}}(x_\star)\cup\bar{B}_{3\|\tilde{x}_0-x_\star\|}(x_\star)$. 
Therefore, we have $S_k \geq \frac{1}{\gdupperboundL}$. From \eqref{lem12-1}, we conclude \eqref{lem12-2}. 
\qed
\section{Omitted proofs in Section~\ref{sec:adanag-g}} 

\subsection{Proof of \Cref{thm:adanag_g}} \label{appendix:proof_of_adanag_g}

Similar to the proof of the convergence of \adanag, we state and prove lemmas similar to \Cref{lemma:lyapunov_analysis}, \Cref{lemma:lyapunov_analysis_k0} and \Cref{lemma:stepsize_lowerbound}.

\begin{lemma} \label{lemma:adanag_g_lyapunov_analysis} 
Suppose $\set{x_k}_{k\ge0}$ is generated by \adanagG\ with $\set{\tau_k}_{k\ge0}$, $\set{\alpha_k}_{k\ge0}$ satisfying the assumptions of \Cref{thm:adanag_g}. 
Define Lyapunov function as
\begin{equation*}  
    V_{k} 
    = \ssz_{k+1} A_k  \pr{ f(x_k) - f_\star } 
    +  \frac{1}{2} \ssz_{k}^2  B_k \norm{ \nabla f(x_k) }^2
    + \frac{1}{2} \norm{ z_{k+1} - x_\star }^2, 
    \qquad \forall k\ge -1
\end{equation*} 
where $x_{-1} = x_0$, $s_{-1}=s_0$ and $B_{-1}=B_0 + \alpha_0^2\tau_0^2$.  
Then
\begin{equation*}
    V_{k+1} - V_k \le 
         - \frac{1}{2} \min\set{s_k^2 B_k, \frac{\ssz_{k+1}}{L_{k+1}} A_k } \norm{ \nabla f(x_k) }^2, \qquad \forall k\ge0
\end{equation*}
holds, and  
\begin{equation*}
    V_0\le V_{-1} + \ssz_{0} \talpha_0 \tau_{0} \pr{ f(x_{0}) - f_\star - \inner{ \nabla f(x_{0}) }{ x_{0} - x_\star } }
\end{equation*}
holds, when $f$ is a locally smooth convex function. 
\end{lemma}
 
\begin{proof}
    For $k \geq 0$, the proof follows by repeating the proof argument of \Cref{lemma:lyapunov_analysis}, with $\theta_{k+2}$ replaced by $\tau_k$.     
    The only thing we need to verify is that \eqref{ineq:AdaNAG_generalized_stepsize} holds when $\theta_{k+2}$ is replaced with $\tau_k$. 
    We now verify this point. 
    Starting with \eqref{ineq:AdaNAG_generalized_stepsize-a} and \eqref{ineq:AdaNAG_generalized_stepsize-c}, we can see they follow immediately from the step size rule of \eqref{eq:AdaNAG-G_stepsize}.  
    Next, observe that  
    \begin{equation*}   
        B_{k} 
        = \alpha_{k}^2\tau_{k}^2 \pr{ \frac{(\tau_{k}-1)^2}{\alpha_{k-1}\tau_{k-1}^2} - 1 } 
        = \frac{A_{k-1}^2}{A_{k-2} + \alpha_{k-1}\tau_{k-1}} - \alpha_{k}^2\tau_{k}^2
        \,\, \iff \,\,
        \frac{A_{k-2}+\alpha_{k-1}\tau_{k-1}}{A_{k-1}}  = \frac{A_{k-1}}{B_{k} + \alpha_{k}^2\tau_{k}^2},
    \end{equation*}    
    so the definition of $B_k$ implies $\frac{A_{k-2}+\alpha_{k-1}\tau_{k-1}}{A_{k-1}}  = \frac{A_{k-1}}{B_{k} + \alpha_{k}^2\tau_{k}^2}$. Combining with \eqref{ineq:AdaNAG_generalized_stepsize-a}, we obtain
    \eqref{ineq:AdaNAG_generalized_stepsize-b}.

    Now we move on to $k=-1$. Recalling $x_0=z_0$ and \eqref{eq:z_square_diff}, we have
    \begin{equation*}
        \begin{aligned}
            \norm{ z_{1} - x_\star }^2  - \norm{ x_{0} - x_\star }^2  
            &= - 2\ssz_{0} \talpha_0 \tau_{0} \inner{ \nabla f(x_{0}) }{ x_0- x_\star } + \ssz_{0}^2 \talpha_0^2 \theta_{0}^2 \norm{ \nabla f(x_{0}) }^2.
        \end{aligned}
    \end{equation*}
    Therefore
\begin{equation*}
    \begin{aligned}
        V_{0} - V_{-1} 
        &= \ssz_{1} A_0 \pr{ f(x_{0}) - f_\star } 
        + \frac{1}{2} \norm{ z_{1} - x_\star }^2 
        + \frac{1}{2} \ssz_{0}^2 B_0 \norm{ \nabla f(x_{0}) }^2  
        - \pr{ \frac{1}{2} \norm{x_{0} - x_\star}^2 +  \frac{1}{2} \ssz_{0}^2  B_{-1} \norm{ \nabla f(x_{0}) }^2 } \\
        &= \pr{ \ssz_{0} \talpha_0 \tau_{0} - \ssz_{1} A_0 } \pr{ f_\star - f(x_{0}) } 
        +  \ssz_{0} \talpha_0 \tau_{0} \pr{ f(x_{0}) - f_\star - \inner{ \nabla f(x_{0}) }{ x_{0} - x_\star } }  \\
        &\le \ssz_{0} \talpha_0 \tau_{0} \pr{ f(x_{0}) - f_\star - \inner{ \nabla f(x_{0}) }{ x_{0} - x_\star } }.
    \end{aligned}
\end{equation*}
The inequality follows from $\ssz_{0} \talpha_0 \tau_{0} - \ssz_{1} A_0 = A_0 \pr{ \frac{\alpha_0\tau_0}{A_0} \ssz_{0} - \ssz_{1} } \ge 0$, which holds true by \eqref{eq:AdaNAG-G_stepsize}. This concludes the proof.  
\end{proof}

\begin{lemma} \label{lemma:adanag_g_stepsize_lowerbound}
    Suppose $\set{s_k}_{k\ge0}$ is generated by \adanagG\ with $\set{\tau_k}_{k\ge0}$, $\set{\alpha_k}_{k\ge0}$ satisfying the assumptions in  \Cref{thm:adanag_g}. Set $s_0$ as in \Cref{thm:adanag_g} and 
    let $f$ be an $L$-smooth convex function. 
    Then 
    \begin{equation*}
        \ssz_k \ge \frac{r}{\alpha_k} \frac{1}{\upperboundL}, \qquad \forall k\ge 1.
    \end{equation*} 
\end{lemma}

\begin{proof}
    Define $S_k= \min\big\{ \frac{1}{L_0}, \frac{1}{L_{1}}, \dots, \frac{1}{L_{k}} \big\}$, as in \eqref{eq:large_S}. 
    We first prove the following by induction. 
    \begin{equation}    \label{eq:adanag_g_stepsize_lowerbound_tight}
        \ssz_k \ge \frac{\rzeroAdaNAG }{\alpha_k}  S_{k}, \qquad \forall k\ge 1.
    \end{equation}
    \begin{itemize}
        \item 
        $k=1$.             
        From \eqref{eq:AdaNAG-G_stepsize}, the assumption $s_0 = \frac{A_{0}}{\alpha_{0}\tau_{0}}  \frac{\rzeroAdaNAG}{\alpha_1} \inverLzero$, 
        we have $\ssz_1 \ge \min\set{ \frac{\alpha_{0}\tau_{0}}{A_{0}} s_{0}, \, \frac{\rzeroAdaNAG}{\alpha_1} \frac{1}{L_{1}} } = \frac{\rzeroAdaNAG}{\alpha_1} S_{1}$.
        
        \item 
        $k\ge 2$. 
        From the second inequality of \eqref{ineq:adanag_AB_condition_for_step_size_lowerbound}, we have $\frac{A_{k-1}+\alpha_{k}\tau_{k}}{A_{k}} = \frac{\alpha_k}{\alpha_{k+1}} \frac{\tau_k}{\tau_{k+1}(\tau_{k+1}-1)}\ge \frac{\alpha_k}{\alpha_{k+1}}$. \ 
        Applying the third inequality of \eqref{ineq:adanag_AB_condition_for_step_size_lowerbound}, the induction hypothesis $\ssz_k \ge \frac{\rzeroAdaNAG }{\alpha_k}  S_{k}$, and above inequality, we have  
        \begin{equation*}
            \begin{aligned}
                \ssz_{k+1} 
                &\ge \min \set{ \frac{\alpha_k}{\alpha_{k+1}} \frac{\rzeroAdaNAG }{\alpha_k} S_{k}, \frac{\rzeroAdaNAG}{\alpha_{k+1}} \frac{1}{L_{k+1}} } 
                = \frac{\rzeroAdaNAG}{ \alpha_{k+1}} \min\set{ S_{k} , \, \frac{1}{L_{k+1}} } 
                = \frac{\rzeroAdaNAG}{ \alpha_{k+1}} S_{k+1}.
            \end{aligned}
        \end{equation*}
        We conclude the result of \eqref{eq:adanag_g_stepsize_lowerbound_tight} by induction. 
    \end{itemize}  
     Next, from \Cref{lemma:local_smoothness}, we have $\frac{1}{L_k} \ge \frac{1}{L}$ for $k \ge 1$ and $\frac{1}{L_0}\ge \frac{1}{L}$ by its definition \eqref{eq:initial_guess}. 
    Recalling the definition of $S_k$, we obtain $S_k \ge \frac{1}{\upperboundL}$ for $k \ge 0$.
    We obtain the desired conclusion by \eqref{eq:adanag_g_stepsize_lowerbound_tight}. 
\end{proof}

\begin{proof} [Proof of \Cref{thm:adanag_g}]
    The proof follows from \Cref{lemma:adanag_g_lyapunov_analysis} and \Cref{lemma:adanag_g_stepsize_lowerbound}.  
    From \Cref{lemma:adanag_g_lyapunov_analysis} and \eqref{eq:L_smooth_ineq} we have
    \begin{equation*}
        s_{k+1}A_k (f(x_k) - f_\star )
        \le V_k
        \le \dots
        \le V_0 
        \le V_{-1} -\frac{\ssz_{0} \talpha_0 \tau_{0}}{2L} \norm{ \nabla f(x_0)}^2 = \frac{1}{2} \upperboundconstant,
    \end{equation*}
    and so $f(x_k) - f_\star  \le  \frac{ 1 }{ s_{k+1} A_k } \upperboundconstant$. 
    From \Cref{lemma:adanag_g_stepsize_lowerbound}, we have $s_{k+1}\alpha_{k+1} \ge \frac{\rbound}{\upperboundL}$ for $k\ge0$. 

    The last statement can be proved with a similar argument as in Proposition~\ref{prop:adanag_gradient_norm}. 
    Since $\alpha_k \in (0,1]$, by \Cref{lemma:adanag_g_stepsize_lowerbound} we have $s_k \ge \frac{r}{\upperboundL}$ for $k \ge 1$ . 
    Thus, we have $\min\big\{s_k^2, \frac{s_{k+1}}{L_{k+1}}\big\} \ge \frac{r^2}{\upperboundL^2}$.
    Therefore, from \Cref{lemma:adanag_g_lyapunov_analysis} we obtain
    \begin{equation*}
        V_{k+1} - V_k \le - \frac{r^2}{2\upperboundL^2} \min\set{A_k, B_k} \norm{ \nabla f(x_k) }^2, \qquad \forall k \ge 0.
    \end{equation*}
    Summing up from $0$ to $k-1$, applying $V_0\le \frac{1}{2} \upperboundconstant$, we obtain
    \begin{equation*}
        \pr{ \frac{r^2}{2\upperboundL^2} \sum_{i=1}^{k} \min\set{A_i, B_i} } \min_{i\in\set{1,\dots, k}} \norm{\nabla f(x_i)}^2
        \le \frac{r^2}{2\upperboundL^2} \sum_{i=1}^{k} \min\set{A_i, B_i} \norm{ \nabla f(x_i) }^2 \le \frac{1}{2}\upperboundconstant,
    \end{equation*}
    which implies our desired result. 

    Lastly, from the last inequality of \eqref{ineq:adanag_AB_condition_for_step_size_lowerbound}, we have $\limsup_{k\to\infty} \frac{A_k}{B_k} < \infty$. Therefore, we have $A_k = \cO\pr{\min\set{A_k,B_k}}$, we conclude $\frac{1}{\sum_{i=1}^{k} \min\set{A_i,B_i}} = \cO\pr{ \frac{1}{\sum_{i=1}^{k} A_i} } = \cO\pr{ \frac{1}{\sum_{i=1}^{k} \alpha_i\tau_i^2} }$.
\end{proof}

\subsection{Proof of \Cref{lemma:asymptotic_lower_bound_step_size}} \label{appendix:proof_of_asymptotic_lower_bound_step_size}
 Take $\epsilon>0$. 
For notation simplicity, name 
\[
    r_k^s = \frac{A_{k-1}+\alpha_{k}\tau_{k}}{A_{k}},
    \qquad 
    r_k^L = \pr{ \frac{A_{k}}{B_{k}} + \frac{B_{k+1} + \alpha_{k+1}^2\tau_{k+1}^2}{A_{k}} }^{-1}.
\] 
Since $\lim_{k \to \infty} r_k^L = \bar{r}$, and by the assumption on $\frac{A_{k-1} + \alpha_k \tau_k}{A_k}$, there exists $\NN > 0$ such that 
\begin{equation}    \label{eq:asymptotic_lemma_N}
    k \ge \NN 
    \qquad \Longrightarrow \qquad 
    r_k^L \ge \bar{r} - \epsilon, \quad r_k^s \ge 1.
\end{equation}

Suppose there exists $\MM \ge \NN$ such that $s_\MM \ge (\bar{r} - \epsilon) \frac{1}{L}$. 
Since $r_\MM^s \ge 1$ holds by \eqref{eq:asymptotic_lemma_N} and $\frac{1}{L_{\MM+1}} \ge \frac{1}{L}$ holds by \Cref{lemma:local_smoothness}, this implies
\begin{equation*}
    s_{\MM+1} 
    = \min\set{ r_\MM^s s_\MM, \, r_\MM^{L} \frac{1}{L_{\MM+1}} }
    \ge \min\set{ s_\MM, \, \pr{ \bar{r} - \epsilon } \frac{1}{L_{\MM+1}} }
    \ge \pr{ \bar{r} - \epsilon } \frac{1}{L}.
\end{equation*}
By induction, we can conclude $s_\mm \ge \pr{ \bar{r} - \epsilon } \frac{1}{L}$ for all $\mm \ge \MM$. 
Thus, it suffices to show such $\MM$ exists.  

Proof by contradiction. Assume that such a $\MM$ does not exist, \ie, assume that
\begin{equation}    \label{eq:contradiction assumption}
    \MM \ge \NN 
    \qquad \Longrightarrow \qquad 
    s_{\MM} < (\bar{r} - \epsilon) \frac{1}{L}.
\end{equation}
Suppose $\mm \ge N$. 
Then from \eqref{eq:contradiction assumption}, \eqref{eq:asymptotic_lemma_N} and the fact $\frac{1}{L} \le \frac{1}{L_{\mm+1}}$, we obtain $ s_{\mm+1} < (\bar{r} - \epsilon) \frac{1}{L} \le  r_\mm^L \frac{1}{L_{\mm+1}}$. 
This implies that $s_{\mm+1} \ne r_\mm^L \frac{1}{L_{\mm+1}}$, which means that the first option of \eqref{eq:AdaNAG-G_stepsize} has been activated, and thus $s_{\mm+1} = r_\mm^s s_\mm$ holds. 
Therefore, we conclude that $\mm \ge N$ implies $s_{\mm+1} = r_\mm^s s_\mm$. By applying this fact recursively, we obtain:
\begin{equation*}
    s_\MM 
    = \prod_{\mm=\NN}^{\MM} r_\mm^s s_\mm
    = \frac{\alpha_{\NN}}{\alpha_\MM} \prod_{\mm=\NN}^{\MM} \frac{\tau_{\mm}^2}{\tau_{\mm+1}(\tau_{\mm+1}-1)}, \qquad \forall \MM>\NN.
\end{equation*}
Applying the assumption $\prod_{k=1}^{\infty} \frac{\tau_k^2}{\tau_{k+1}(\tau_{k+1} - 1)} = \infty$, we get $\lim_{\MM \to \infty} s_\MM = \infty$.  
This implies that there exists $\MM$ such that $s_\MM \ge (\bar{r} - \epsilon)\frac{1}{L}$, which contradicts \eqref{eq:contradiction assumption}.  
This completes the proof.
\qed

\subsection{Proof of Corollary~\ref{cor:adanag_12}}    \label{appendix:adanag_12_proof}

Here, we state and prove the result for the generalized parameter choice $\tau_k = \frac{(k+2)+p}{p}$, where $p > 2$. 
\begin{proposition} \label{prop:ananaga_g_p}
    Let $\tau_k = \frac{(k+2)+p}{p}$, where $p>2$. 
    Set other parameters and assumptions as Corollary~\ref{cor:adanag_12}. 
    Then for $\set{x_k}_{k\ge0}$ generated by \adanagG, we have:  
    \begin{equation}    \label{eq:convergence_rate_generalized_p}
        f(x_k) - f_\star 
        \le \frac{1}{s_{k+1}} \frac{p^2 (k+p+3)}{(k+3) (k+4)^2} 
        \upperboundconstant
        = \cO\pr{ \frac{ \upperboundL }{k^2} }, 
        \qquad
        \min_{i\in\set{1,\dots, k}} \norm{\nabla f(x_i)}^2
        \le  
        \cO\pr{ \frac{ \upperboundL^2 }{k^3} }.
    \end{equation}
    Here, 
    $s_k \ge \frac{27}{(p+3) \left(2 p^2+8 p+17\right)} \frac{1}{\upperboundL}$ and \vspace{-1mm} 
    \[
        \liminf_{k\to\infty} s_k \ge \frac{1}{3} \frac{1}{\upperboundL}.  
    \]
\end{proposition}

\begin{proof}
    We first show that \eqref{ineq:adanag_AB_condition_for_step_size_lowerbound} holds.  
Name $r_k^L = \Big( \frac{A_{k}}{B_{k}} + \frac{B_{k+1} + \alpha_{k+1}^2\tau_{k+1}^2}{A_{k}} \Big)^{-1}$.  
Using elementary calculations, we can verify: 
\begin{subequations}
\begin{align}
    \alpha_k &= \frac{(k+3)^2}{2 (k+p+2)^2} \in (0,1] \label{eq:adanag_g_12_alpha} \\
    \frac{\tau_{k}^2}{\tau_{k+1}(\tau_{k+1}-1)} &= \frac{(k+p+2)^2}{(k+3) (k+p+3)} = 1 + \frac{(p-2)(k+p+3)+1}{(k+3) (k+p+3)} \ge 1  \label{eq:adanag_g_12_growth_rate} \\
    \alpha_{k+1} r_k^L  
    &= \frac{k+3}{2 (k+p+3)} \left( 2\pr{\frac{k+p+2}{k+3}}^2+1\right)^{-1}  
    \ge \alpha_{1} r_0^L   
    =: \rbound. \label{eq:adanag_g_12_r}
\end{align}
\end{subequations}
\eqref{eq:adanag_g_12_growth_rate} holds since $p>2$.
\eqref{eq:adanag_g_12_r} can be shown by proving that $\alpha_{k+1} r_k^L$ is an increasing sequence, which follows from the facts that $\frac{k+3}{k+p+3} = 1 - \frac{p}{k+p+3}$ is increasing and $\frac{k+p+2}{k+3} = 1 + \frac{p - 1}{k+3}$ is decreasing. 
Next, note that $\alpha_k\tau_k^2 = \Theta\pr{k^2}$. 
Therefore, we obtain the convergence rates \eqref{eq:convergence_rate_generalized_p} by \Cref{thm:adanag_g}. 
Also, from \Cref{thm:adanag_g} and \eqref{eq:adanag_g_12_alpha}, we have $s_k \alpha_k \geq \frac{\rbound}{\upperboundL}$ and $\alpha_k \leq \frac{1}{2}$, which together imply $s_k \geq \frac{2\rbound}{\upperboundL}$. 
We obtain $r=\frac{27}{2 (p+3) \left(2 p^2+8 p+17\right)}$ by substituting $k=0$ to \eqref{eq:adanag_g_12_r}. 

Next, we check the assumptions of \Cref{lemma:asymptotic_lower_bound_step_size}. Since $\bar{\alpha}:=\lim_{k\to\infty}\alpha_k = \frac{1}{2}$ and $\lim_{k\to\infty} \frac{\tau_{k+1}}{\tau_k} = 1$, we obtain:
\begin{equation*}
    \bar{r} 
     = \lim_{k\to\infty} \pr{ \frac{A_{k}}{B_{k}} + \frac{B_{k+1} + \alpha_{k+1}^2\tau_{k+1}^2}{A_{k}} }^{-1}
     = \pr{ \frac{1}{\bar{\alpha}} + \bar{\alpha} + \bar{\alpha} }^{-1}
     = \pr{ 2 + \frac{1}{2} + \frac{1}{2} }^{-1}
     = \frac{1}{3}.
\end{equation*}
The condition $\prod_{k=1}^{\infty} \frac{\tau_k^2}{\tau_{k+1}(\tau_{k+1}-1)} =\infty$ follows from \eqref{eq:p>2_growth_infinity}. 
Finally, since $p>2$, there exists $N>0$ such that $(N+3)(p-2)>1$. 
For such $N$, we can verify that
\begin{equation*}
    \frac{A_{k-1}+\alpha_{k}\tau_{k}}{A_{k}} 
    = \frac{(k+3) (k+p+3)}{(k+4)^2}
    = 1 + \frac{(k+3) (p-2)-1}{(k+4)^2}
    \ge 1, \qquad \forall k\ge N.
\end{equation*}
We conclude $\liminf_{k\to\infty} s_k \ge \frac{1}{3} \frac{1}{\upperboundL}$ by \Cref{lemma:asymptotic_lower_bound_step_size}. 
\end{proof}

\begin{proof}   [Proof of Corollary~\ref{cor:adanag_12}]
    The proof can be obtained by substituting \( p=12 \) into Proposition~\ref{prop:ananaga_g_p}. 
\end{proof}

\subsection{Proof of Corollary~\ref{cor:adanag_half}}    \label{appendix:adanag_half_proof}

We first show that \eqref{ineq:adanag_AB_condition_for_step_size_lowerbound} holds.  
Name $r_k^L = \Big( \frac{A_{k}}{B_{k}} + \frac{B_{k+1} + \alpha_{k+1}^2\tau_{k+1}^2}{A_{k}} \Big)^{-1}$.  
Using elementary calculations, we can verify that:
\begin{equation*}
    \begin{aligned}
        \alpha_k &= \frac{1}{2}  \in (0,1] \\
        \frac{\tau_{k}^2}{\tau_{k+1}(\tau_{k+1}-1)} 
        &= \frac{2 (k+3)}{2 (k+4) -\sqrt{k+4}}
        = 1 + \frac{\sqrt{k+4}-2}{2 (k+4) -\sqrt{k+4}} 
        \ge 1 \\
        \alpha_{k+1} r_k^L   
        &\ge \alpha_{1} r_0^L \approx 0.102 > \frac{1}{10} =: \rbound.
    \end{aligned}
\end{equation*}
Therefore, we obtain $s_k\alpha_k \ge \frac{\rbound}{\upperboundL} \ge \frac{1}{10} \frac{1}{\upperboundL}$ from \Cref{thm:adanag_g}, and we conclude that $s_k \ge \frac{1}{5} \frac{1}{\upperboundL}$. 
Also, from \Cref{thm:adanag_g} and the fact $\alpha_k\tau_k^2 = \Theta\pr{k}$,  we obtain the convergence results of Corollary~\ref{cor:adanag_half}.

Next, we verify the assumptions of \Cref{lemma:asymptotic_lower_bound_step_size}. Since $\alpha_k = \frac{1}{2}$ and $\lim_{k\to\infty} \frac{\tau_{k+1}}{\tau_k} = 1$, we have:
\begin{equation*}
    \bar{r} 
     = \lim_{k\to\infty} \pr{ \frac{A_{k}}{B_{k}} + \frac{B_{k+1} + \alpha_{k+1}^2\tau_{k+1}^2}{A_{k}} }^{-1}
     = \pr{ 2 + \frac{1}{2} + \frac{1}{2} }^{-1}
     = \frac{1}{3}.
\end{equation*}
Also, since $\alpha_k = \alpha_{k+1}$, from the second inequality of \eqref{ineq:adanag_AB_condition_for_step_size_lowerbound} we obtain:
\begin{equation*}
    \frac{A_{k-1}+\alpha_{k}\tau_{k}}{A_{k}} 
    = \frac{\alpha_k}{\alpha_{k+1}} \frac{\tau_k^2}{\tau_{k+1}(\tau_{k+1}-1)} 
    = \frac{\tau_k^2}{\tau_{k+1}(\tau_{k+1}-1)} 
    \ge 1, \qquad \forall k\ge0.
\end{equation*}
Finally, from $\frac{\sqrt{k+4}-2}{2 (k+4) -\sqrt{k+4}}\ge0$ we have:
\begin{equation*}
    \prod_{k=1}^{\infty} \frac{\tau_k^2}{\tau_{k+1}(\tau_{k+1}-1)} 
    = \prod_{k=1}^{\infty} \pr{ 1 + \frac{\sqrt{k+4}-2}{2 (k+4) -\sqrt{k+4}} } 
    \ge  1 + \sum_{k=1}^{\infty} \frac{\sqrt{k+4}-2}{2 (k+4) -\sqrt{k+4}}
    = \infty.
\end{equation*} 
This completes the proof.
\qed

\end{document}